\documentclass[reqno, a4paper]{amsart}
\usepackage{a4wide}
\usepackage{graphicx}
\usepackage{bm}
\usepackage{stmaryrd}
\SetSymbolFont{stmry}{bold}{U}{stmry}{m}{n}
\usepackage{graphicx,color,overpic,rotating}
\usepackage[pdftex,dvipsnames]{xcolor}
\usepackage{mathrsfs}
\usepackage[labelformat=simple]{subcaption}
\usepackage[final,kerning=true,spacing=true,factor=1100,stretch=10,shrink=10]{microtype}

\usepackage{enumitem}

\usepackage{vem_macros}

 \newtheorem{theorem}{Theorem}
 \newtheorem{lemma}[theorem]{Lemma}
 \newtheorem{assumption}[theorem]{Assumption}
 \newtheorem{remark}[theorem]{Remark}
 \newtheorem{definition}[theorem]{Definition}

\numberwithin{equation}{section}
\numberwithin{theorem}{section}

\newcommand{\citer}[1]{~\cite{#1}}
\newcommand{\citers}[1]{~\cite{#1}}

\newcounter{saveenum}
\newenvironment{enumerate-assumption}
    {\begin{enumerate}[label=\textbf{A\arabic*}, leftmargin=0.75cm]
    \setcounter{enumi}{\value{saveenum}}
    }
    {
    \setcounter{saveenum}{\value{enumi}}
    \end{enumerate}
    }

\title[Adaptive non-hierarchical Galerkin methods for parabolic problems]{Adaptive non-hierarchical Galerkin methods\\ for parabolic problems with application to\\ moving mesh and virtual element methods}

 \author[A. Cangiani]{Andrea Cangiani}
 \address[A. Cangiani]{School of Mathematical Sciences, University of Nottingham, University Park, Nottingham, NG7 2RD, UK}
 \email{Andrea.Cangiani@nottingham.ac.uk}
 
 \author[E. H. Georgoulis]{Emmanuil H. Georgoulis}
 \address[E. H. Georgoulis]{School of Mathematics and Actuarial Science, University of Leicester, Leicester, LE1 7RH, UK \emph{and} Department of Mathematics, School of Applied MAthematical and Physical Sciences, National Technical University of Athens, Zografou 15780, Greece \emph{and} IACM-FORTH, Crete Greece}
 \email{Emmanuil.Georgoulis@le.ac.uk}
 
 \author[O. J. Sutton]{Oliver J. Sutton}
 \address[O. J. Sutton]{School of Mathematical Sciences, University of Nottingham, University Park, Nottingham, NG7 2RD, UK}
 \email{Oliver.Sutton@nottingham.ac.uk}
\date{3$^{\text{rd}}$ April 2020}

\begin{document}

\begin{abstract}
	We present \emph{a posteriori} error estimates for inconsistent and non-hierarchical Galerkin methods for linear parabolic problems, allowing them to be used in conjunction with very general mesh modification for the first time.
	We treat schemes which are non-hierarchical in the sense that the spatial Galerkin spaces between time-steps may be completely unrelated from one another.
	The practical interest of this setting is demonstrated by applying our results to finite element methods on moving meshes and using the estimators to drive an adaptive algorithm based on a virtual element method on a mesh of arbitrary polygons.
	The a posteriori error estimates, for the error measured in the $L^2(H^1)$ and $L^{\infty}(L^2)$ norms, are derived using the elliptic reconstruction technique in an abstract framework designed to precisely encapsulate our notion of inconsistency and non-hierarchicality and requiring no particular compatibility between the computational meshes used on consecutive time-steps, thereby significantly relaxing this basic assumption underlying previous estimates.
\end{abstract}

\maketitle

\section{Introduction}

Computable error estimates 
are used within simulations of natural and physical phenomena to ensure that accurate and reliable results are produced as efficiently as possible.
For those governed by systems of partial differential equations (PDEs), such \emph{a posteriori} computable error estimates are
often employed to drive adaptive algorithms, in which key components of the numerical scheme such as the computational mesh are automatically modified to focus computational effort in specific regions where higher resolution is required.
Although estimates such as these have been widely studied, many open questions remain.
In particular, although our understanding of error estimation for elliptic problems is by now rather mature (see
\citers{Ainsworth:2000cv,Verfurth:uz,Brenner:2008tq} for instance), the literature on error estimation for parabolic or hyperbolic systems is substantially less complete.

Optimal order a posteriori error estimates for linear parabolic problems in the $L^2(H^1)$ 
norm may be proven using direct energy arguments~\cite{Picasso:1998iw,Chen:2004ck}. 
Although the same arguments provide an estimate of the (higher order) error in the $L^{\infty}(L^2)$ norm, the resulting estimators are in fact typically of suboptimal order.
The first a posteriori error estimates in the $L^{\infty}(L^2)$ norm which were numerically demonstrated to be of optimal order were derived using duality techniques by Eriksson and Johnson~\cite{ERIKSSON:1991be,ERIKSSON:1995in}.
The alternative \emph{elliptic reconstruction} technique, introduced by Makridakis and Nochetto~\cite{Makridakis:2003ws}, allows a posteriori error estimates to be derived for the 
$L^{\infty}(L^2)$ norm via energy arguments by introducing elliptic reconstructions of the discrete solution.
The reconstruction splits the error into an elliptic component, which is estimated using existing a posteriori error estimates derived for an associated elliptic problem, and a parabolic component which satisfies a differential equation with data which may be numerically verified to be controlable at optimal order; see~\citer{Makridakis:2007ef} for an overview.

However, existing error estimates for parabolic problems in the $L^{\infty}(L^2)$ norm, and seemingly all such estimates of the error measured in norms weaker than the $L^2(H^1)$ energy norm, crucially rely, to the best of our knowledge, on the assumption that the discrete function spaces are \emph{hierarchical}.
By this, we refer to the case when the intersection of the finite element spaces used on consecutive time steps is itself a finite element space offering similar approximation properties.
Here we fill this gap by deriving error estimates for a model parabolic reaction-diffusion problem, in the $L^2(H^1)$ norm and the $L^{\infty}(L^2)$ norm, which do not place this requirement on the spaces.
We note that unsuitable non-hierarchical mesh modification can lead to divergent numerical methods in the context of evolution PDEs~\cite{Dupont:1982}, and the effects of non-hierarchicality in subsequent spatial finite element spaces for an evolution PDE can therefore introduce new challenges and behaviours. 
The a posteriori error bounds presented in this work could eventually be used in understanding such phenomena further and possibly aid the user in avoiding such scenarios in practical simulations.

The apparently innocuous assumption of hierarchicality is particularly restrictive in practice and unrepresentative of the general case, as exemplified by the following five scenarios in which non-hierarchicality naturally appears:
\begin{enumerate}[leftmargin=0.65cm]
	\item \textbf{Non-hierarchical refinement or coarsening.}
				A bilinear polynomial space on a rectangular element has the basis $\{1, x, y, xy\}$, yet if this element is refined into two triangular elements then the discrete space on each may be spanned only by the basis $\{1, x, y\}$.
				The function $xy$ cannot then be represented on the refined element and the spaces are therefore not hierarchical.
				Refining a mesh of squares into a mesh of triangles in the presence of homogeneous Dirichlet boundary conditions can mean the intersection of the two global finite element spaces is just the zero function.
	\item \textbf{Moving meshes.}
				If a mesh node is moved, then piecewise polynomial functions with respect to the original mesh cannot in general be represented on the modified mesh.
				Such a situation arises, for instance, in classical moving mesh and $r$-adaptive methods~\cite{Budd:2009hxb}, and in fluid-structure interaction problems.
	\item \textbf{Non-polynomial discrete function spaces.}
				Common non-polynomial discrete function spaces are naturally non-hierarchical under refinement.
				For instance, on meshes with polygonal elements, function spaces are typically directly tailored to the physical geometry of the elements, often in the form of rational functions or solutions to local boundary value problems and are therefore not hierarchical.
	\item \textbf{Boundary conditions.}
				If a non-polynomial essential boundary condition is incorporated into the space, hierarchicality is lost because the boundary traces of the discrete functions change when an element adjacent to the boundary is refined.
	\item \textbf{Domain approximation.}
				Hierarchicality is automatically lost if the mesh only approximates the problem domain or interior interfaces, so that the boundaries of the mesh change with refinement~\cite{Dorfler:1998,Ainsworth:2017,Cangiani:2018,Cangiani:2020}.
\end{enumerate}

The results we present here tackle challenges (1), (2) and (3) above, with a particular focus on treating schemes incorporating very general forms of mesh modification.
We demonstrate this with two examples: a conforming finite element method built on a moving mesh  (Section~\ref{sec:fem}), and 
a virtual element method (Section~\ref{sec:vem}). 
In the latter example, we also demonstrate the effectivity of the error estimators to drive a mesh adaptive algorithm exploiting meshes consisting of arbitrary polygonal elements.
Despite the fundamental appeal of using polygonal meshes in adaptive algorithms for time dependent problems, due to their natural ability to handle coarsening operations by simply merging arbitrary patches of elements, there does not appear to be any existing literature in this area, aside from the doctoral thesis of~\citer{Sutton:2017vc}.

We remove the assumption of hierarchical spaces in two stages,
producing two distinct estimates for the error component measuring the modification of the discrete spaces between time-steps, given in Lemmas~\ref{lem:residualEstimates:twoMesh:local} and~\ref{lem:residualEstimates:twoMesh:global} respectively.
Firstly, we suppose that the \emph{meshes} are still hierarchical, in the sense that one is constructed from the other by coarsening or refining a small number of elements, even though the function spaces themselves are not. This setting  is particularly applicable to scenarios (1), (3), and (4) above.
The form of this estimate mimics that of previous analogous estimates in the hierarchical setting~\cite{Lakkis:2006jk}, but with two extra terms which achieve a degree of `smallness' from the fact that they are only active on those few elements which are modified.

Secondly, we consider the case when the meshes may be completely different between time-steps, thereby incorporating moving mesh schemes (as in scenario 2) or the complete re-meshing of the domain.
The result hinges on the introduction of an \emph{elliptic transfer operator} (Definition~\ref{def:ellipticTransfer}) which provides a natural representation on one mesh of a discrete function defined on another, with respect to the PDE being studied.
The key role played by the elliptic transfer operator in the analysis 
is to enable a `discrete integration by parts' to be performed, ultimately replacing a term in the estimate which otherwise scales sub-optimally with an optimal term.

Although we study this operator in the context of backward Euler time-stepping, its properties mean that it may be expected to be of more general interest.
For example, it has previously been shown by B\"ansch et al.~\cite{Bansch:2012gu,Bansch:2013cu} that Crank-Nicolson time-stepping schemes can become unstable under mesh refinement when the previous solution is projected onto the new mesh.
Instead, in~\citers{Bansch:2012gu,Bansch:2013cu} they produce a stable scheme by also transferring the discrete Laplacian of the solution to the new mesh.
This is unnecessary for the elliptic transfer operator, at the expense of solving an additional elliptic problem for the transferred solution, since the discrete Laplacian of the transferred function is simply the $L^2(\domain)$-orthogonal projection onto the new mesh of the discrete Laplacian of the original function.

We first present the estimates in Section~\ref{sec:abstractFramework}, in the abstract framework of a general inconsistent non-hierarchical Galerkin method satisfying certain approximation properties.
We then show, in Section~\ref{sec:fem}, how this translates into the simpler context of a conforming finite element scheme constructed on a moving (time dependent) mesh, with numerical examples demonstrating the behaviour of the error estimate.

As a second example, in Section~\ref{sec:vem} we take a detailed look at how our results apply to a virtual element discretisation, incorporating adaptive meshes composed of general polygonal elements. 
The virtual element method (VEM), introduced in~\citer{VEIGA:2013wi}, is a generalisation of the finite element method to meshes containing general polygonal or polyhedral elements.
The application of virtual element methods to time-dependent problems is still in its infancy, with only schemes and convergence results presented for a model heat equation~\cite{Vacca:2015jh} and a Cahn-Hilliard problem~\cite{Antonietti:2016kr}.
Instead, virtual element methods for elliptic problems are already well developed; in particular, there is a growing literature on a posteriori error estimates and
adaptivity~\cite{BeiraodaVeiga:2015ku,Cangiani:2016ug,Berrone_apost,Mora:2016un,Weisser19,Antonietti19,daVeiga2019}, which may be utilised through the elliptic reconstruction framework as described above.
Similarly, adaptive algorithms incorporating agglomeration techniques have been applied by the discontinuous Galerkin community~\cite{Bassi:2012bt,Collis:2016gp} to efficiently discretise stationary problems on complicated domains from an initial fine mesh.

Here, for the first time, we exploit polygonal meshes for the adaptive solution of time-dependent problems by applying our abstract results to an adaptive virtual element method incorporating general mesh coarsening and refinement.
The $L^{\infty}(L^2)$ and $L^2(H^1)$ error estimates we present both appear to be novel. 
We  examine their practical behaviour through a series  of fixed mesh convergence benchmarks and adaptive tests, confirming that they are effective even in challenging adaptive situations.
Moreover, developing the mesh adaptive scheme itself requires the introduction of various new auxiliary components which may also be of independent interest in other contexts, such as operators to transfer discrete solutions between meshes which remain computable and accurate, even when the discrete basis functions themselves are assumed to be unknown.

We conclude with a brief discussion of our results in Section~\ref{sec:conclusion}.

\section{Model problem and notation}\label{sec:modelandnotation}
For $\omega \subset \Re^{m}$, with $m \in \mathbb{N}$, and functions $v,w \in L^2(\omega)$, we denote the $L^2(\omega)$ inner product by $(v,w)_{\omega} = \int_{\omega} vw \d x$.
We further use $\norm{\cdot}_{W^{k,p}(\omega)}$ and $\abs{\cdot}_{W^{k,p}(\omega)}$ to denote the standard norm and seminorm on the Sobolev space $W^{k,p}(\omega)$ for $k \geq 0$ and $p \in [1,\infty]$ (for further details see~\citer{Adams:2003wi}, for example).
In the case of $p=2$, we shall denote the $L^2(\omega)$ norm by $\norm{\cdot}_{\omega}$ and the $H^k(\omega)$ norm and seminorm by $\norm{\cdot}_{\omega,k}$ and $\abs{\cdot}_{\omega,k}$, respectively.
If $\omega = \domain$, the physical domain, then we shall omit the subscripts $\omega$ above.

Let $\finaltime > 0$ and let $\domain \subset \Re^{\spacedim}$, with $\spacedim = 2,3$, be a convex polytope.
We focus on the model parabolic problem: find $u : \domain \times [0,\finaltime] \to \Re$ satisfying
\begin{align}\label{eq:pde}
 \begin{split}
 u_t(x,t) - \diffop u(x,t) &= \force(x,t)\phantom{u_0(x)0} \text{ for } (x,t) \in \domain \times (0,\finaltime],
 \\
 u(x,0) &= u_0(x)\phantom{\force(x,t)0} \text{ for } x \in \domain,
 \\
 u(x,t) &= 0\phantom{u_0(x)\force(x,t)} \text{ for } (x,t) \in \boundary \times (0,\finaltime],
 \end{split}
\end{align}
with $\diffop$ denoting the second order linear elliptic reaction-diffusion operator
\begin{align*}
 \diffop v = \nabla \cdot (\diff \nabla v) - \reac v,
\end{align*}
where $\reac \in L^2(\domain)$ is such that there exists a constant $\reacLower \in \Re$ with $\reac(x) \geq \reacLower \geq 0$ for almost every $x \in \domain$.
We suppose that $\diff : \domain \to \Re^{\spacedim \times \spacedim}$ is symmetric and positive definite,
i.e. there exist constants $\ellipLower, \ellipUpper > 0$ such that
	$\ellipLower \abs{\vec{v}}^2 \leq \vec{v}^{\top} \diff(x) \vec{v} \leq \ellipUpper \abs{\vec{v}}^2$
for all $\vec{v} \in \Re^{\spacedim}$ and almost every $x \in \domain$, where $\abs{\cdot}$ denotes the Euclidean norm on $\Re^{\spacedim}$.

\begin{remark}
The results we present can be extended to non-convex domains via the careful application of weighted estimates; see, for example,~\citers{Liao:2003ke,Wihler:2007tv}. 
Furthermore, the non-hierarchicality introduced by removing the assumption  on $\Omega$ being a polytope would naturally fit within our framework.
Finally, more general boundary conditions can also be treated.
We do not pursue these here to avoid introducing additional (addressable) technicalities.
\end{remark}

Let $\A : H^1_0(\domain) \times H^1_0(\domain) \to \Re$ denote the bilinear form
\begin{align*}
	\A(v,w) = (\diff \nabla v, \nabla w) + (\reac v,w),
\end{align*}
and let $\triplenorm{v}^2 = \A(v, v)$ denote the norm induced on $H^1_0(\domain)$.
We shall also use the notation $\A_{\omega}$ to represent the bilinear form $\A$ with its component integrals taken over the set $\omega$.
We observe that $\A$ is continuous in $\triplenorm{\cdot}$, and this norm is equivalent to the $H^1(\domain)$ seminorm, i.e. there exists a constant $\Cequiv > 0$ such that
\begin{align}\label{eq:normEquivalence}
	\Cequiv^{-1} \abs{\testfn}_1 \leq \triplenorm{\testfn} \leq \Cequiv \abs{\testfn}_1,
\end{align}
for all $\testfn \in H^1_0(\domain)$.
Consequently, the Poincar\'e-Friedrichs-type inequality
\begin{align}\label{eq:PoincareFriedrichs}
	\norm{\testfn} \leq \Cpf \triplenorm{\testfn},
\end{align}
holds for any $\testfn \in H^1_0(\domain)$, with constant $\Cpf > 0$ depending on $\Cequiv$ and $\domain$.

The problem~\eqref{eq:pde} can be posed in the weak form: find $u \in L^{2}(0,\finaltime; H^1_0(\domain))$ with $u_t \in L^{2}(0,\finaltime; H^{-1}(\domain))$ such that
\begin{align}
	(u_t(t), \testfn) + \A(u(t), \testfn) = (\force(t), \testfn) \text{ for all } \testfn \in H^1_0(\domain) \text{ and a.e. } t \in [0,\finaltime].
	\label{eq:continuousProblem}
\end{align}
Standard arguments ensure that this problem possesses a unique solution~\cite{Evans:2010ec}.

\section{Abstract error estimates}\label{sec:abstractFramework}

We develop a posteriori error estimates in the abstract framework of an inconsistent Galerkin method built around discrete function spaces which may not be hierarchical, with no compatibility required between the spaces used on different time-steps.
In particular,
given a partition $\{t^{\curr}\}_{\curr = 0}^{N}$ of the time domain $[0,\finaltime]$, with $\timestep^{\curr} = t^{\curr} - t^{\prev} > 0$ for $\curr \in \{ 1,\dots,N \}$, we suppose that the scheme is formed of the following components.

\begin{assumption}[Components of the discrete framework]\label{assump:abstractFramework}
	For each $\curr \in \{0,\dots,N\}$ we assume that there exists
	\begin{enumerate-assumption}
		\item \label{item:abstractFramework:mesh}
			A \emph{mesh} $\mesh[\curr]$, dividing $\domain$ into a finite number of non-overlapping polytopic elements $\E$, such that the cardinality of the set $\sides[\E]$, denoting the set of sides of $\E$ (co-dimension one planar facets; edges when $\spacedim = 2$, faces when $\spacedim = 3$), is uniformly bounded.
			Further, there exists a constant $\meshReg > 0$ which is uniformly bounded with respect to mesh modification, satisfying $h_{\side} \geq \meshReg h_{\E}$ for each side $\side \in \sides[\E]$, where $h_{\omega}$ denotes the diameter of the set $\omega \subset \Re^{\spacedim}$.
		\item \label{item:abstractFramework:fespace}
			A \emph{finite-dimensional discrete function space} $\Vh[\E]$ for each $\E \in \mesh[\curr]$, which may be combined to build the conforming global discrete function space
			\begin{align}\label{eq:globalSpaceDefinition}
				\Vh[\curr] := \{ \trialfn \in H^1_0(\domain) : \trialfn|_{\E} \in \Vh[\E] \text{ for each } \E \in \mesh[\curr] \}.
			\end{align}
		\item \label{item:abstractFramework:bilinearForms}
			A pair of \emph{local discrete bilinear forms} $\m[\curr]_{\E} : \Vh[\E] \times \Vh[\E] \to \Re$ and $\A[\curr]_{\E} : \Vh[\E] \times \Vh[\E] \to \Re$ for each $\E \in \mesh[\curr]$,
			approximating the $L^2(\E)$ inner product $(\cdot, \cdot)_{\E}$ and the bilinear form $\A_{\E}$, respectively.
			These are summed to form the global discrete bilinear forms $\mh[\curr], \Ah[\curr] : \Vh[\curr] \times \Vh[\curr] \to \Re$, namely, 
			\begin{align}\label{eq:abstractGlobalBilinearForms}
				\mh[\curr](\cdot, \cdot) := \sum_{\E \in \mesh[\curr]} \m[\curr]_{\E}(\cdot, \cdot)
				\quad\text{ and }\quad
				\Ah[\curr](\cdot, \cdot) := \sum_{\E \in \mesh[\curr]} \A[\curr]_{\E}(\cdot, \cdot).
			\end{align}
			which are assumed to be inner products on $\Vh[\curr]$.
		\item \label{item:abstractFramework:forcingTerm}
			An elementwise projection operator $\brokenProj[\curr] : L^2(\domain) \to L^2(\domain)$ providing an \emph{approximation of the forcing data} $\forceProj[\curr] = \brokenProj[\curr] \force[\curr]$, where $\force[\curr] \equiv \force(t^{\curr})$, for which there exists a constant $\Cdata > 0$ such that for any $\testfn \in H^1(\domain)$
			\begin{align}\label{eq:approximationProperties:dataApproximation}
				(\force[\curr] - \forceProj[\curr], \testfn) \leq \Cdata \norm{h_{\curr} (\force[\curr] - \forceProj[\curr])} \triplenorm{\testfn}.
			\end{align}
			We further introduce $\forceh[\curr] = \LTwoRecon[\curr] \forceProj[\curr] \in \Vh[\curr]$, where the projection operator $\LTwoRecon[\curr] : L^2(\domain) \to \Vh[\curr]$ satisfies
			\begin{align}\label{eq:forcehDefinition}
				\mh[\curr](\LTwoRecon[\curr] \forceProj[\curr], \testcurr) = (\forceProj[\curr], \testcurr) \quad \text{ for all } \testcurr \in \Vh[\curr].
			\end{align}
		\item \label{item:abstractFramework:meshTransfer}
			A \emph{transfer operator} $\transfer[\curr] : \Vh[\prev] \to \Vh[\curr]$ which may be practically computed.
	\end{enumerate-assumption}
\end{assumption}

The mesh skeleton, formed as the set of all element sides in the mesh $\mesh[\curr]$, will be denoted by $\sides[\curr] = \bigcup_{\E \in \mesh[\curr]} \sides[\E]$, and we introduce the mesh-size function $h_{\curr} : \domain \to \Re$ associated with $\mesh[\curr]$ such that
$h_{\curr}(x) = h_{\E}$ for $x \in \E \in \mesh[\curr]$ and $h_{\curr}(x) = h_{\side}$ for $x \in \side \in \sides[\curr]$.
We further introduce the skeleton norm $\norm{\cdot}_{\sides[\curr]}^2 = \sum_{\side \in \sides[\curr]} \norm{\cdot}_{\side}^2$.
For brevity, we describe the inconsistency of the bilinear forms as follows.

\begin{definition}[Representation of inconsistency]
\label{def:inconsistencyRepresentation}
	For $w^{\curr}, \testcurr \in \Vh[\curr]$, let
	\begin{align*}
		\inconsistencym[\curr](w^{\curr}, \testcurr) &= (w^{\curr}, \testcurr) - \mh[\curr](w^{\curr}, \testcurr)
		\quad \text{and} \quad
		\inconsistencya[\curr](w^{\curr}, \testcurr) = \A(w^{\curr}, \testcurr) - \Ah[\curr](w^{\curr}, \testcurr).
	\end{align*}
	Then, we say that a scheme is inconsistent if there exists an $\curr \in \mathbb{N}$ such that $\inconsistencym[\curr](w^{\curr}, \testcurr) \neq 0$ or $\inconsistencya[\curr](w^{\curr}, \testcurr) \neq 0$.
\end{definition}

\begin{remark}[Approximations of the forcing data]
	We introduce $\forceProj[\curr]$ and $\forceh[\curr]$ separately above in order to separate the discretisation of the forcing data from the projection of it into the discrete space which naturally arises in the analysis.
	For a finite element scheme, $\brokenProj[\curr]$ could be taken as the identity operator, a Lagrangian interpolation operator, or a local projection into a finite element space, for example.
	Similarly, if the bilinear forms are consistent, $\LTwoRecon[\curr]$ is simply the $L^2(\domain)$-orthogonal projector onto $\Vh[\curr]$.
	Beyond the potential for $\forceProj[\curr]$ to be discontinuous, the crucial difference between the two is that $\forceh[\curr]$ is required to be zero on $\boundary$ since $\Vh[\curr] \subset H^1_0(\domain)$.
	Defining these separately ensures that the data is approximated at optimal order in the final estimate.
\end{remark}
\subsection{Numerical scheme}

The discrete scheme we pose for approximating solutions to the problem~\eqref{eq:pde} is: given $U^0 \in \Vh[0]$ approximating $u_0$, for each $\curr = 1,\ldots,N$ find $\ucurr \in \Vh[\curr]$ satisfying
\begin{align}\label{eq:fullyDiscreteScheme}
	\mh[\curr] \left( \dfrac{\ucurr - \transfer[\curr] \uprev}{\timestep^{\curr}}, \testcurr \right) + \Ah[\curr](\ucurr, \testcurr) = \mh[\curr](\forceh[\curr], \testcurr)
	\quad \text{ for all }
	\testcurr \in \Vh[\curr].
\end{align}
The fact that $\mh[\curr]$ is an inner product on $\Vh[\curr]$ implies the following equivalent pointwise form of the numerical scheme: given $U^0$, find $\ucurr \in \Vh[\curr]$ satisfying
\begin{align}\label{eq:discreteSchemePointwise}
	\timederivh[\curr]  \ucurr - \ellipReconRHSLag{\curr}{\curr}{\ucurr} = \forceProj[\curr]
	\quad
	\text{ for each } \curr \in \{ 1,\ldots,N \}.
\end{align}
Here, we have used the following discrete differential operators, noting that $\diffoph[\curr]$ is the analogue of the discrete Laplacian operator (cf. \citer{Thomee:2006th}) in the current setting.

\begin{definition}[Discrete differential operators]\label{def:discreteSpatialOperators}
	Let $\timederivh[\curr] : \Vh[\curr] \to \Vh[\curr]$ denote the \emph{discrete time derivative operator}, defined by
	\begin{align*}
		\timederivh[\curr] \ucurr := \frac{\ucurr - \transfer[\curr] \uprev}{\timestep^{\curr}} \in \Vh[\curr].
	\end{align*}
	We also define the \emph{discrete spatial operator} $\diffoph[\curr] : \Vh[\curr] \to \Vh[\curr]$ such that, for $\trialcurr \in \Vh[\curr]$
	\begin{align}\label{eq:discreteLaplacianDef}
		-\mh[\curr](\diffoph[\curr] \trialcurr, \testcurr) = \Ah[\curr](\trialcurr, \testcurr) \qquad \forall \testcurr \in \Vh[\curr],
	\end{align}
	and the \emph{data-dependent discrete spatial operator} $\ellipReconRHSLag{\curr}{m}{} : \Vh[\curr] \to L^2(\domain)$ given by
	\begin{align*}
		-\ellipReconRHSLag{\curr}{m}{\trialcurr} = -\diffoph[\curr] \trialcurr - (\LTwoRecon[\curr] - \identity) \brokenProj[\curr] \force[m],
	\end{align*}
	where $\identity$ denotes the identity operator.
	We emphasise that the superscript indicates the discrete space used in the construction, while the subscript represents the time step at which the data $\force[m]$ is evaluated.
\end{definition}

The definition of $\LTwoRecon[\curr]$ ensures that the discrete spatial operators are related by
\begin{align}\label{eq:modifiedDiffOpInconsistency}
	(\ellipReconRHSLag{\curr}{m}{\trialcurr}, \testcurr) = (\diffoph[\curr] \trialcurr, \testcurr) + \inconsistencym[\curr](\LTwoRecon[\curr] \brokenProj[\curr] \force[m], \testcurr).
\end{align}

\subsection{Solution reconstructions}
The forthcoming analysis revolves around the concept of an \emph{elliptic reconstruction operator}, introduced by Makridakis and Nochetto~\cite{Makridakis:2003ws}, which we define here as follows.

\begin{definition}[Elliptic reconstruction operator]
	\label{def:fullydiscrete:ellipticReconstruction}
	For each $\curr, m \in \{0,\dots,N\}$, we define the elliptic reconstruction operator $\ellipReconLag{\curr}{m}{} : \Vh[\curr] \to H^1_0(\domain)$, satisfying
	\begin{equation}\label{eq:ellipticReconstruction}
		\!\A(\ellipReconLag{\curr}{m}{\trialcurr}, \testfn)
		=
		- (\ellipReconRHSLag{\curr}{m}{\trialcurr}, \testfn)
		=
		-(\diffoph[\curr] \trialcurr + (\LTwoRecon[\curr] - \identity) \brokenProj[\curr] \force[m], \testfn)
     \text{ for all } \testfn \in H^1_0(\domain),
	\end{equation}
	with the same super/subscript convention as in Definition~\ref{def:discreteSpatialOperators}.
\end{definition}

The inconsistency of the elliptic reconstrucion in this framework is recorded in the following lemma.
When the discrete bilinear forms are consistent, this reduces to the conventional Galerkin orthogonality relationship $\A(\ellipReconLag{\curr}{\curr}{\trialcurr} - \trialcurr, \testcurr) = 0$.

\begin{lemma}[Elliptic reconstruction inconsistency]\label{lem:ellipreconInconsistency}
	For $\trialcurr \in \Vh[\curr]$, the elliptic reconstruction satisfies
	\begin{align*}
		\A(\trialcurr - \ellipReconLag{\curr}{m}{\trialcurr}, \testcurr)
		&=
		\inconsistencya[\curr](\trialcurr, \testcurr)
		+
		\inconsistencym[\curr](\diffoph[\curr] \trialcurr + \LTwoRecon[\curr] \brokenProj[\curr] \force[m], \testcurr)
	\quad \text{ for all } \testcurr \in \Vh[\curr].
	\end{align*}
\end{lemma}
\begin{proof}
	The result follows by substituting~\eqref{eq:modifiedDiffOpInconsistency} into the definition~\eqref{eq:ellipticReconstruction}, alongside the expansion
		$(\diffoph[\curr] \trialcurr, \testcurr) = -\A(\trialcurr, \testcurr) + \inconsistencya[\curr](\trialcurr, \testcurr) + \inconsistencym[\curr](\diffoph[\curr] \trialcurr, \testcurr)$.
\end{proof}

The \emph{time} and \emph{space-time reconstructions} of the discrete solutions are defined as
\begin{align}\label{eq:timeReconstructions}
	\timerecon(t) = \sum_{n = 0}^{N} \lcurr(t) \ucurr
	\quad \text{ and } \quad
	\spacetimerecon(t) = \sum_{n = 0}^{N} \lcurr(t) \ellipReconLag{\curr}{\curr}{\ucurr},
\end{align}
respectively where, for each $\curr \in \{ 0,\dots,N \}$, the continuous piecewise linear function $\lcurr : [0,T] \to [0,1]$, designed to satisfy $\ell^{i}(t^{j}) = \delta_{ij}$ where $\delta_{ij}$ is Kronecker's delta, is given by
\begin{align}
	\lcurr(t) =
	\begin{cases}
		\frac{t - t^{\prev}}{\timestep^{\curr}} & \text{ for } t \in [t^{\prev}, t^{\curr}],
		\\
		\frac{t^{\curr+1} - t}{\timestep^{\curr}} & \text{ for } t \in [t^{\curr}, t^{\curr+1}],
		\\
		\quad 0 & \text{ otherwise}.
	\end{cases}
\end{align}

These reconstructions split the error $\err(t) = u(t) - \timerecon(t)$ into a \emph{parabolic component} $\ctserr(t) = u(t) - \spacetimerecon(t)$ and an \emph{elliptic component} $\reconerr(t) = \spacetimerecon(t) - \timerecon(t)$.
The power of the elliptic reconstruction approach is that the elliptic component of the error may be estimated using the standard techniques for elliptic problems. This is because a discrete function $\trialcurr \in \Vh[\curr]$ may be viewed as the discrete approximate solution (in the framework of Assumption~\ref{assump:abstractFramework}) to the elliptic problem satisfied by $\ellipReconLag{\curr}{\curr}{\trialcurr}$.
Consequently, terms of the form $\norm{\ellipReconLag{\curr}{\curr}{\trialcurr} - \trialcurr}$ are simply the error of an elliptic problem.
This is particularly attractive in the present context of `exotic' spatial discretisations, and so for now we encapsulate this in the following assumption.
Concrete examples of such estimates are derived in Lemma~\ref{lem:residualEstimates:singleMesh}.

\begin{assumption}[Elliptic reconstruction error estimate]\label{assump:ellipticEstimates:singleMesh}
	We assume that
	\begin{enumerate-assumption}
		\item \label{item:ellipticEstimates:singleMesh}
			There exist \emph{elliptic reconstruction  estimators} $\ellipest{L^2}^{\curr}, \ellipest{H^1}^{\curr} : \Vh[\curr] \times L^2(\domain) \to \Re$ providing, for any $\trialcurr \in \Vh[\curr]$,  the estimates
			\begin{align*}
				\norm{\trialcurr - \ellipReconLag{\curr}{\curr}{\trialcurr}}
				&\leq
				\ellipest{L^2}^{\curr}(\trialcurr,
				\force[\curr])
				\quad\text{ and }\quad
				\triplenorm{\trialcurr - \ellipReconLag{\curr}{\curr}{\trialcurr}}
				\leq
				\ellipest{H^1}^{\curr}(\trialcurr,
				\force[\curr]).
			\end{align*}
	\end{enumerate-assumption}
\end{assumption}

\subsection{Error equation}
The parabolic component $\ctserr(t)$ of the error is estimated via an error equation.
Testing the pointwise form~\eqref{eq:discreteSchemePointwise} of the scheme with an arbitrary $\testfn \in H^1_0(\domain)$, using the definitions of the reconstructions
and recalling the variational problem~\eqref{eq:continuousProblem} gives
\begin{equation}\label{eq:L2H1errorequation}
	(\err_t, v) + \A(\ctserr, v) = (\force(t) - \forceProj[\curr], v) + \A(\ellipReconLag{\curr}{\curr}{\ucurr} - \spacetimerecon(t), v) + (\timederivh[\curr] \ucurr - \timerecon_t(t), v),
\end{equation}
for $t \in (t^{\prev}, t^{\curr}]$, which may be further expressed as
\begin{equation}\label{eq:LinfL2errorequation}
	(\ctserr_t, v) + \A(\ctserr, v) = (\force(t) - \forceProj[\curr], v) + \A(\ellipReconLag{\curr}{\curr}{\ucurr} - \spacetimerecon(t), v) + (\timederivh[\curr] \ucurr - \spacetimerecon_t(t), v).
\end{equation}
Following~\citer{Georgoulis:2011ita}, we note that the former form of the error equation is more convenient for deriving $L^2(0,t; H^1(\domain))$ norm estimates, while the latter can be used for estimates in the $L^{\infty}(0,t; L^2(\domain))$ norm.
The difference between their right-hand sides is in the final term:
for~\eqref{eq:L2H1errorequation} this is
\begin{align}\label{eq:meshTransferCharacterisation}
	(\timederivh[\curr] \ucurr - \timerecon_t(t), v) = \frac{1}{\timestep^{\curr}} ( \uprev - \transfer[\curr] \uprev, v ),
\end{align}
which naturally estimates the error from transferring solutions between meshes, while~\eqref{eq:LinfL2errorequation} contains
\begin{align*}
	(\timederivh[\curr]\ucurr - \spacetimerecon_t(t), v)
	=
	\frac{1}{\timestep^{\curr}}
	( (\ucurr - \ellipReconLag{\curr}{\curr}{\ucurr}) - (\transfer[\curr] \uprev - \ellipReconLag{\prev}{\prev}{\uprev}), v ),
\end{align*}
the nature of which is slightly more subtle, and the estimation of which presents the key difficulty of the $L^{\infty}(0,t; L^2(\domain))$ norm estimate in the non-hierarchical setting.
The derivation of such estimates is the focus of Section~\ref{sec:abstract:ellipticEstimators}, but for now we just assume the existence of the following estimator.

\begin{assumption}[Elliptic reconstruction time derivative estimator]\label{assump:ellipticEstimates:twoMesh}
	Let $\trialcurr \in \Vh[\curr]$ for $\curr \in \{0, \dots, N\}$, and let $\trialfn^{\ellipRecon}(t) = \sum_{\curr = 0}^{N} \ell^{\curr}(t) \ellipReconLag{\curr}{\curr}{\trialcurr}$.
	We assume that
	\begin{enumerate-assumption}
		\item \label{item:ellipticEstimates:twoMesh}
			A \emph{time derivative estimator} $\ellipest{L^2}^{\partial_t} : \Vh[\curr] \times \Vh[\prev] \times L^2(\domain) \times L^2(\domain) \to \Re$ exists with
			\begin{align*}
				\norm{\timederivh[\curr] \trialcurr - \trialfn^{\ellipRecon}_t(t)}
				&
				\leq
				\ellipest{L^2}^{\partial_t}(\trialcurr, \trialprev,
				\force[\curr], \force[\prev])
				\quad \text{ for }
				t \in (t^{\prev}, t^{\curr}].
			\end{align*}
	\end{enumerate-assumption}
\end{assumption}

We now focus on estimating the terms of the error equations above, which will utilise the following individual error estimators as shown in Lemma~\ref{lem:individualTerms}.

\begin{definition}[Estimator terms]\label{def:estimatorterms}
	For $\curr \in \{1,\dots,N\}$ and $t \in [t^{\prev}, t^{\curr}]$ define:
	\begin{itemize}
		\item the \emph{elliptic reconstruction error estimators}
			\begin{align*}
				\ellipreconest[L^2](t) = \sum_{\curr = 0}^{N} \lcurr(t) \ellipest{L^2}^{\curr}(\ucurr, \force[\curr]),
				\quad \text{ and } \quad
				\ellipreconest[H^1](t) = \sum_{\curr = 0}^{N} \lcurr(t) \ellipest{H^1}^{\curr}(\ucurr, \force[\curr]),
			\end{align*}
		\item the \emph{space error estimator}, where $\ellipest{L^2}^{\partial_t}$ satisfies Assumption~\ref{item:ellipticEstimates:twoMesh},
			\begin{align*}
				\spaceest(t)
				&=
				\ellipest{L^2}^{\partial_t}(\ucurr, \uprev, \force[\curr], \force[\prev]),
			\end{align*}
		\item the \emph{time error estimator}
			\begin{align*}
				\timeest(t) &= \norm{\ellipReconRHSLag{\curr}{\curr}{\ucurr} - \ellipReconRHSLag{\prev}{\prev}{\uprev}},
			\end{align*}
		\item the \emph{data approximation error estimators} for \emph{time} and \emph{space}
			\begin{align*}
				\dataest{T}(t) = \norm{\force(t) - \force[\curr]},
				\quad\text{ and }\quad
				\dataest{S}(t) = \Cdata \norm{h_{\curr} (\force[\curr] - \forceProj[\curr])},
				\quad \text{ respectively,}
			\end{align*}
		\item the \emph{mesh transfer error estimator}
			\begin{align*}
				\meshchangeest(t) = \frac{1}{\timestep^{\curr}} \norm{ \uprev - \transfer[\curr] \uprev }.
			\end{align*}
	\end{itemize}
\end{definition}

\begin{lemma}[Bounds for individual terms]\label{lem:individualTerms}
	Suppose that Assumptions~\ref{item:abstractFramework:mesh}--\ref{item:ellipticEstimates:twoMesh} are satisfied.
	Then, for each $n = 1,\dots,N$ and $t \in (t^{\prev}, t^{\curr}]$, the terms of the error equations may be estimated by separate contributions from the spatial error and temporal error
	\begin{align*}
		(\timederivh[\curr] \ucurr - \spacetimerecon_t(t), v )
		\leq
			\spaceest(t) \norm{v}
		\quad\text{ and }\quad
		\A(\ellipReconLag{\curr}{\curr}{\ucurr} - \spacetimerecon(t), v)
		\leq
			\timeest(t) \norm{v},
	\end{align*}
	respectively,
	and the data approximation error and mesh transfer error
	\begin{align*}
		(\force(t) - \forceProj[\curr], v)
		\leq
			\dataest{T}(t) \norm{v}
			+
			\dataest{S}(t) \triplenorm{v}
		\quad\text{ and }\quad
		(\timederivh[\curr] \ucurr - \timerecon_t(t), v)
		\leq
			\meshchangeest(t) \norm{v}.
	\end{align*}
\end{lemma}
\begin{proof}
	The spatial error estimate follows from Assumption~\ref{item:ellipticEstimates:twoMesh}.
	The temporal error estimate is derived by using the property that $\lprev(t) \leq 1$ and
	the expansion
	\begin{align*}
		\A(\ellipReconLag{\curr}{\curr}{\ucurr} - \spacetimerecon(t), v)
		&=
		\lprev \A(\ellipReconLag{\curr}{\curr}{\ucurr} - \ellipReconLag{\prev}{\prev}{\uprev}, v)
		=
		\lprev (\ellipReconRHSLag{\prev}{\prev}{\uprev} - \ellipReconRHSLag{\curr}{\curr}{\ucurr}, v).
	\end{align*}
	The data approximation estimate is shown by adding and subtracting $(\force[\curr], \testfn)$ and applying~\eqref{eq:approximationProperties:dataApproximation}.
	The bound for the mesh transfer error follows from~\eqref{eq:meshTransferCharacterisation}.
\end{proof}

\subsection{Parabolic a posteriori error estimates}

We recall some results from \citer{Sutton:2018ex} on \emph{exponentially weighted time accumulations} in Lemma~\ref{lem:timeAccumulations}.
The key appeal of these is to enable $L^{\infty}(L^2)$ error estimates in which the estimator terms accumulate through time in the minimal $L^p([0,t])$ norm for any $p \in [1,\infty]$.
The effectivities of such estimators can therefore become constant with $t$ (since the error and estimator may both accumulate in the $L^{\infty}([0,t])$ norm), rather than growing like $t$ or $t^{1/2}$ as they would if only $L^1$ or $L^2$ accumulations were used respectively (see~\citer{Sutton:2018ex} for details).
In the statement of the lemma, the term $F$ represents an estimator term accumulating with simulation time, and $\errorSurrogate$ represents the error to be estimated.
The terms estimated in each case are typical terms encountered in the $L^{\infty}(0,t; L^2(\domain))$ error analysis in Theorem~\ref{thm:reconstructionError}.

\begin{lemma}[Exponentially weighted time accumulations~{\cite[Lemma 4.9]{Sutton:2018ex}}]
	\label{lem:timeAccumulations}
	Let $p \in [1,\infty]$, $\lambda \in [0,1]$, and $r \in (0,T]$.
	We introduce the
	\emph{accumulation weighting coefficients}
	\begin{align*}
		c_{p,r} := \norm{\beta_r}_{L^q(0,r)}
		=
		\begin{cases}
			 \Big( \dfrac{1 - e^{-q \alpha_\lambda r}}{q \alpha_\lambda} \Big)^{1/q} &\text{ for } p \in (1,\infty],
			 \\
			 \qquad 1 &\text{ for } p = 1,
		\end{cases}
	\end{align*}
	where $q$ satisfies $\frac{1}{p} + \frac{1}{q} = 1$, $\beta_r(s) = e^{\alpha_\lambda(s-r)}$, and $\alpha_\lambda = \frac{2(1 - \lambda) }{(\Cequiv\Cpf)^2}$.

	Let $t > 0$ and suppose that $F \in L^{p^{\star}}(0,t)$ for some $p^{\star} \in [1,\infty]$, with $F(s) \geq 0$ for a.e. $s \in [0,t]$.
	Then, for $\errorSurrogate \in L^{\infty}(0,t; L^2(\domain))$, the estimate
	\begin{align}\label{eq:lplinfacc}
		\int_{0}^{t}
		&e^{\alpha_\lambda (s - t)}
		F(s)
		\norm{\errorSurrogate(s)}
		\d s
		\leq
		\Big(\accumulation{[1,\infty]}F\Big)
		\Big(\max_{s \in [0, t]} \norm{\errorSurrogate(s)}\Big),
	\end{align}
	holds, and $\errorSurrogate \in L^{2}(0,t; H^1(\domain))$ satisfies
	\begin{align}\label{eq:lpl2acc}
		\int_{0}^{t}
		&e^{\alpha_\lambda (s - t)}
		F(s)
		\triplenorm{\errorSurrogate(s)}
		\d s
		\leq
		\Big(\specialaccumulation{[2,\infty]}F\Big)
		\Big( \int_{0}^{t} e^{\alpha_\lambda (s - t)} \triplenorm{\errorSurrogate(s)}^2 \d s \Big)^{1/2}.
	\end{align}
\end{lemma}

We are now fully equipped to estimate the error in the $L^2(0,t; H^1(\domain))$ and $L^{\infty}(0,t; L^2(\domain))$ norms committed by the abstract non-hierarchical and inconsistent scheme~\eqref{eq:fullyDiscreteScheme}.

\begin{theorem}[Abstract a posteriori error estimates]\label{thm:reconstructionError}
	Let $u \in L^{\infty}(0,t; L^2(\domain)) \cap L^2(0,t; H^1(\domain))$ be the solution to~\eqref{eq:continuousProblem} 
	and let $U$ be the linear time reconstruction of the solution to the scheme~\eqref{eq:fullyDiscreteScheme} defined in~\eqref{eq:timeReconstructions}.
	Then, under the assumptions of Lemma~\ref{lem:individualTerms}, for a.e. $t \in (0,T]$ the error $\err(t) = u(t) - U(t)$ satisfies the $L^2(0,t; H^1(\domain))$ estimate
	\begin{align*}
		\Big( \int_{0}^{t} \triplenorm{\err(s)}^2 \d s \Big)^{1/2}
		&\leq
		C \Big(
			\norm{\err(0)}
			+
			\norm{\ellipreconest[H^1]}_{L^2(0,t)}
			+
			\norm{\dataest{S}}_{L^2(0,t)}
			+
			\min_{p \in \{1,2\}}
			\norm{\timeest}_{L^p(0,t)}
			\\&\qquad\qquad\qquad
			+
			\min_{p \in \{1,2\}}
			\norm{\meshchangeest}_{L^p(0,t)}
			+
			\min_{p \in \{1,2\}}
			\norm{\dataest{T}}_{L^p(0,t)}
		\Big),
	\end{align*}
	where the constant $C > 0$ depends only on $\Cpf$, and the $L^{\infty}(0,t; L^2(\domain))$ estimate
	\begin{align*}
	\max_{s \in [0,t]} \norm{\err(s)}
	&\leq
	C\Big(
		\norm{\err(0)}
		+
		\norm{\ellipreconest[L^2]}_{L^{\infty}(0,t)}
		+
\accumulation{[1,\infty]}
		\spaceest
		\\&\qquad\qquad
		+
		\accumulation{[1,\infty]}
		\timeest
		+
		\accumulation{[1,\infty]}
		\dataest{T}
		+
		\specialaccumulation{[2,\infty]} \dataest{S}
		\Big),
	\end{align*}
	where the constant $C > 0$ depends only on $\lambda$ from Lemma~\ref{lem:timeAccumulations}.
\end{theorem}
\begin{proof}
	We begin with the $L^2(0,t; H^1(\domain))$ error estimate.
	Selecting $\testfn = \err$ in~\eqref{eq:L2H1errorequation} and applying the bounds of Lemma~\ref{lem:individualTerms} to the terms on the right-hand side provides
	\begin{align*}
		\frac{1}{2} \frac{d}{dt} \norm{\err(t)}^2
		+
		\triplenorm{\ctserr(t)}^2
		\leq
		\triplenorm{\reconerr(t)}
		\triplenorm{\ctserr(t)}
		+
		\dataest{S}(t)
		\triplenorm{\err(t)}
		+
			\term{}(t)
		\norm{\err(t)},
	\end{align*}
	where
		$\term{} = \timeest + \dataest{T} + \meshchangeest$.
	Applying the triangle inequality and Young's inequality $ab \leq \frac{1}{2\delta} a^2 + \frac{\delta}{2} b^2$, valid for any $a,b \geq 0, \delta > 0$, we deduce that
	\begin{align*}
		\triplenorm{\reconerr(t)}
		\triplenorm{\ctserr(t)}
		+
		\dataest{S}(t)
		\triplenorm{\err(t)}
		\leq
		\frac{1 + 2\delta^2}{2\delta} \triplenorm{\reconerr(t)}^2
		+
		\frac{3\delta}{2} \triplenorm{\ctserr(t)}^2
		+
		\frac{1}{2\delta} \dataest{S}(t)^2.
	\end{align*}
	Therefore, integrating over $s \in [0,t]$ for some $t \in [0,T]$, we find that
	\begin{align}\label{eq:L2H1parabolic:intermediateBound}
		\norm{\err(t)}^2
		+
		\gamma \int_{0}^{t} \triplenorm{\ctserr(s)}^2 \d s
		&
		\leq
		\norm{\err(0)}^2
		+
		\int_{0}^{t}
			(
			\xi \triplenorm{\reconerr(s)}^2
			+
			\frac{1}{\delta} \dataest{S}(s)
		+
		2
			\term{}(s)
			\norm{\err(s)}
			)
		\d s,
	\end{align}
	where $\gamma = 2 - 3 \delta$ and $\xi =	\frac{1 + 2\delta^2}{\delta}$.
	H\"older's inequality and~\eqref{eq:PoincareFriedrichs} give
	\begin{align*}
		\int_{0}^{t}
			\term{}(s)
			\norm{\err(s)}
		\d s
		&\leq
		\min \Big\{
			\max_{s \in [0,t]} \norm{\err(s)}
			\int_{0}^{t}
				\term{}(s)
			\d s,
		\\&\qquad\qquad
			\Cpf
			\Big(\int_{0}^{t}
				\triplenorm{\err(s)}^2
			\d s
			\int_{0}^{t}
				\term{}^2(s)
			\d s\Big)^{1/2}
		\Big\},
	\end{align*}
	to which we further apply Young's inequality and the triangle inequality to find that
	\begin{align*}
		2
		\int_{0}^{t}
			\term{}(s)
			\norm{\err(s)}
		\d s
		&
		\leq
		\min \Big\{
			\delta
			\max_{s \in [0,t]} \norm{\err(s)}^2
			+
			\frac{1}{\delta}
			\Big(
				\int_{0}^{t}
					\term{}(s)
				\d s
			\Big)^2,
			\\&\qquad\qquad\qquad\qquad
			\int_{0}^{t}
			(
				\frac{\Cpf}{\delta}
				\term{}^2(s)
				+
				\delta
				\triplenorm{\err(s)}^2
			)
			\d s
		\Big\}.
	\end{align*}
	Substituting this into~\eqref{eq:L2H1parabolic:intermediateBound} produces an estimate with a right-hand side which is a non-decreasing function of $t$, and consequently we deduce the bound
	\begin{align*}
		&\max_{s \in [0,t]} \norm{\err(s)}^2
		+
		\gamma \int_{0}^{t} \triplenorm{\ctserr(s)}^2 \d s
		\leq
		\norm{\err(0)}^2
		+
		\int_{0}^{t}
		(
			\xi \triplenorm{\reconerr(s)}^2
			+
			\frac{1}{\delta} \dataest{S}(s)
		)
		\d s
		\\&\qquad
		+
		\min \Big\{
			\delta
			 \max_{s \in [0,t]} \norm{\err(s)}^2
			+
			\frac{1}{\delta}
			\Big(
				\int_{0}^{t}
					\term{}(s)
				\d s
			\Big)^2,
			\int_{0}^{t}
			(
				\delta
				\triplenorm{\err(s)}^2
				+
				\frac{\Cpf}{\delta}
				\term{}^2(s)
			)
			\d s
		\Big\}.
	\end{align*}
	Upon simplification, this implies that there exists a constant $C > 0$ such that
	\begin{align*}
		C \norm{\ctserr}_{L^2(H^1)}
		&
		\leq
		\norm{\err(0)}
		+
		\Big(
		\int_{0}^{t}
		(
			\triplenorm{\reconerr(s)}^2
			+
			\dataest{S}(s)
		)
		\d s
		\Big)^{1/2}
		+
		\min_{p\in\{1,2\}}
				\norm{\term{}}_{L^p(0,t)},
	\end{align*}
	and the $L^2(0,t; H^1(\domain))$ estimate follows by recalling the elliptic reconstruction error estimates for $\reconerr$ and
	noting that we could have treated the terms of $\term{}$ separately (we combined them here for brevity).

	We now turn to the $L^{\infty}(0,t; L^2(\domain))$ estimate.
	For this, we argue as in~\cite[Lemma 4.10]{Sutton:2018ex}, although we recall the main points of the argument here for clarity.
	Selecting $\testfn = \ctserr$ in the error equation~\eqref{eq:LinfL2errorequation}, applying the individual bounds of Lemma~\ref{lem:individualTerms}, and introducing $
		\termalt{}
		=
		\spaceest
		+
		\timeest
		+
		\meshchangeest
		+
		\dataest{T}$
	produces the estimate
	\begin{align*}
		\frac{1}{2} \frac{d}{dt} \norm{\ctserr(t)}^2
		+
		\triplenorm{\ctserr(t)}^2
		\leq
		\termalt{}(t)
		\norm{\ctserr(t)}
		+
		\dataest{S}(t) \triplenorm{\ctserr(t)}.
	\end{align*}
	The Poincar\'e-Friedrichs inequality~\eqref{eq:PoincareFriedrichs} implies that, for any $\lambda \in [0,1]$,
	\begin{align*}
		\frac{1}{2} \frac{d}{dt} \norm{\ctserr(t)}^2
		+
		\alpha
		\norm{\ctserr(t)}^2
		+
		\lambda
		\triplenorm{\ctserr(t)}^2
		\leq
		\frac{1}{2} \frac{d}{dt} \norm{\ctserr(t)}^2
		+
		\triplenorm{\ctserr(t)}^2,
	\end{align*}
	with $\alpha \equiv \alpha_\lambda = \frac{2(1 - \lambda) }{\Cpf^2}$ (from Lemma~\ref{lem:timeAccumulations}; we omit the subscript for brevity), and hence
	\begin{align*}
		\frac{1}{2}\frac{d}{dt} \Big( e^{\alpha t} \norm{\ctserr(t)}^2 \Big) + \lambda e^{\alpha t} \triplenorm{\ctserr(t)}^2 \leq
			e^{\alpha t}
			\termalt{}(t) \norm{\ctserr(t)}
			+
			e^{\alpha t}
			\dataest{S}(t)
			\triplenorm{\ctserr(t)}.
	\end{align*}
	Integrating over $s \in [0, t]$, multiplying both sides by $e^{-\alpha t}$, 
	and invoking Lemma~\ref{lem:timeAccumulations} with $\xi = \ctserr$
	we thus obtain
	\begin{align*}
		\frac{1}{2} \norm{\ctserr(t)}^2
		 +
		 \lambda
		 \int_{0}^{t}  e^{\alpha(s-t)} \triplenorm{\ctserr(s)}^2 \d s
		&
		\leq
		\frac{1}{2} e^{-\alpha t} \norm{\ctserr(0)}^2
		+
		\max_{s \in [0,t]} \norm{\ctserr(s)}
			\accumulation{[1,\infty]} \termalt{}
		\\&\qquad
			+
			\Big(\int_{0}^{t} e^{\alpha(s-t)} \triplenorm{\ctserr(s)}^2 \d s\Big)^{1/2}
			\specialaccumulation{[2,\infty]} \dataest{S}.
	\end{align*}
	We next
  apply Young's inequality with
	$\delta = 2\lambda$ to the final two factors,
	 obtaining
	\begin{align*}
		&\norm{\ctserr(t)}^2
		\leq
		\norm{\ctserr(0)}^2
		+
		2
		\max_{s \in [0,t]} \norm{\ctserr(s)}
			\accumulation{[1,\infty]} \termalt{}
			+
			\frac{2}{\lambda}
			\Big(
			\specialaccumulation{[2,\infty]} \dataest{S}\Big)^{2},
	\end{align*}
  since $e^{-\alpha t} \leq 1$.
	The right-hand side of this bound is a non-decreasing function of $t$, and thus, as before, we deduce that it provides a bound on $\max_{s \in [0,t]} \norm{\ctserr(s)}^2$.
	Applying Young's inequality again,
	we find that
	\begin{align*}
		&
		\max_{s \in [0,t]} \norm{\ctserr(s)}
		\leq
		\max\Big\{1, \sqrt{\frac{2}{\lambda}}\Big\}
		\Big(
		\norm{\ctserr(0)}
		+
		\accumulation{[1,\infty]} \termalt{}
		+
		\specialaccumulation{[2,\infty]} \dataest{S}
		\Big).
	\end{align*}
	The $L^{\infty}(0,t; L^2(\domain))$ error estimate follows from the elliptic reconstruction estimates and noting, as before, that we could have considered the terms of $\termalt{}$ separately.
	We note that the dependence of the weightings $c_{p,t}$ on $\lambda$ means that the optimal value of $\lambda$ is not clear, and depends on the time $t$.
\end{proof}

\subsection{Residual-type elliptic error estimates}\label{sec:abstract:ellipticEstimators}

We now derive residual-type elliptic estimators satsifying Assumptions~\ref{item:ellipticEstimates:singleMesh} (in Lemma~\ref{lem:residualEstimates:singleMesh})
and~\ref{item:ellipticEstimates:twoMesh} (in Lemmas~\ref{lem:residualEstimates:twoMesh:local}
and~\ref{lem:residualEstimates:twoMesh:global}).
The key developments here are the estimates satisfying Assumption~\ref{item:ellipticEstimates:twoMesh}. 
In particular,  Lemmas~\ref{lem:residualEstimates:twoMesh:local}
and~\ref{lem:residualEstimates:twoMesh:global} are best suited to the case of local or global mesh modification, respectively, although neither result requires any particular compatibility between the spaces.

For vector-valued quantities which may be discontinuous across the mesh skeleton, we
define the \emph{jump operator} $\jump{\cdot}$ across a mesh interface $\side \in \sides[\curr]$ as
\begin{align}
	\jump{\vec{v}}|_{\side} =
	\begin{cases}
		\vec{v}^+ \cdot \vec{n}_{\side}^{+} + \vec{v}^- \cdot \vec{n}_{\side}^{-}
		&\text{ if } \side \cap \boundary = \emptyset,
		\\
		0
		&\text{ otherwise},
	\end{cases}
\end{align}
with the following notation: if $\side \cap \boundary = \emptyset$, then there exist $\E[+], \E[-] \in \mesh[\curr]$ such that $\side \subset \partial\E[+]\cap\partial\E[-]$;
the trace of the function $\vec{v}$ on $\side$ from within $\E[\pm]$ is therefore denoted by $\vec{v}^{\pm}$, and $\vec{n}_\side^{\pm}$ denotes the unit outward normal on $\side$ with respect to $\E[\pm]$.
We may now introduce the following residual operators which form the basis of the error estimates in this section.

\begin{definition}[Residuals]
	Let $m, \curr \in \{ 0,\dots,N \}$. For $\trialcurr \in \Vh[\curr]$, let $\jumpresidual[\curr] : \Vh[\curr] \to L^2(\sides[\curr])$  denote the \emph{jump residual operator}
	\begin{align*}
		(\jumpresidual[\curr] \trialcurr)|_{\side}
		=
		\jump{\diff \nabla \trialcurr}_{\side},
	\end{align*}
	the definition of which is extended by zero to the whole of $\domain$.
	Further, let $\elresidualLag{\curr}{m} : \Vh[\curr] \to L^2(\domain)$ denote the \emph{element residual operator}, defined as
	\begin{align*}
		(\elresidualLag{\curr}{m}{\trialcurr} )|_{\E}
		=
		\diffop (\trialcurr |_{\E}) - (\ellipReconRHSLag{\curr}{m}{\trialcurr})|_{\E}
		=
		\diffop (\trialcurr |_{\E}) - (\diffoph[\curr] \trialcurr + (\LTwoRecon[\curr] - \identity) \brokenProj[\curr] \force[m])|_{\E}
	\end{align*}
	for each $\E \in \mesh[\curr]$, where $\ellipReconRHSLag{\curr}{m}{}$ is from Definition~\ref{def:fullydiscrete:ellipticReconstruction}.
	We emphasise that the superscript denotes the discrete space the operator acts on while the subscript indicates the time-step at which the PDE data is evaluated.
\end{definition}

The following approximation properties are required for the estimates.

\begin{assumption}[Approximation properties of the discrete framework]\label{assump:approximationProperties}
We suppose that the components of the discrete framework specified in Assumption~\ref{assump:abstractFramework} satisfy:
\begin{enumerate-assumption}
	\item \label{item:approximationProperties:interpolationEstimates}
		There exists $\Cinterp > 0$, depending only on the regularity of the mesh, such that:
		\begin{itemize}
			\item For any $\trialfn \in H^{1}(\domain)$, there exists a function $\trialcurr_{\clement} \in \Vh[\curr]$ satisfying
					\begin{equation}\label{eq:quasiInterpolation}
						\!\!\norm{\trialfn - \trialcurr_{\clement}}_{0,\E} + h_{\E} \abs{\trialfn - \trialcurr_{\clement}}_{1,\E} \leq \Cclem h_{\E} \abs{\trialfn}_{1,\patch{\E}}
						\text{ for all }
						\E \in \mesh[\curr],
					\end{equation}
					where $\patch{\E}$ denotes the usual finite element patch relative to $\E$ consisting of elements sharing a vertex with $\E$.
			\item For any $\trialfn \in H^{2}(\domain)$, there exists a function $\trialcurr_{\interp} \in \Vh[\curr]$ satisfying
					\begin{equation}\label{eq:interpolationEstimate}
						\!\!\norm{\trialfn - \trialcurr_{\interp}}_{0,\E} + h_{\E} \abs{\trialfn - \trialcurr_{\interp}}_{1,\E} \leq \Cinterp h_{\E}^{2} \abs{\trialfn}_{2, \E}
						\text{ for all }
						\E \in \mesh[\curr],
					\end{equation}
					constructed locally such that if $\mesh[m] \ni \E \in \mesh[\curr]$ then $\trialcurr_{\interp}|_{\E} = \trialfn^m_{\interp}|_{\E}$.
			\end{itemize}
		\item \label{item:approximationProperties:inconsistencyEstimate}
			There exist \emph{inconsistency estimators} $\projErrorLTwo[\E], \projErrorEnergy[\E] : \Vh[\E] \to \Re$ such that
			\begin{align*}
				\inconsistencym[\E](w^{\curr}, \testcurr) \leq (\projErrorLTwo[\E] w^{\curr}) (\projErrorLTwo[\E] \testcurr)
				\quad \text{ and } \quad
				\inconsistencya[\E](w^{\curr}, \testcurr) \leq (\projErrorEnergy[\E] w^{\curr}) (\projErrorEnergy[\E] \testcurr),
			\end{align*}
			for $\trialcurr, \testcurr \in \Vh[\curr]$. Moreover, if $\trialfn \in H^{2}(\domain)$ and $\trialfn_{\clement}, \trialfn_{\interp} \in \Vh[\curr]$ satisfy~\eqref{eq:quasiInterpolation} and~\eqref{eq:interpolationEstimate} respectively, then
			\begin{align*}
				\projErrorLTwo[\E] \trialfn_{\clement} + h_{\E} \projErrorEnergy[\E] \trialfn_{\clement} \leq \Cincons h_{\E} \abs{\trialfn}_{1, \patch{\E}}
				 \text{ and }
					\projErrorLTwo[\E] \trialfn_{\interp} + h_{\E} \projErrorEnergy[\E] \trialfn_{\interp} \leq \Cincons h_{\E}^2 \abs{\trialfn}_{2, \E},
			\end{align*}
			where $\Cincons>0$ only depends on $\Cinterp$, the problem data, and the regularity of $\mesh[\curr]$.
	\end{enumerate-assumption}
\end{assumption}

These inconsistency estimators provide the following estimates for the inconsistency of the global discrete bilinear forms, which we record here.
\begin{lemma}[Global inconsistency estimate]\label{lem:globalInconsistencyEstimate}
	Let $\trialcurr \in \Vh[\curr]$ and let $X$ denote either $L^2$ or $\A$.
	For $r \in \Re$, we denote weighted global inconsistency estimates by
	\begin{align*}
		\projError{X}{\curr} (h_{\curr}^r, \trialcurr) = \Big(\sum_{\E \in \mesh[\curr]} (h_{\E}^r \projError{X}{\E} \trialcurr)^2 \Big)^{1/2},
	\end{align*}
	abbreviated to $\projError{X}{\curr} (\trialcurr) $ when $r=0$.
	If $\testfn \in H^{2}(\domain)$ and $\testfn_{\clement}, \testfn_{\interp} \in \Vh[\curr]$ satisfy~\eqref{eq:quasiInterpolation} and~\eqref{eq:interpolationEstimate} respectively,
	then there exists a constant $\Cglob = \alpha \Cincons$, where $\alpha$ only depends on the regularity of $\mesh[\curr]$, such that
	\begin{align*}
		\inconsistencya[\curr](\trialcurr, \testfn_{\clement})
		\leq
		\Cglob
		\projErrorEnergy[\curr](\trialcurr)
		\abs{\testfn}_{1}
			&\quad\text{ and } \quad
		\inconsistencya[\curr](\trialcurr, \testfn_{\interp})
		\leq
		\Cglob
		\projErrorEnergy[\curr](h_{\curr}, \trialcurr)
		\abs{\testfn}_{2},
		\\
		\inconsistencym[\curr](\trialcurr, \testfn_{\clement})
		\leq
		\Cglob
		\projErrorLTwo[\curr](h_{\curr}, \trialcurr)
		\abs{\testfn}_{1}
			&\quad\text{ and } \quad
		\inconsistencym[\curr](\trialcurr, \testfn_{\interp})
		\leq
		\Cglob
		\projErrorLTwo[\curr](h_{\curr}^2, \trialcurr)
		\abs{\testfn}_{2}.
	\end{align*}
\end{lemma}

Equipped with these basic components, the following single mesh elliptic estimators may be derived using standard techniques developed for finite element methods (accounting for this inconsistent setting through Lemmas~\ref{lem:ellipreconInconsistency} and~\ref{lem:globalInconsistencyEstimate}).

\begin{lemma}[Single mesh residual-type error estimators]\label{lem:residualEstimates:singleMesh}
	There exists a positive constant $\Cellip$ depending only on the regularity of the domain and the mesh geometry so that Assumption~\ref{item:ellipticEstimates:singleMesh} is satisfied with
	\begin{align*}
		\ellipest{L^2}^{\curr}(\trialcurr,
		 \force[\curr])
		&=
		\Cellip
		\big(
			\norm{h_{\curr}^2 \elresidualLag{\curr}{\curr}{\trialcurr}}^2
			+
			\norm{h_{\curr}^{3/2} \jumpresidual[\curr] \trialcurr }_{\sides[\curr]}^2
			+
			\inconGroup[\curr](\trialcurr, \force[\curr])^2
		\big)^{1/2},
		\\
		\ellipest{H^1}^{\curr}(\trialcurr,
		\force[\curr])
		&=
		\Cellip
		\big(
			\norm{h_{\curr} \elresidualLag{\curr}{\curr}{\trialcurr}}^2
			+
			\norm{h_{\curr}^{1/2} \jumpresidual[\curr] \trialcurr }_{\sides[\curr]}^2
			+
			\inconGroupEnergy[\curr](\trialcurr, \force[\curr])^2
		\big)^{1/2}.
	\end{align*}
	where $\inconGroup[\curr]$ and $\inconGroupEnergy[\curr]$ respectively represent the inconsistency terms
	\begin{align*}
		\inconGroup[\curr](\trialcurr, \force[\curr])
		&=
		\big(
			\projErrorLTwo[\curr](h_{\curr}^2, \diffoph[\curr] \trialcurr + \forceh[\curr])^2
			+
			\projErrorEnergy[\curr](h_{\curr}, \trialcurr)^2
		\big)^{1/2},
		\\
		\inconGroupEnergy[\curr](\trialcurr, \force[\curr])
		&=
		\big(
			\projErrorLTwo[\curr](h_{\curr}, \diffoph[\curr] \trialcurr + \forceh[\curr])^2
			+
			\projErrorEnergy[\curr](\trialcurr)^2
		\big)^{1/2}.
	\end{align*}
\end{lemma}

\subsubsection{Time derivative estimator for local mesh modification}
For the time derivative estimator, we first present a result which is best suited for the case when each mesh in the sequence is constructed from its predecessor by modifying (e.g. coarsening or refining) a relatively small number of elements. In this situation, it is meaningful to exploit the equivalence of the interpolant satisfying~\eqref{eq:interpolationEstimate} in regions where there is no mesh change.

\begin{lemma}[Interpolation estimate for locally modified meshes]\label{lem:meshCompatibility}
	Let $n,m \in \{0,\dots,N\}$ and $F \in L^2(\domain)$.
	For $\testfn \in H^2(\domain)$, let $\testfn^n \in \Vh[n]$ and $\testfn^m \in \Vh[m]$ be the interpolants of $\testfn$
	satisfying~\eqref{eq:interpolationEstimate}.
	Then,
	\begin{align*}
		(F, \testfn^n - \testfn^m)
		\leq
		2\Cclem
		\norm{\hat{h}_{n,m}^2 F}_{\mesh[n] \setminus \mesh[m]}
		\abs{\testfn}_2,
		\quad\text{where} \quad
		\norm{F}_{\mesh[n] \setminus \mesh[m]}^2 = \sum_{\E[n] \in \mesh[n]\setminus\mesh[m]} \norm{F}_{\E[n]}^2,
	\end{align*}
	and $\hat{h}_{n,m}(x) := \max\{ h_n(x), h_m(x)\}$.
	Here, $\mesh[n] \setminus \mesh[m] = \{ \E\in\mesh[n] : \E \not \in \mesh[m] \}$.
\end{lemma}
\begin{proof}
	By Assumption~\ref{item:approximationProperties:interpolationEstimates}, $\testfn^n - \testfn^m = 0$ by construction on elements $\E \not \in \mesh[n] \setminus \mesh[m]$, and therefore
	\begin{align*}
		(F, \testfn^n - \testfn^m)
		&=
		\sum_{\E[n] \in \mesh[n] \setminus \mesh[m]}
			(F, \testfn^n - \testfn^m)_{\E[n]}
			\\
		&=
		\sum_{\E[n] \in \mesh[n] \setminus \mesh[m]}
			(F, \testfn^n - \testfn)_{\E[n]}
		+
		\sum_{\E[m] \in \mesh[m] \setminus \mesh[n]}
			(F, \testfn - \testfn^m)_{\E[m]}.
	\end{align*}
	The result now follows from~\eqref{eq:interpolationEstimate} and the fact that
		$\bigcup_{\E \in \mesh[n] \setminus \mesh[m]}
		\E
		=
		\bigcup_{\E \in \mesh[m] \setminus \mesh[n]}
		\E$,
	i.e. both sets of elements form partitions of the same parts of the domain $\domain$.
\end{proof}

\begin{lemma}[Local mesh modification time derivative estimator]\label{lem:residualEstimates:twoMesh:local}
	Assumption~\ref{item:ellipticEstimates:twoMesh} is satisfied by the \emph{local mesh modification time derivative estimator}
	\begin{align}\label{eq:twoMesh:localBound}
		\ellipest{L^2}^{\partial_t}(\trialcurr, \trialprev,
		 \force[\curr], \force[\prev])
		&=
		\frac{1}{\timestep^{\curr}}
		\Big(
		\norm{\transfer[\curr] \trialprev - \trialprev}^2
		\\&\qquad\quad
		+\Cellip
		\Big(
			\norm{h_{\curr}^2 (\elresidualLag{\curr}{\curr}{\trialcurr} - \elresidualLag{\prev}{\prev}{\trialprev})}^2
			 \notag
		\\&\qquad\qquad\qquad\quad
			+
			\norm{h_{\curr}^{3/2} (\jumpresidual[\curr] \trialcurr - \jumpresidual[\prev] \trialprev)}_{\sides[\curr] \cup \sides[\prev]}^2
		\notag
		\\&\qquad\qquad\qquad\quad
			+
			\norm{\hat{h}_{\prev, \curr}^2 \elresidualLag{\prev}{\prev}{\trialprev}}_{\mesh[\prev] \setminus \mesh[\curr]}^2
		\notag
		\\&\qquad\qquad\qquad\quad
			+
			\norm{\hat{h}_{\prev, \curr}^{3/2} \jumpresidual[\prev] \trialprev}_{\sides[\prev] \setminus \sides[\curr]}^2
		\notag
		\\&\qquad\qquad\qquad\quad
		+
    	\inconGroup[\curr](\trialcurr, \force[\curr])^2
	    +
	    \inconGroup[\prev](\trialprev, \force[\prev])^2
		\Big)\Big)^{1/2},
		\notag
	\end{align}
	where the constant $\Cellip>0$ only depends on the regularity of the domain and mesh, $\hat{h}_{\prev, \curr}$ is defined in Lemma~\ref{lem:meshCompatibility}, and $\inconGroup[\curr]$ is defined in Lemma~\ref{lem:residualEstimates:singleMesh}.
\end{lemma}

\begin{proof}
	Let $t\in(t^{\prev},t^{\curr}]$.
	Adopting the notation of Assumption~\ref{item:ellipticEstimates:twoMesh}, we split
	\begin{align*}
		\timestep^{\curr} \norm{\timederivh[\curr] \trialcurr - \trialfn^{\ellipRecon}_t(t)}
		\leq
		\norm{(\ellipReconLag{\curr}{\curr}{\trialcurr} - \trialcurr) - (\ellipReconLag{\prev}{\prev}{\trialprev} - \trialprev)}
		+
		\norm{\transfer[\curr] \trialprev - \trialprev},
	\end{align*}
	and the estimate is obtained by further estimating the first term.
	To this end, let $\dualfn \in H^2(\domain) \cap H^1_0(\domain)$ be the unique solution to the dual elliptic problem
	\begin{align}\label{eq:twoMesh:dualEllipticProblem:local}
		\A(\testfn, \dualfn) = ((\ellipReconLag{\curr}{\curr}{\trialcurr} - \trialcurr) - (\ellipReconLag{\prev}{\prev}{\trialprev} - \trialprev), \testfn) \quad \text{ for all } \testfn\in H^1_0(\domain),
	\end{align}
	which satisfies
		$\abs{\dualfn}_2 \leq \Creg \norm{(\ellipReconLag{\curr}{\curr}{\trialcurr} - \trialcurr) - (\ellipReconLag{\prev}{\prev}{\trialprev} - \trialprev)}$.
	Introducing $\dualfnh[\curr] \in \Vh[\curr]$ and $\dualfnh[\prev] \in \Vh[\prev]$ as the interpolants of $\dualfn$
	satisfying~\eqref{eq:interpolationEstimate},
	we split
	\begin{align*}
		&\norm{(\ellipReconLag{\curr}{\curr}{\trialcurr} - \trialcurr) - (\ellipReconLag{\prev}{\prev}{\trialprev} - \trialprev)}^2
		\notag
		\\
		&\qquad\qquad=
		\A(\ellipReconLag{\curr}{\curr}{\trialcurr} - \trialcurr, \dualfnh[\curr])
		\notag
		+
		\A(\trialprev - \ellipReconLag{\prev}{\prev}{\trialprev}, \dualfnh[\prev])
		\\&\qquad\qquad\quad
    +
		\A((\ellipReconLag{\curr}{\curr}{\trialcurr} - \trialcurr) - (\ellipReconLag{\prev}{\prev}{\trialprev} - \trialprev), \dualfn - \dualfnh[\curr])
		\notag
		\\&\qquad\qquad\quad
		+
		\A(\trialprev - \ellipReconLag{\prev}{\prev}{\trialprev}, \dualfnh[\curr] - \dualfnh[\prev])
		\notag
		=
		\term{1} + \term{2} + \term{3} + \term{4}.
	\end{align*}
	The terms $\term{1}$ and $\term{2}$
	express the inconsistency of $\ellipReconLag{\curr}{\curr}{}$ and $\ellipReconLag{\prev}{\prev}{}$
	and are estimated by $\inconGroup[\curr](\trialcurr, \force[\curr]) \abs{\dualfn}_2$ and $\inconGroup[\prev](\trialprev, \force[\prev]) \abs{\dualfn}_2$ respectively using Lemmas~\ref{lem:ellipreconInconsistency} and~\ref{lem:globalInconsistencyEstimate}.
	Applying the definition of the elliptic reconstruction and residuals, and integrating by parts, $\term{3}$ becomes
	\begin{align*}
		\term{3}
		&
		=
		(\elresidualLag{\curr}{\curr}{\trialcurr} - \elresidualLag{\prev}{\prev}{\trialprev}, \dualfn - \dualfnh[\curr])
		+
		\sum_{\side \in \sides[\curr] \cup \sides[\prev]}
			(\jumpresidual[\curr] \trialcurr - \jumpresidual[\prev] \trialprev, \dualfn - \dualfnh[\curr])_{\side},
	\end{align*}
	and
  Assumption~\ref{item:approximationProperties:interpolationEstimates} and the scaled trace inequality produce
	\begin{align*}
		&
		\term{3}
		\leq
		\Cclem \abs{\dualfn}_2
		\Big(
				\norm{h_{\curr}^2 (\elresidualLag{\curr}{\curr}{\trialcurr} - \elresidualLag{\prev}{\prev}{\trialprev})}^2
			+
			\beta
				\norm{h_{\curr}^{3/2} (\jumpresidual[\curr] \trialcurr - \jumpresidual[\prev] \trialprev) }_{\sides[\curr]}^2
		\Big)^{1/2},
	\end{align*}
	where $\beta > 0$ depends on the regularity of the mesh.

  The term $\term{4}$ represents the non-hierarchicality of the discrete spaces.
	Integrating by parts and recalling the definition of the elliptic reconstruction it becomes
  \begin{align*}
		\term{4}
    =
    (\elresidualLag{\prev}{\prev}{\trialprev}, \dualfnh[\curr] - \dualfnh[\prev])
    +
    \sum_{\side \in \sides[\prev]}
    (\jumpresidual[\prev], \dualfnh[\curr] - \dualfnh[\prev])_{\side}.
  \end{align*}
  The local nature of the interpolant implies that $(\dualfnh[\curr] - \dualfnh[\prev])|_{\E} = 0$ for $\E \in \mesh[\curr] \cap \mesh[\prev]$, and the term may be estimated using Lemma~\ref{lem:meshCompatibility} to find
	\begin{align*}
		\term{4}
		\leq
		C
		\abs{\dualfn}_2
		\Big(
		    \norm{\hat{h}_{\prev, \curr}^2
			\elresidualLag{\prev}{\prev}{\trialprev}}_{\mesh[\prev] \setminus \mesh[\curr]}^2
		    +
		    \norm{\hat{h}_{\prev, \curr}^{3/2}
			\jumpresidual[\prev] \trialprev}_{\sides[\prev] \setminus \sides[\curr]}^2
		\Big)^{1/2}.
	\end{align*}

	The result then follows by invoking the regularity estimate for $\dualfn$.
\end{proof}

Although the proof above makes no assumptions on the relationship between $\mesh[\curr]$ and $\mesh[\prev]$, the estimate itself is clearly best suited to the case when $\mesh[\curr]$ is built by modifying (e.g. coarsening or refining) a small subset of elements from mesh $\mesh[\prev]$.
This is exploited by the estimate through the cancellation of the interpolants into the two discrete spaces in areas of the mesh which are not modified, providing some `smallness' to the estimator terms.

\subsubsection{Time derivative estimator for global mesh modification}
In the case of global mesh modification (e.g. through a complete remeshing procedure or a moving mesh method), however, the mesh elements and edges \emph{cannot} be expected to coincide.
Thus, even if the mesh is only modified slightly, the estimator above will not benefit from local cancellation and may provide a rather pessimistic estimate of the error.

Instead, we handle this case with a separate estimate, presented in Lemma~\ref{lem:residualEstimates:twoMesh:global}, which is slightly more cumbersome but directly relies on similarities between the discrete spaces rather than geometrical similarities between the meshes.
It is based on the following \emph{elliptic transfer operator}, which enables a `discrete integration by parts' to be performed in the analysis to replace a suboptimal term by an optimal one; see Remark~\ref{rem:ellipticTransferRole}.
\begin{definition}[Elliptic transfer operator]\label{def:ellipticTransfer}
	The \emph{elliptic transfer operator} $\canonicalTransfer[\curr] : \Vh[\prev] \to \Vh[\curr]$ satisfies
	\begin{align*}
		\A(\ellipReconLag{\curr}{\prev}{\canonicalTransfer[\curr] \trialprev}, \testcurr)
		=
		\A(\ellipReconLag{\prev}{\prev}{\trialprev}, \testcurr)
		\quad \text{ for all } \trialprev \in \Vh[\prev] \text{ and } \testcurr \in \Vh[\curr].
	\end{align*}
\end{definition}
This essentially transfers the solution from one mesh to another by projecting its elliptic reconstruction onto the new mesh.
Indeed, it satisfies the crucial properties
\begin{align*}
	\ellipProj[\curr] \ellipReconLag{\curr}{\prev}{\canonicalTransfer[\curr] \trialprev}
	=
	\ellipProj[\curr] \ellipReconLag{\prev}{\prev}{\trialprev}
	\quad \text{ and } \quad
	\LTwoProj[\curr] \ellipReconRHSLag{\curr}{\prev}{\canonicalTransfer[\curr] \trialprev}
	=
	\LTwoProj[\curr] \ellipReconRHSLag{\prev}{\prev}{\trialprev},
\end{align*}
where $\ellipProj[\curr]: H^1_0(\domain) \to \Vh[\curr]$ and $\LTwoProj[\curr] : L^2(\domain) \to \Vh[\curr]$ denote the $\A$-orthogonal elliptic projector and $L^2(\domain)$-orthogonal projector onto $\Vh[\curr]$ respectively.
This transfer operator may be practically computed by inverting a stiffness matrix since 
it satisfies
\begin{align*}
	\diffoph[\curr]\canonicalTransfer[\curr] \trialprev
	=
	\LTwoProj[\curr]  \big( \diffoph[\prev] \trialprev
	+
	(\LTwoRecon[\prev] - \identity) \brokenProj[\prev] \force[\prev]
	-
	(\LTwoRecon[\curr] - \identity) \brokenProj[\curr] \force[\prev] \big),
\end{align*}
and this further confirms its existence
and that $\canonicalTransfer[\curr]\trialprev = \trialprev$ when $\Vh[\curr] = \Vh[\prev]$.

\begin{remark}[The role of the elliptic transfer operator]\label{rem:ellipticTransferRole}
	Roughly speaking, the key role played by the elliptic transfer operator is to enable a `discrete integration by parts' by converting a discrete spatial operator on one mesh to a discrete spatial operator on a different mesh which can then be moved onto the test function in an optimal manner.
	In a simplified setting, we perform the following:
	\begin{align*}
		(\diffoph[\prev] \trialprev - \diffoph[\curr] \transfer[\curr] \trialprev, \phi^{\curr})
		=
		(\diffoph[\curr] (\canonicalTransfer[\curr] - \transfer[\curr]) \trialprev, \phi^{\curr})
		=
		((\canonicalTransfer[\curr] - \transfer[\curr]) \trialprev, \diffoph[\curr] \phi^{\curr}),
	\end{align*}
	and apply a stability estimate of the form
		$\norm{\diffoph[\curr] \phi^{\curr}} \leq C \abs{\phi}_2$.
	The elliptic transfer operator is thus responsible for ensuring that the final estimate is of optimal order, because the $H^2(\domain)$-like term $\norm{\diffoph[\prev] \trialprev - \diffoph[\curr] \transfer[\curr] \trialprev}$ which would otherwise be present (and of suboptimal order) is replaced by the (optimal order) $L^2(\domain)$ term $\norm{\canonicalTransfer[\curr] \trialprev - \transfer[\curr] \trialprev}$.
\end{remark}

\begin{lemma}[Time derivative estimator for global mesh modification]\label{lem:residualEstimates:twoMesh:global}
	There exist positive constants $C_1, C_2$ depending only on the domain and mesh regularity, such that
	\begin{align*}
		&\norm{(\ellipReconLag{\curr}{\curr}{} - \Id) \trialcurr - (\ellipReconLag{\curr}{\prev}{} - \Id) \transfer[\curr] \trialprev}
		\leq
		E_1
		\quad \text{and} \quad
		\norm{\ellipReconLag{\curr}{\prev}{\transfer[\curr] \trialprev} - \ellipReconLag{\prev}{\prev}{\trialprev}}
		\leq
		E_2,
	\end{align*}
	where
	\begin{align*}
		E_1
		&=
		C_1
			\Big(
				\norm{h_{\curr}^2 (\elresidualLag{\curr}{\curr}{\trialcurr} - \elresidualLag{\curr}{\prev}{\transfer[\curr] \trialprev})}^2
				+
				\norm{h_{\curr}^{3/2}
					\jumpresidual[\curr] (\trialcurr - \transfer[\curr] \trialprev)}_{\sides[\curr]}^2
		\\&\qquad\qquad
				+
				\inconGroup[\curr](\trialcurr - \transfer[\curr] \trialprev, \force[\curr] - \force[\prev]) ^2
			\Big)^{1/2},
	\end{align*}
	measures error due to time-stepping and spatial discretisation, and
	\begin{align*}
		E_2
		&=
		C_2
			\Big(
				\norm{h_{\curr}^2 \big( \ellipReconRHSLag{\curr}{\prev}{\transfer[\curr] \trialprev}
				-
				\ellipReconRHSLag{\prev}{\prev}{\trialprev} \big)}^2
			\\&\qquad\quad
				+
				\norm{\transfer[\curr] \trialprev - \canonicalTransfer[\curr] \trialprev}^2
				+
				\sum_{\E \in \mesh[\curr]}
				h_{\E}^2
				\triplenorm{\transfer[\curr] \trialprev - \canonicalTransfer[\curr] \trialprev}_{\E}^2
			\\&\qquad\quad
				+
				\projErrorLTwo[\curr](h_{\curr}^2, \diffoph[\curr] (\transfer[\curr] \trialprev - \canonicalTransfer[\curr] \trialprev))^2
				+
				\projErrorEnergy[\curr](h_{\curr}, \transfer[\curr] \trialprev - \canonicalTransfer[\curr] \trialprev)^2
			\Big)^{1/2},
	\end{align*}
	measures error produced by modifying the discrete space.
	Consequently, Assumption~\ref{item:ellipticEstimates:twoMesh} is satisfied by
	the \emph{time derivative estimator for global mesh modification}
	\begin{align}
		\ellipest{L^2}^{\partial_t}(\trialcurr, \trialprev,
		 \force[\curr], \force[\prev])
		&=
		\frac{1}{\timestep^{\curr}} (E_1 + E_2).
		\label{eq:twoMesh:globalBound}
	\end{align}
\end{lemma}

\begin{proof}
	Adopting the notation of Assumption~\ref{item:ellipticEstimates:twoMesh}, we split the target term as
	\begin{align*}
		\timestep^{\curr}\norm{\timederivh[\curr] \trialcurr - \trialfn^{\ellipRecon}_t(t)}
		\leq
		&
		\norm{(\ellipReconLag{\curr}{\curr}{} - \Id) \trialcurr - (\ellipReconLag{\curr}{\prev}{} - \Id) \transfer[\curr] \trialprev}
		\\&
		+ \norm{\ellipReconLag{\curr}{\prev}{\transfer[\curr] \trialcurr} - \ellipReconLag{\prev}{\prev}{\trialprev}}
		= \term{1}+\term{2},
	\end{align*}
	and show that $E_1$ estimates $\term{1}$ (involving only quantities on one mesh at both the time-steps) and $E_2$ estimates $\term{2}$ (containing terms at time $t^{\prev}$ on both meshes).

	To estimate $\term{1}$, let $\dualfn \in H^2(\domain) \cap H^1_0(\domain)$ satisfy the dual elliptic problem
	\begin{align}\label{eq:twoMesh:dualEllipticProblem:global}
		\A(\testfn, \dualfn) = ((\ellipReconLag{\curr}{\curr}{} - \Id) \trialcurr - (\ellipReconLag{\curr}{\prev}{} - \Id) \transfer[\curr] \trialprev
		, \testfn)
		 \quad \text{ for all } v \in H^1_0(\domain),
	\end{align}
	which provides the regularity estimate
		$\abs{\dualfn}_2
		\leq
		\Creg
		\term{1}$
	and let $\dualfnh[\curr] \in \Vh[\curr]$ denote the interpolant of $\dualfn$ satisfying~\eqref{eq:interpolationEstimate}.
	Then,
	we have
	\begin{align*}
		\term{1}^2
		=&
		\A(
		(\ellipReconLag{\curr}{\curr}{} - \Id)\trialcurr - (\ellipReconLag{\curr}{\prev}{} - \Id) \transfer[\curr] \trialprev
		, \dualfnh[\curr])
		\\
		&+
		\A(
		(\ellipReconLag{\curr}{\curr}{} - \Id)\trialcurr - (\ellipReconLag{\curr}{\prev}{} - \Id) \transfer[\curr] \trialprev
		, \dualfn - \dualfnh[\curr])
		= \term{3} + \term{4}.
		\notag
	\end{align*}
	The term $\term{3}$ expresses the reconstruction inconsistency, so Lemmas~\ref{lem:ellipreconInconsistency} and~\ref{lem:globalInconsistencyEstimate} give
	\begin{align*}
		\term{3}
		\leq
			C \abs{\dualfn}_2
				\inconGroup[\curr](\trialcurr - \transfer[\curr]\trialprev, \force[\curr] - \force[\prev]).
	\end{align*}
	For $\term{4}$, the definition of the elliptic reconstruction and integration by parts provide
	\begin{align*}
		\term{4}
		&=
		(\elresidualLag{\curr}{\curr}{\trialcurr} - \elresidualLag{\curr}{\prev}{\transfer[\curr] \trialcurr}, \dualfn - \dualfnh[\curr])
		+
		\sum_{\side \in \sides[\curr]}
			(\jumpresidual[\curr] (\trialcurr - \transfer[\curr] \trialprev), \dualfn - \dualfnh[\curr])_{\side},
	\end{align*}
	and the Cauchy-Schwarz inequality, Assumption~\ref{item:approximationProperties:interpolationEstimates} and the trace inequality produce
	\begin{align*}
		\term{4}
		\leq
		\Cclem \abs{\dualfn}_2
		\Big(
			\norm{h_{\curr}^2 (\elresidualLag{\curr}{\curr}{\trialcurr} - \elresidualLag{\curr}{\prev}{\transfer[\curr] \trialprev} )}^2
			+
			\beta
			\norm{h_{\curr}^{3/2}
				\jumpresidual[\curr] (\trialcurr - \transfer[\curr] \trialprev)}_{\sides[\curr]}^2
		\Big)^{1/2},
	\end{align*}
	where $\beta$ depends on the regularity of the mesh.
	The estimate $E_1$ then follows by invoking the regularity estimate for $\dualfn$.

	Next, we show how $\term{2}$ may be estimated by $E_2$.
	For this, we use the auxiliary dual problem: find $\phi \in H^2(\domain) \cap H^1_0(\domain)$ such that
	\begin{align*}
		\A(v, \phi) = (\ellipReconLag{\curr}{\prev}{\transfer[\curr] \trialprev} - \ellipReconLag{\prev}{\prev}{\trialprev}, v) \quad \text{ for all } v \in H^1_0(\domain),
	\end{align*}
	which satisfies the regularity estimate
		$\abs{\phi}_2 \leq \Creg
		\term{2}$.
	We therefore obtain
	\begin{align*}
		\term{2}^2
		=
		\A(\ellipReconLag{\curr}{\prev}{\transfer[\curr] \trialprev} - \ellipReconLag{\prev}{\prev}{\trialprev}, \phi),
	\end{align*}
	and the definition of the elliptic transfer operator and elliptic reconstruction imply
	\begin{align*}
		\term{2}^2
		&
		=
		\big(
			\ellipReconRHSLag{\prev}{\prev}{\trialprev}
			-
			\ellipReconRHSLag{\curr}{\prev}{\transfer[\curr] \trialprev}
			, \phi - \phi^{\curr}
		\big)
		+
		\big(
			\diffoph[\curr] (\canonicalTransfer[\curr] \trialprev - \transfer[\curr] \trialprev)
			, \phi^{\curr}
		\big).
	\end{align*}
	The Cauchy-Schwarz inequality and Assumption~\ref{item:approximationProperties:interpolationEstimates} provide
	\begin{align*}
		&\big(
			\ellipReconRHSLag{\curr}{\prev}{\transfer[\curr] \trialprev}
			-
			\ellipReconRHSLag{\prev}{\prev}{\trialprev}, \phi - \phi^{\curr}
		\big)
		\leq
		C \abs{\phi}_2
			\norm{h_{\curr}^2 \big( \ellipReconRHSLag{\curr}{\prev}{\transfer[\curr] \trialprev}
			-
			\ellipReconRHSLag{\prev}{\prev}{\trialprev} \big)},
	\end{align*}
	while
	the definition of $\diffoph[\curr]$ gives
	\begin{align*}
		(
			\diffoph[\curr] (\transfer[\curr] - \canonicalTransfer[\curr]) \trialprev, \phi^{\curr}
		)
		&=
		\A((\transfer[\curr] - \canonicalTransfer[\curr]) \trialprev, \phi)
		+
		\A((\transfer[\curr] - \canonicalTransfer[\curr]) \trialprev, \phi^{\curr} - \phi)
		\\&\qquad
		+
		\inconsistencym[\curr] ((\transfer[\curr] - \canonicalTransfer[\curr]) \trialprev, \phi^{\curr})
		+
		\inconsistencya[\curr]((\transfer[\curr] - \canonicalTransfer[\curr]) \trialprev, \phi^{\curr}).
	\end{align*}
Assumptions~\ref{item:approximationProperties:interpolationEstimates} and \ref{item:approximationProperties:inconsistencyEstimate}
and an $L^2$---$H^2$ splitting of $\A$ then provide the estimate
	\begin{align*}\label{eq:discreteIntegrationByParts}
		&(
			\diffoph[\curr] (\transfer[\curr] - \canonicalTransfer[\curr]) \trialprev, \phi^{\curr}
		)
		\leq
		C
		\abs{\phi}_2
		\Big(
			\norm{(\transfer[\curr] - \canonicalTransfer[\curr]) \trialprev}^2
			+
			\sum_{\E \in \mesh[\curr]} h_{\E}^2 \triplenorm{(\transfer[\curr] - \canonicalTransfer[\curr]) \trialprev}_{\E}^2
		\\&\qquad\qquad\qquad\qquad
			+
			\projErrorLTwo[\curr](h_{\curr}^2, \diffoph[\curr] (\transfer[\curr] - \canonicalTransfer[\curr]) \trialprev)^2
			+
			\projErrorEnergy[\curr](h_{\curr}, (\transfer[\curr] - \canonicalTransfer[\curr]) \trialprev)^2
		\Big)^{1/2},
		\notag
	\end{align*}
	and the result follows by using the regularity estimate for $\phi$.
\end{proof}

We observe that the first four terms of the estimator $E_1$ mimic those of the previous estimate~\eqref{eq:twoMesh:localBound}.
The fifth term of~\eqref{eq:twoMesh:localBound} is instead replaced by the mesh modification estimator $E_2$ which satisfies the desirable property of being zero when $\Vh[\curr] = \Vh[\prev]$ if $\transfer[\curr]$ becomes the identity operator.
The estimator $E_2$ can therefore be seen as being composed of two groups of terms: the first term measures the impact that transferring the solution to the new space has on the modified discrete spatial operator (and is just an $L^2(\domain)$ projection error when $\transfer[\curr] = \canonicalTransfer[\curr]$), while the other terms compare the extent to which $\transfer[\curr]$ is different from $\canonicalTransfer[\curr]$ (and are zero when $\transfer[\curr] = \canonicalTransfer[\curr]$).
The estimator $E_2$ could be simplified further under the additional assumption of a local $H^1$ to $L^2$ inverse estimate to remove the third term.

\section{Application to finite element discretisations}\label{sec:fem}

A conventional conforming finite element method may be seen to fit the abstract framework of Section~\ref{sec:abstractFramework}.
To satisfy Assumption~\ref{item:abstractFramework:mesh}, we suppose that the computational mesh $\mesh[\curr]$ is formed of elements which are the image of a reference simplex or hypercube under a non-degenerate affine mapping.
For simplicity, we suppose that all elements in the mesh are images of the same reference element.
The local discrete function space $\Vh[\E]$ required by Assumption~\ref{item:abstractFramework:fespace} may then be constructed on each element $\E \in \mesh[\curr]$ as
\begin{equation}
	\Vh[\E] :=
	\begin{cases}
		\PE{\k} &\text{if } \E \text{ is the image of a simplex},
		\\
		\QE{\k} &\text{if } \E \text{ is the image of a hypercube},
	\end{cases}
\end{equation}
where $\PE{\k}$ denotes the space of polynomials of \emph{total} degree $\k$ on $\E$, and $\QE{\k}$ denotes the space of tensor-product polynomials of \emph{maximum} degree $\k$ on $\E$.

The discrete bilinear forms of Assumption~\ref{item:abstractFramework:bilinearForms} are constructed consistently, with
\begin{align*}
	\mh[\curr](\trialcurr, \testcurr) = (\trialcurr, \testcurr)
	\quad\text{ and }\quad
	\Ah[\curr](\trialcurr, \testcurr) = \A(\trialcurr, \testcurr),
\end{align*}
meaning that
\begin{align*}
	\inconsistencym[\curr](\trialcurr, \testcurr) = 0
	=
	\inconsistencya[\curr](\trialcurr, \testcurr),
\end{align*}
thus satisfying condition~\ref{item:approximationProperties:inconsistencyEstimate} with $\projErrorLTwo[\E] = \projErrorEnergy[\E] = 0$.
To satisfy Assumption~\ref{item:abstractFramework:forcingTerm} we take $\brokenProj[\curr]$ to be the identity operator (implying that $\forceProj[\curr] = \force[\curr]$ and so $\dataest{S}(t) = 0$),
and $\forceh[\curr] \in \Vh[\curr]$ is therefore defined (by~\eqref{eq:forcehDefinition}) to be the $L^2(\domain)$-orthogonal projection of $\force[\curr]$ into $\Vh[\curr]$.
Alternatively, to assess the impact of numerical integration, $\forceProj[\curr]$ could be fixed as the polynomial interpolating $\force[\curr]$ at the nodes of an appropriate quadrature scheme.

We consider the heat equation
\begin{align}\label{eq:modelHeatEquation}
	u_t - \alpha \Delta u = \force(x,y,t),
\end{align}
where $\alpha = 0.01$, and $\force$ and the boundary conditions are constructed to give the exact solution
\begin{align}\label{eq:fem:numerics:hat}
	u(r, t) = g(t) s(-m(t) (r^2 - r_0^2)),
\end{align}
where $r^2 = x^2 + y^2$ and
\begin{align*}
	s(\xi) = 1 + \tanh(\xi), \quad m(t) = \frac{100}{3 t + 2}, \quad g(t) = \frac{10}{t^2 + 20}, \quad r_0 = 0.15,
\end{align*}
for $(x,y) \in \domain = [0,1]^2$ and $t \in [0, 5]$.
This solution mimics the behaviour of an initially localised concentration of solute (shown in Fig.~\ref{fig:fem:numerics:hat:meshes}) dissolving into a medium, featuring initially steep gradients which diffuse away over time.
We discretise this problem using an ad-hoc moving mesh method, in which the topology of the mesh remains fixed throughout the simulation (i.e. there is no refinement or coarsening), although the vertices of the elements are mapped so that the areas of refinement are concentrated around the layers of the solution by a time-dependent mapping function determined a priori.
Sample meshes are shown in Fig.~\ref{fig:fem:numerics:hat:meshes}.
Comparable numerical results for an analogous estimator using a fixed computational mesh are given in~\cite[Section 7]{Sutton:2018ex}.

In this setting, it is appropriate to select Lemma~\ref{lem:residualEstimates:twoMesh:global} to provide the time derivative estimator, and by performing mesh transfer using the elliptic transfer operator from Definition~\ref{def:ellipticTransfer}, the spatial error estimator takes the simpler form
\begin{align*}
	\spaceest(t)
	=
	\frac{C}{\timestep^{\curr}}
	\Big(&
		\norm{h_{\curr}^2 (\elresidualLag{\curr}{\curr}{\trialcurr} - \elresidualLag{\curr}{\prev}{\canonicalTransfer[\curr] \trialprev})}^2
		+
		\norm{h_{\curr}^{3/2}
			\jumpresidual[\curr] (\trialcurr - \canonicalTransfer[\curr] \trialprev)}_{\sides[\curr]}^2
		\\&
		+
		\norm{h_{\curr}^2 \big( \Proj{\curr}{L^2} - \identity \big)
		\ellipReconRHSLag{\prev}{\prev}{\trialprev}}^2
	\Big)^{1/2}.
\end{align*}
Since the chief novelty here is the $L^{\infty}(0,t; L^2(\domain))$ estimate, for brevity we do not present the results for the $L^2(0,t; H^1(\domain))$ estimate.

The results are computed as a sequence of simulations, using linear finite element spaces over progressively finer meshes.
The computational mesh for each simulation is constructed by warping a fixed primitive mesh formed of uniform square elements, using the same (time-dependent) mesh warping function for each simulation.
The mesh used on simulation $i$, with $i \geq 1$, consists of $2^{2(i+2)}$ elements, and the diameter of the elements on the primitive mesh may therefore be computed in each case as $h_i = 2^{1/2 - 2(i+1)}$.
By linking the size of the time-step $\timestep_i$ to the size of $h_i$, we may therefore define a meaningful notion of convergence rate with respect to $i$ as a function of time, computed for a quantity $F^i(t)$ as
\begin{align}
	\operatorname{rate}_i(t) = \frac{\log(F^i(t)) - \log(F^{i-1}(t))}{\log(h_i) - \log(h_{i-1})}.
\end{align}

In Figures~\ref{fig:fem:numerics:hat:2} and~\ref{fig:fem:numerics:hat:1} we plot the results obtained with $\timestep_i = h_i^2$ and $\timestep_i = h_i$ respectively.
Each line on the plots represents  a single simulation in the sequence, and the solid line represents the results obtained on the finest mesh.
Subfigure (A) in each case shows the behaviour of the error and estimator alongside the \emph{effectivity} of the estimator, calculated as the ratio of estimated to true error, i.e.
\begin{align}
	\operatorname{effectivity}_i(t) = \frac{\operatorname{estimator}_i(t)}{\norm{u - \timerecon_i}_{L^{\infty}(0,t; L^2(\domain))}},
\end{align}
where, $\operatorname{estimator}_i(t)$ and $\timerecon_i(t)$ denote the estimator and discrete solution calculated on mesh $i$.

A priori error estimates for finite element discretisations of parabolic problems with backward Euler time-stepping (albeit on fixed computational meshes) lead us to expect the simulation error to be of the order $\mathcal{O}(h^2 + \timestep)$, translating to an expected convergence rate with respect to $i$ of 2 and 1 for the results in Figures~\ref{fig:fem:numerics:hat:2} and~\ref{fig:fem:numerics:hat:1}, respectively.
This is indeed what we observe numerically, for both the error and estimator.
Moreover, the effectivity values of approximately 16 and 2 in the two cases respectively indicate a good level of agreement between the true and estimated error.
For the simulations with $\timestep_i = h_i$ (Figure~\ref{fig:fem:numerics:hat:1}), it appears that the effectivity initially grows as $i$ increases (and thus the spatial and temporal discretisations become finer), before levelling off on the final simulations.
This may be attributed to the fact that the true error converges faster than expected between the first simulations, a pre-asymptotic effect which is reflected to a lesser extent by the estimator.
Once the asymptotic regime is reached, however, both the error and estimator converge at the same rate and the effectivities stabilise.
In the case of $\timestep_i = h_i^2$, we observe that the error and estimator both express approximately the expected convergence rate for all $i$, and the effectivity is therefore predictably smaller and more stable.

A further key observation we draw from these results is that the behaviour of the components of the estimator (plotted in subfigure (B) of each figure) appears to vindicate our choice of names for them.
In particular, when $\timestep_i = h_i$, the time estimator and data approximation estimator converge at order one, whereas when $\timestep_i = h_i^2$, they both converge at order two.
This matches the behaviour we would expect from a priori estimates, given the scaling of the time-step, and we note that it is these components which are of order one and therefore responsible for restricting the estimate to the correct convergence rate when $\timestep_i = h_i$.
The component of the estimator designated as measuring the spatial error may be seen to converge at rate two in both cases, alongside the elliptic component of the estimator.

\begin{figure}
	\centering{%
	\includegraphics[width=0.3\linewidth]{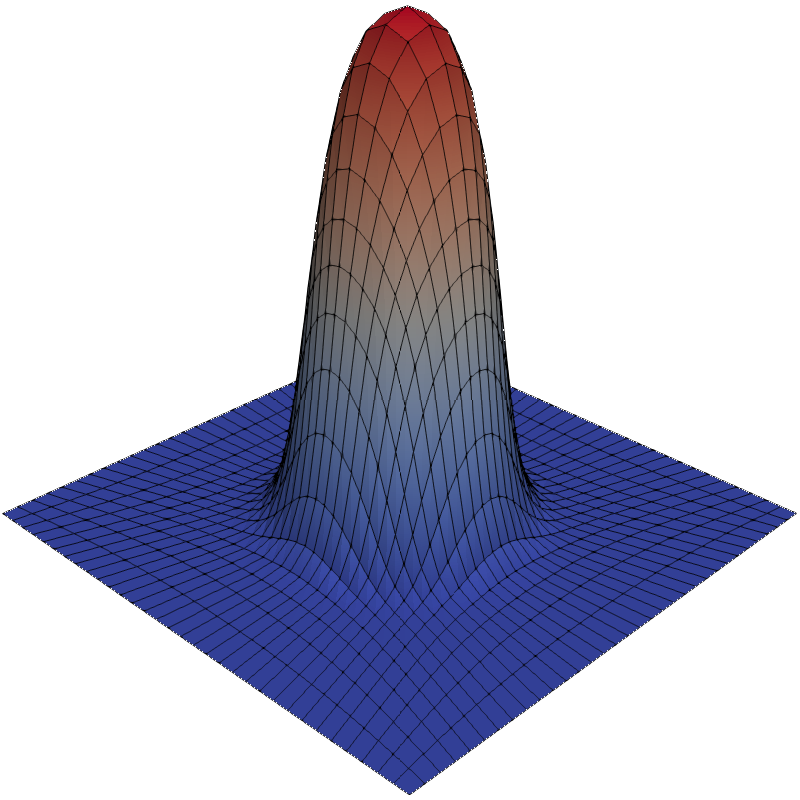}\quad%
	\includegraphics[width=0.3\linewidth]{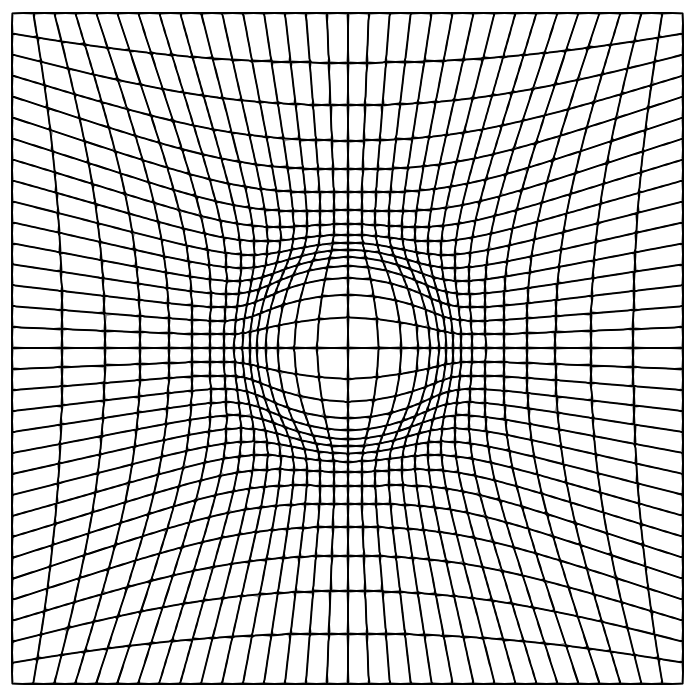}\quad%
	\includegraphics[width=0.3\linewidth]{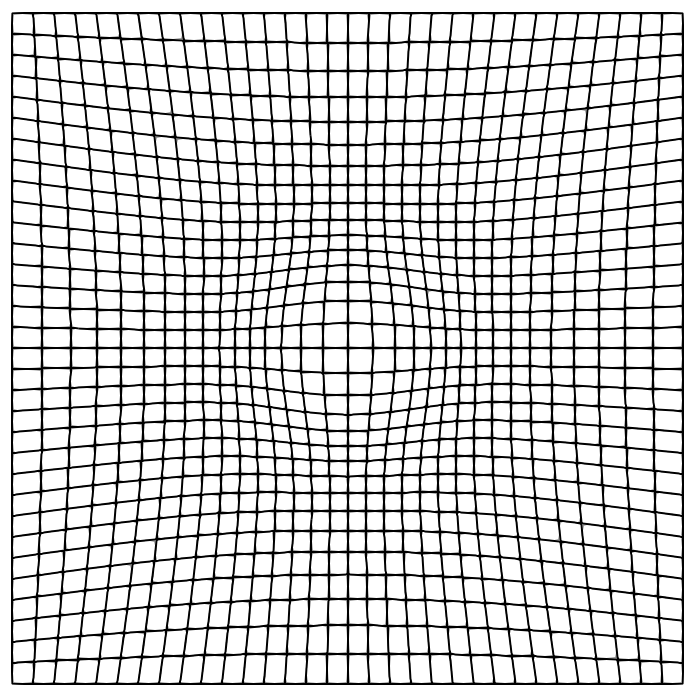}%
	}
	\caption{Initial condition and initial and final meshes, selected a priori to fit the layers of the solution~\eqref{eq:fem:numerics:hat}.}
	\label{fig:fem:numerics:hat:meshes}
\end{figure}

\begin{figure}%
  \centering{%
  \subcaptionbox{Error and estimator}[\textwidth]{\includegraphics[width=0.9\textwidth]{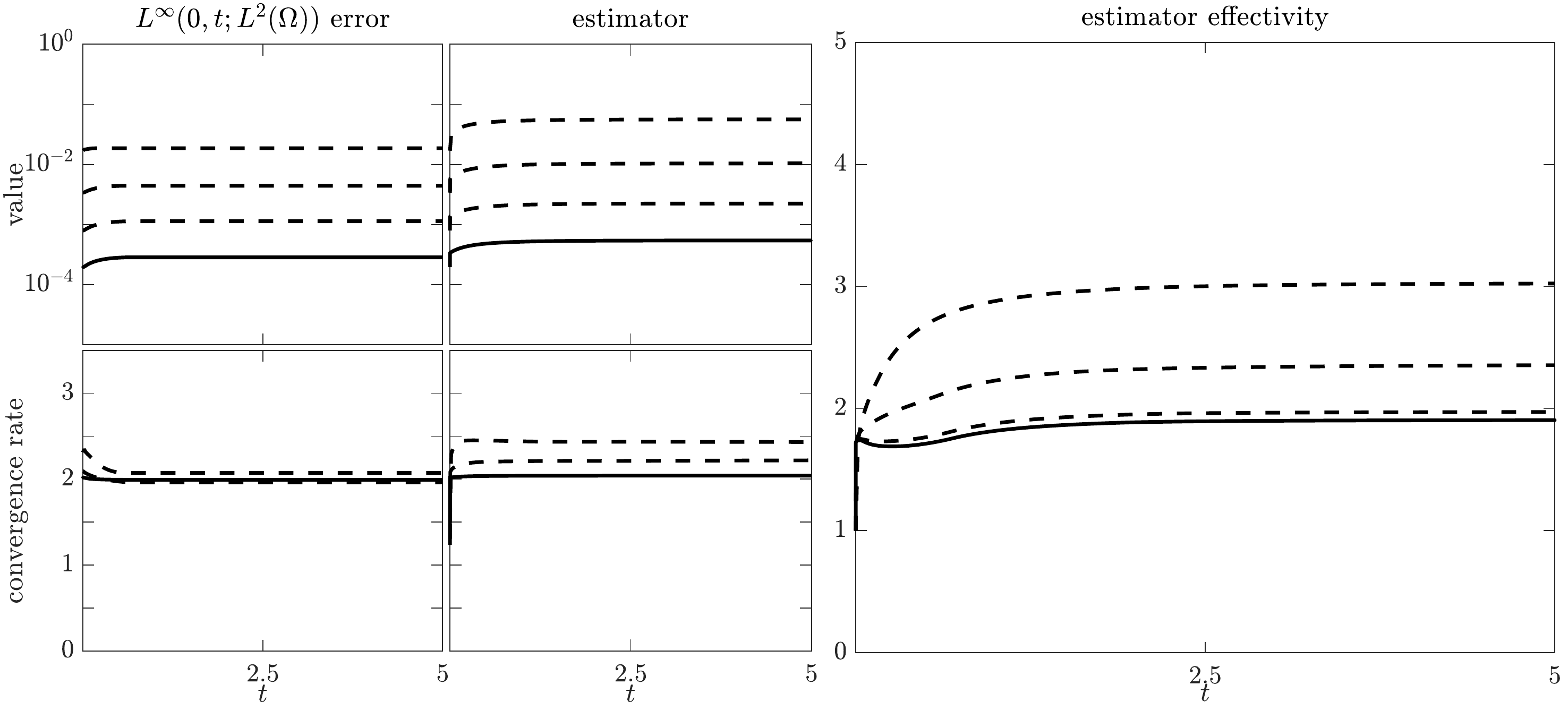}\vspace{-0.75em}}
  \vspace{1em}

  \subcaptionbox{Components of the estimator}[\textwidth]{\includegraphics[width=0.9\textwidth]{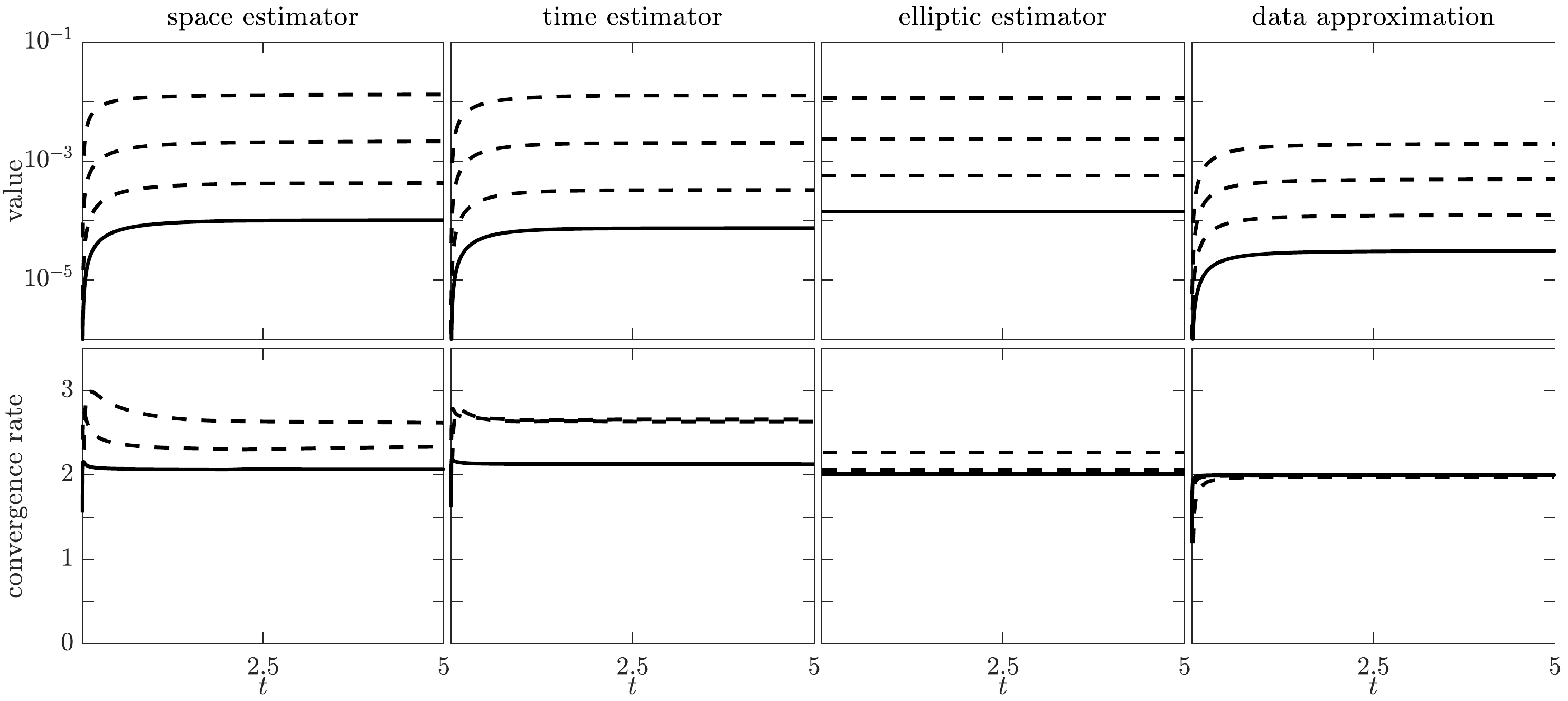}\vspace{-0.75em}}
  }
  \caption{Behaviour of the error and estimator for the moving mesh example~\eqref{eq:fem:numerics:hat} with $\tau \approx h^2$. Solid lines indicate results on the finest mesh.}
  \label{fig:fem:numerics:hat:2}
\end{figure}

\begin{figure}%
  \centering{%
  \subcaptionbox{Error and estimator}[\textwidth]{\includegraphics[width=0.9\textwidth]{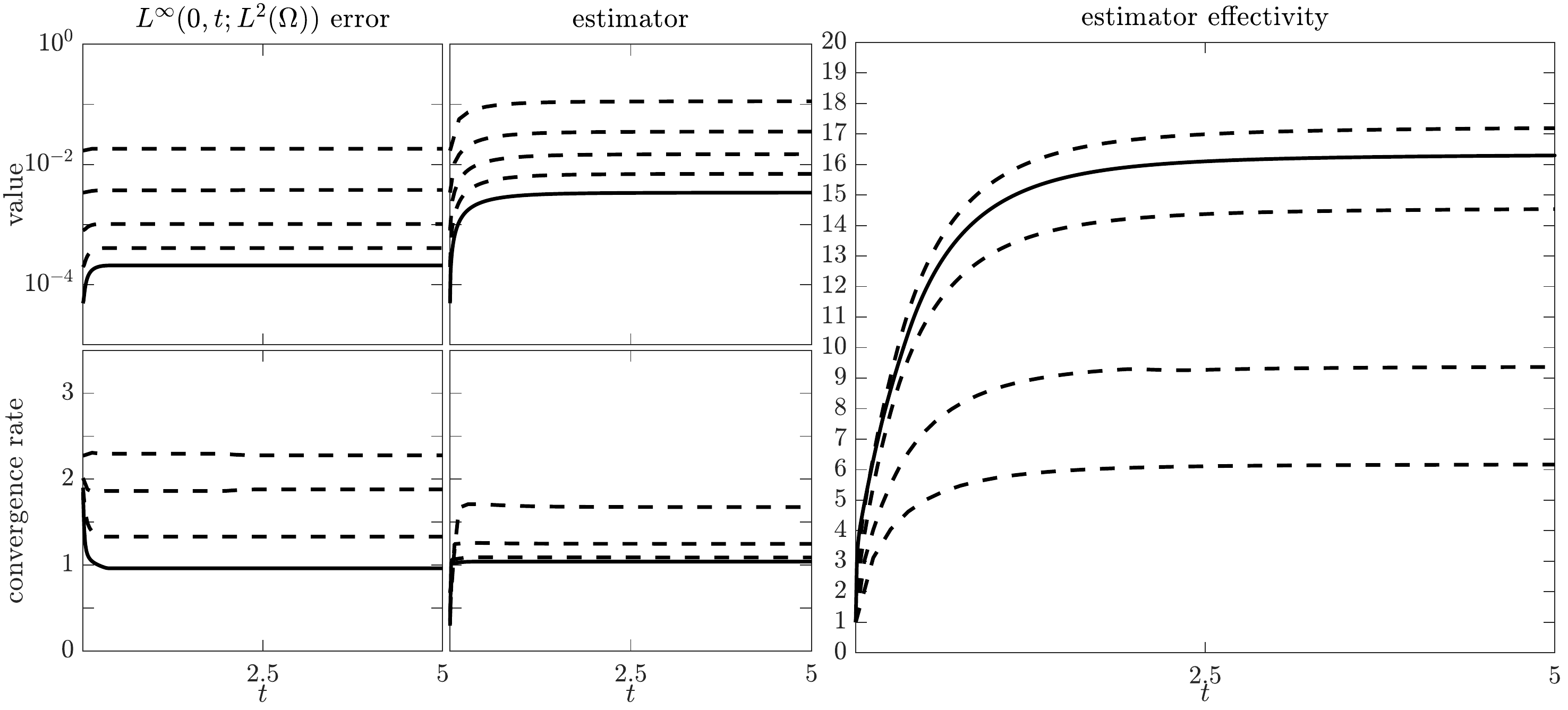}\vspace{-0.75em}}
  \vspace{1em}

  \subcaptionbox{Components of the estimator}[\textwidth]{\includegraphics[width=0.9\textwidth]{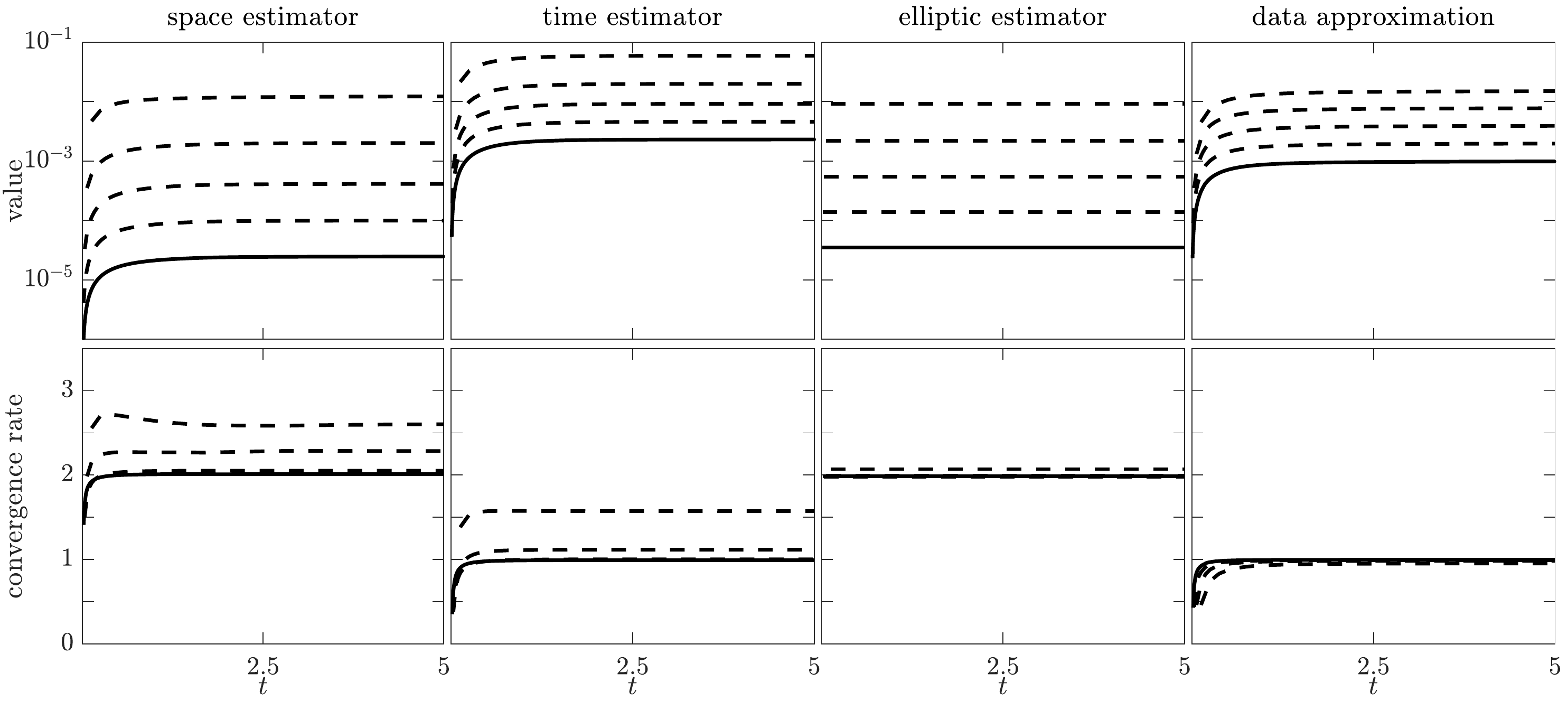}\vspace{-0.75em}}
  }
  \caption{Behaviour of the error and estimator for the moving mesh example~\eqref{eq:fem:numerics:hat} with $\tau \approx h$. Solid lines indicate results on the finest mesh.}
  \label{fig:fem:numerics:hat:1}
\end{figure}

\section{Application to virtual element discretisations}\label{sec:vem}

We now show how a virtual element method fits into the abstract framework above, and derive computable error estimates which are valid for general adaptive polygonal meshes.
In particular, the construction of the discrete function spaces used in the virtual element method means that even hierarchical refinement of the mesh does not lead to hierarchical sequences of spaces.
Table~\ref{table:vem:discreteComponents} provides an outline of the construction of the method's components.

\begin{table}
\begin{tabular}{c|c|c}
	Item & Conditions & Construction
	\\ \hline
	Mesh & Required by~\ref{item:abstractFramework:mesh} & Section~\ref{sec:vem:mesh}
	\\ 
	Discrete space & Required by~\ref{item:abstractFramework:fespace}, satisfying~\ref{item:approximationProperties:interpolationEstimates}
 & Section~\ref{sec:vem:spaces}
	\\ 
	Bilinear forms & Required by~\ref{item:abstractFramework:bilinearForms}, satisfying~\ref{item:approximationProperties:inconsistencyEstimate} & Section~\ref{sec:vem:bilinearForms}
	\\ 
	Forcing data & Required by~\ref{item:abstractFramework:forcingTerm}, satisfying~\eqref{eq:approximationProperties:dataApproximation} & Section~\ref{sec:vem:bilinearForms}
	\\ 
	Mesh transfer operator & Required by~\ref{item:abstractFramework:meshTransfer} & Section~\ref{sec:vem:meshtransfer}
	\\ 
	Elliptic estimates & Satisfying~\ref{item:ellipticEstimates:singleMesh} and~\ref{item:ellipticEstimates:twoMesh} & Section~\ref{sec:vem:ellipticEstimates}
\end{tabular}
\caption{The sections in which the components of a virtual element method satisfying the conditions of the abstract discrete scheme are constructed.}
\label{table:vem:discreteComponents}
\end{table}

\subsection{Mesh}\label{sec:vem:mesh}
The virtual element method may be applied in the context of very general polygonal or polyhedral meshes satisfying the following conditions.
\begin{assumption}[Polytopic mesh]
	We suppose that the mesh $\mesh[\curr]$ satisfies Assumption~\ref{item:abstractFramework:mesh} alongside
	\begin{enumerate-assumption}
		\item \label{item:vem:mesh:starshaped}
			Every element $\E$ of $\mesh[\curr]$ is star-shaped with respect to a ball of radius $\meshReg h_{\E}$;
		\item For $\spacedim=3$, every side $\side \in \sides[\curr]$ viewed as a $2$-dimensional element satisfies conditions~\ref{item:abstractFramework:mesh} and~\ref{item:vem:mesh:starshaped}.
	\end{enumerate-assumption}
\end{assumption}
For simplicity, these assumptions will be used in the subsequent analysis although they are more restrictive than is required in practice~\cite{Brenner:2018ce,BeiraodaVeiga:2017fh}.

\subsection{Local projection operators}\label{sec:vem:projectors}
The construction of the method hinges around the construction of elementwise $H^1$- and $L^2$-orthogonal projection operators onto piecewise polynomials, defined as follows.
For an element $\E \in \mesh[\curr]$ and polynomial degree $\ell \in \mathbb{N}$,
let $\Proj{\E}{\nabla,\ell} :H^1(\E)\rightarrow\PE{\k}$ denote
the $H^1(\E)$-orthogonal projector onto polynomials of degree $\ell$ (when $\ell = \k$, we will use the abbreviation $\Proj{\E}{\nabla} = \Proj{\E}{\nabla, \k}$), defined such that, for any $\trialfn \in H^1(\E)$,
\begin{align*}
	(\nabla \trialfn - \nabla \Proj{\E}{\nabla,\ell} \trialfn, \nabla p)_{\E} = 0, \quad \text{ for all } p \in \PE{\k},
\end{align*}
and
$\intdE (\trialfn - \Proj{\E}{\nabla,\ell} \trialfn) \dS = 0$
 for
 $\ell = 1$,
 or
$\intE (\trialfn - \Proj{\E}{\nabla,\ell} \trialfn) \dx = 0$
 otherwise.

The $L^2(\E)$-orthogonal projector $\Proj{\E}{\ell} : L^2(\E) \to \PE{\ell}$ of $\trialfn \in L^2(\E)$ satisfies
\begin{align*}
	(\trialfn - \Proj{\E}{\ell} \trialfn, p)_{\E} = 0, \quad \text{ for all } p \in \PE{\ell},
\end{align*}
and we define $\Proj{\curr}{\ell} : L^2(\domain) \to \P{\curr}{\ell}$ as
$
	(\Proj{\curr}{\ell} \trialfn)|_{\E} = \Proj{\E}{\ell} (\trialfn|_{\E})
$
for each $\E \in \mesh[\curr]$.
The following approximation properties for $\Proj{\E}{\ell}$ on star-shaped domains $\E$ are provided by~\citer{Brenner:2008tq}.
\begin{theorem}[Projection error estimate]
	\label{thm:polynomialApproximation}
	Let $\ell \geq 0$ be an integer and $1 \leq \reg \leq \ell+1$.
	Then, there exists a positive constant $\Cproj$ depending only on $\ell$ and the mesh regularity such that, for any $\E \in \mesh[\curr]$ and $\trialfn \in H^{\reg}(\E)$, we have
	\begin{equation*}
		\norm{\trialfn - \Proj{\E}{\ell} \trialfn}_{0,\E} + h_{\E} \abs{\trialfn - \Proj{\E}{\ell} \trialfn}_{1,\E} \leq \Cproj h_{\E}^\reg \abs{\trialfn}_{\reg,\E}.
	\end{equation*}
\end{theorem}

\subsection{Virtual element space}\label{sec:vem:spaces}
We now recall the construction of the local virtual element space of order $\k \in {\mathbb N}$; see~\cite{Ahmad:2013uaa,Cangiani:2016vh} for further details.
On each element $\E \in \mesh[\curr]$, the local virtual element space $\Vh[\E]$ consists of a subspace $\PE{\k}$ of polynomials complemented by a subspace of non-polynomial functions which are implicitly defined as solutions to local boundary value problems.
The key to the virtual element methodology is that these extra non-polynomial \emph{virtual} functions never need to be known explicitly, ensuring that the local boundary value problems are never solved in practice.
Instead, the virtual element functions are only accessed through a set of \emph{degrees of freedom} of the following types.

\begin{definition}[Degrees of freedom]
  \label{def:dofs}
  Let $\omega\subset\Re^\spacedim$, $1\le \spacedim\le 3$, be a $\spacedim$-dimensional
  polytope (i.e. a line segment, polygon, or polyhedron,
  respectively). For any sufficiently regular function $v$ on $\omega$, we define the following types of \emph{degrees of freedom}:
  \begin{itemize}
  \item $\NO{\omega}$ are the \emph{nodal values}. For a vertex $\vertex$
    of $\omega$, $\NO{\omega}_\vertex (v):=v(\vertex)$ and
    $\NO{\omega}:=\{\NO{\omega}_\vertex: \vertex \text{ is a
      vertex}\}$;
  \item $\MO{\omega}{l}$ are the \emph{polynomial moments} up to order
    $\ell \ge 0$:
    \begin{equation*}
      \MO{\omega}{\mindex}(v) = \frac{1}{\abs{\omega}} (v, m_{\mindex})_\omega
      \quad\text{ with }\quad
      m_{\mindex}
      :=
      \left(\frac{ \vx - \vx_{\omega}}{h_{\omega}}\right)^{\mindex}
      \text{ and}\quad
      \abs{\mindex} \le \ell,
    \end{equation*}
    where $\mindex$ is a multi-index with $\abs{\mindex} := \alpha_1
    +\cdots +\alpha_\spacedim$ and $x^{\mindex} := x_1^{\alpha_1} \dots
    x_\spacedim^{\alpha_\spacedim}$ in a local coordinate system, and $x_\omega$
    denotes the barycentre of $\omega$. Further,
    $\MO{\omega}{\ell}=\{\MO{\omega}{\mindex}:\abs{\mindex} \leq \ell \}$
		and
    $\MO{\omega}{-1}:=\emptyset$.
  \end{itemize}
\end{definition}

The construction of the space is recursive in the spatial dimension.
We start with a line segment $\edge$ and define $\Vh[\edge] = \P{\edge}{\k}$, which is described by the degrees of freedom
\begin{align*}
	\operatorname{DoF}(\Vh[\edge])
	:=
	\{
		\NO{\edge},
		\quad
		\MO{\edge}{\k-2}
	\}.
\end{align*}
Then, for a polytope $\E \subset \Re^{\spacedim}$ with $\spacedim = 2,3$, the space $\Vh[\E]$ is defined (recursively) as
\begin{align*}
	\VhE
	:=
	\Big\{
	  	v \in H^1(\E) :
			\,\,
		&\Delta v \in \PE{\k},
		\quad
		v|_{\side} \in \Vh[\side] \text{ for all } \side \in \sides[\E],
		\quad
		v|_{\partial\E} \in C^0(\partial\E),
		\\
		&\text{ and }
		\MO{\E}{\mindex}(v-\Proj{\E}{\nabla} v) = 0
		\text{ for all } \abs{\mindex} = \k, \k-1
	\Big\},
\end{align*}
and the set of degrees of freedom describing this space is also defined recursively as
\begin{align*}
	\operatorname{DoF}(\VhE)
	:=
	\bigcup_{\side \in \sides[\E]}
		\operatorname{DoF}(\Vh[\side])
	\cup
		\MO{\E}{\k-2},
\end{align*}
with the convention that degrees of freedom associated with vertices and edges are only counted once, even when they are shared by several edges or faces.
The global space on the mesh $\mesh[\curr]$ is constructed from these local spaces as in~\eqref{eq:globalSpaceDefinition}, and the global degrees of freedom are obtained by collecting the local ones, with the same convention for shared vertices, edges and faces as above.
The fact that these degrees of freedom are unisolvent for the space is proved by~\citers{Ahmad:2013uaa,Cangiani:2016vh,Sutton:2017vc}, for example.
The approximation properties required by Assumption~\ref{item:approximationProperties:interpolationEstimates} are guaranteed by~\cite[Theorem 4.8]{Cangiani:2016vh}, \cite[Theorem 11]{Cangiani:2016ug}, and~\cite[Chapter 3]{Sutton:2017vc}.

To make it explicit that the local boundary value problem defining this space never needs to be solved, the method is constructed with the following notion of \emph{computability}, which is satisfied~\cite{Cangiani:2016vh} by the projectors
$\Proj{\E}{\nabla} v$, $\Proj{\E}{\k} v$,
and $\Proj{\E}{\k-1} \nabla v$ (defined componentwise) for $v \in \VhE$.
We refer to~\citers{Sutton:2017dp,BeiraodaVeiga:2014hza}, for instance, for details on how such a scheme may be implemented in practice.
\begin{definition}[Computability]
  \label{def:vem:computability}
  A term is \emph{computable} if it may be evaluated using only the problem data, the degrees of freedom, and operations on polynomials.
\end{definition}

\subsection{Discrete bilinear forms}\label{sec:vem:bilinearForms}
To discretise the forcing data we take $\brokenProj[\curr] = \Proj{\curr}{\k}$, so that $\forceProj[\curr] = \Proj{\curr}{\k} \force[\curr]$.
Assumption~\ref{item:abstractFramework:forcingTerm} is therefore satisfied because Theorem~\ref{thm:polynomialApproximation} ensures that
\begin{align*}
	(\force[\curr] - \forceProj[\curr], \testfn) = (\force[\curr] - \forceProj[\curr], \testfn - \Proj{\curr}{\k} \testfn) \leq \Cproj \Cequiv \norm{h_{\curr} (\force[\curr] - \forceProj[\curr])} \triplenorm{\testfn}.
\end{align*}
The local virtual element discrete bilinear forms on $\E \in \mesh[\curr]$ are defined as
	\begin{align*}
		\Ah[\curr]_{\E}(\trialfn, \testfn)
		&=
		(\diff  \Proj{\E}{\k-1} \nabla \trialfn,  \Proj{\E}{\k-1} \nabla \testfn)_{\E}
		+
		(\reac  \Proj{\E}{\k} \trialfn,  \Proj{\E}{\k} \testfn)_{\E}
		+
		\SaE(\trialfn - \Proj{\E}{\k} \trialfn,  \testfn - \Proj{\E}{\k} \testfn),
		\\
		\mh[\curr]_{\E}(\trialfn, \testfn) &= (\Proj{\E}{\k} \trialfn,  \Proj{\E}{\k} \testfn)_{\E} + \SmE(\trialfn - \Proj{\E}{\k} \trialfn,  \testfn - \Proj{\E}{\k} \testfn),
	\end{align*}
	 for $\trialfn, \testfn \in \Vh[\E]$, where the \emph{stabilising terms} $\SaE, \SmE : \Vh[\E] \times \Vh[\E] \to \Re$ are given by
	\begin{align*}
		\SaE(\trialfn, \testfn) = \sigma \altStaE(\trialfn, \testfn)
		\quad\text{and}\quad
		\SmE(\trialfn, \testfn) = h_{\E}^{\spacedim} \altStaE(\trialfn, \testfn),
	\end{align*}
	with $\sigma = (\ellipLowerE\ellipUpperE)^{1/2} h_{\E}^{\spacedim - 2} + (\reacLowerE \norm{\reac}_{\infty,\E}  )^{1/2}h_{\E}^{\spacedim}$
	where $\reacLowerE, \ellipLowerE$ and $\ellipUpperE$ denote the local counterparts of $\reacLower, \ellipLower$ and $\ellipUpper$ on $\E$,
	and $\altStaE$ is the Euclidean product between vectors of degrees of freedom.
	The bilinear forms $\mh[\curr]$ and $\Ah[\curr]$ are then given by~\eqref{eq:abstractGlobalBilinearForms}.

The bilinear forms are coercive and continuous in the local discrete (semi-)norms
\begin{align*}
	\triplenorm{\trialfn}_{h, \E}^2 = \Ah[\curr]_{\E}(\trialfn, \trialfn)
	\quad\text{ and }\quad
	\norm{\trialfn}_{h,\E}^2 = \mh[\curr]_{\E}(\trialfn, \trialfn),
\end{align*}
and are \emph{stable} since there exists a constant $\Cstab > 0$ such that, for any $\E \in \mesh[\curr]$,
\begin{align*}
	\Cstab^{-1} \triplenorm{\trialfn}_{\E}^2
	\leq
	\triplenorm{\trialfn}_{h, \E}^2
	\leq
	\Cstab \triplenorm{\trialfn}_{\E}^2
	\quad\text{ and }\quad
	\Cstab^{-1} \norm{\trialfn}_{\E}^2
	\leq
	\norm{\trialfn}_{h,\E}^2
	\leq
	\Cstab \norm{\trialfn}_{\E}^2,
\end{align*}
for all $\trialfn \in \Vh[\E]$.
The bilinear forms also offer \emph{polynomial consistency} such that
\begin{align*}
	\Ah[\curr]_{\E}(q, p) = \A_{\E}(q, p)
	\quad\text{ and }\quad
	\mh[\curr]_{\E}(\trialel, p) = (\trialfn, p)_{\E},
\end{align*}
for all $q, p \in \PE{\k}$.
The polynomial consistency property holds by construction, while the stability property may be proven as in \citers{Cangiani:2016vh,Sutton:2017vc}.
Together, these properties imply the following computable estimate for the inconsistency, which may be proven by arguing as in \citer{Cangiani:2016ug}.

\begin{lemma}[Inconsistency estimate]\label{lem:vem:inconsistencyEstimate}
	Let $\reg \in \{0,1\}$, $\trialcurr \in \Vh[\curr]$ and $\E \in \mesh[\curr]$.
	Suppose that $\diff \in (W^{\reg+1,\infty}(\E))^{\spacedim \times \spacedim}$ and $\reac \in W^{\reg+1, \infty}(\E)$.
	Then, the local virtual element bilinear forms satisfy~\ref{item:approximationProperties:inconsistencyEstimate}
	 with
	\begin{align*}
		(\projErrorLTwo[\E] \trialcurr)^2 = (1 + \Cstab) \norm{ \trialcurr - \Proj{\E}{\k} \trialcurr}_{h, \E}^2,
	\end{align*}
	and
	\begin{align*}
		(\projErrorEnergy[\E] \trialcurr)^2
		=
			C_{\A}
			\Big(&
			\norm{ \big( \Id - \Proj{\E}{\k-1} \big) \diff \Proj{\E}{\k-1} \nabla \trialcurr }_{\E}^2
			+
			\norm{\big( \Id - \Proj{\E}{\k} \big) \reac \Proj{\E}{\k} \trialcurr }_{\E}^2
			\\&
			+
			\triplenorm{\big( \Id - \Proj{\E}{\k-1} \big)\trialcurr}_{h, \E}^2
\Big),
	\end{align*}
	for each $\E \in \mesh[\curr]$, where
	$C_{\A}
	=
	1
	+
	\Cstab
	\Big(\Big(
		\frac{1}{2}
		+
		\frac{1}{2} \Big(\frac{\ellipUpperE}{\ellipLowerE}\Big)^2
		+
		\Big(\frac{1}{\ellipLowerE}\Big)^2
		+
		\Cpf^2
	\Big)\Big)^{1/2}$.
\end{lemma}

\subsection{Computable transfer operators}\label{sec:vem:meshtransfer}
Popular techniques for transferring functions between conventional finite element spaces, such as Lagrangian interpolation or $L^2$ projection, are not appropriate for virtual element functions which cannot be evaluated within elements.
Instead, we introduce \emph{computable transfer operators} (in the sense of Definition~\ref{def:vem:computability}) under the following assumption on the mesh modification.

\begin{assumption}[Coarsening and refinement]\label{ass:coarseandrefine}
Mesh modification is performed via a finite number of coarsening or refinement operations, i.e. such that contiguous patches $P \subset \mesh[\prev]$ of elements are agglomerated into a single element $\E \in \mesh[\curr]$, and individual elements $\E \in \mesh[\prev]$ are refined into patches $P \subset \mesh[\curr]$ of sub-elements such that every side on the boundary of $P$ may be expressed as a subset of a single side of $\E$.
\end{assumption}

In this setting, there are three particularly natural methods of transferring solutions between meshes, detailed below. 
Other techniques for mesh transfer are discussed in~\cite[Chapter 8]{Sutton:2017vc}.

\subsubsection{Local transfer operator}
We introduce local transfer operators which may be applied to either coarsened patches or refined elements, denoted  $\coarsening$ and $\refinement$ respectively.
A computable and inherently local transfer operator $\transfer[\curr] : \Vh[\prev] \to \Vh[\curr]$ may then be constructed by applying either $\coarsening$ or $\refinement$ on each element as required.

\begin{definition}[Local transfer operators]\label{def:ops}
	Let $P$ be a patch of elements and let $\E = \bigcup_{K \in P} K \subset\Re^{\spacedim}$, $\spacedim = 2,3$.
	Define
	\begin{align*}
		\Vh[P] = \Big\{ \vh \in H^1( E ) : \vh |_{K} \in \Vh[K] \text{ for each } K \in P \Big\}.
	\end{align*}

	The \emph{coarsening operator} $\coarsening : \Vh[P] \to \Vh[\E]$
	is defined to be the Lagrange interpolant.

	The \emph{refinement operator} $\refinement : \Vh[\E] \to \Vh[P]$  is defined for $\spacedim=2$ to satisfy
	\begin{align}
	\label{eq:projLocalproblem}
		\Ah[\curr]_{P}(\refinement \trialel, \testpatch)
		&=
		\Ah[\curr]_{\E}(\trialel, \coarsening \testpatch) \text{ for all } \testpatch\in \Vh[P] \cap H^1_0(P),
		\\
		\label{eq:projBoundary}
		(\refinement \trialel)|_{\partial P}
		&=
		\trialel|_{\partial P},
	\end{align}
	for any $\trialel \in \Vh[\E]$, where $\Ah[\curr]_{P} : \Vh[P] \times \Vh[P] \to \Re$ is the elliptic virtual element bilinear form on $P$ given by
	$\Ah[\curr]_{P}(\cdot, \cdot) := \sum_{K \in P} \Ah[\curr]_{K}(\cdot, \cdot)$.
	For $d=3$ the refinement operator is recursively defined as the $\spacedim=2$ construction on each face and to satisfy~\eqref{eq:projLocalproblem} in $\E$.
\end{definition}

Both operators $\coarsening$ and $\refinement$ are computable in the sense of Definition~\ref{def:vem:computability}.
For $\spacedim = 2$, the coarsening operator $\coarsening$ is computable because $\E = \cup_{K \in P} K$ and the boundaries of $\E$ and $P$ coincide; the extension to $\spacedim=3$ is analogous.
The computability of the refinement operator $\refinement$ then follows from that of $\Ah[\curr]_{P}$ and $\coarsening$.
These local transfer operators also preserve polynomials in the sense that $p = \refinement p = \coarsening p$ for $p \in \PE{\k}$, a property which is important to retain the approximation power of the solution under refinement.
Moreover, when $\spacedim = 2$, $\refinement$ does not change the edge values of virtual element functions by construction, although this is not true when $\spacedim=3$ due to the virtual nature of the face spaces.
The patch on which $\refinement$ modifies function values when $\spacedim = 3$ therefore includes the face neighbours of the modified element.
For either $\spacedim=2$ or 3, $\coarsening$ modifies virtual element functions on the neighbours of the coarsened patch unless the sides of $\E$ coincide with those of $P$.

\subsubsection{Polynomial projection}
Let $\mesh[\curr,\prev]$ denote the finest common coarsening of $\mesh[\curr]$ and $\mesh[\prev]$, such that every element of either mesh is a subset of an element of $\mesh[\curr,\prev]$, and let $\P{\curr,\prev}{\k}$ denote the space of polynomials with respect to this mesh.
Then, for any $\vh \in \Vh[\curr]$ or $\Vh[\prev]$, the $L^2$-orthogonal projector $\Proj{\curr,\prev}{\k} : L^2(\domain) \to \P{\curr,\prev}{\k}$ defined by
\begin{align}\label{eq:commonCoarseningProj}
	(\Proj{\curr,\prev}{\k} \vh - \vh, p) = 0 \qquad \text{ for all } p \in \P{\curr,\prev}{\k},
\end{align}
is computable directly from the degrees of freedom of $\vh$.
One may therefore adopt the numerical scheme: given $U^0 \in \Vh[0]$ approximating $u_0$, for each $\curr = 1,\ldots,N$ find $\ucurr \in \Vh[\curr]$ satisfying
\begin{align*}
	\frac{1}{\timestep^{\curr}} \Big( \mh[\curr] (\ucurr, \testcurr ) - (\Proj{\curr,\prev}{\k} \uprev, \testcurr) \Big) + \Ah[\curr](\ucurr, \testcurr) = \mh[\curr](\forceh[\curr], \testcurr)
	\quad \text{ for all }
	\testcurr \in \Vh[\curr],
\end{align*}
which may be expressed in the form of~\eqref{eq:fullyDiscreteScheme} by taking $\transfer[\curr] = \LTwoRecon[\curr] \Proj{\curr,\prev}{\k}$, and our analysis therefore also applies to this scheme.
We note, however, that this operator does not reduce to the identity operator when $\Vh[\curr] = \Vh[\prev]$, meaning that the components of the estimators measuring mesh transfer error will be non-zero even when the mesh is not modified.

\subsubsection{Elliptic transfer operator}
A counterpart of the elliptic transfer operator $\canonicalTransfer[\curr]$ of Definition~\ref{def:ellipticTransfer} can be constructed which is computable in the virtual element context. 
This is defined as the operator $\canonicalTransferVEM[\curr] : \Vh[\prev] \to \Vh[\curr]$
satisfying
\begin{align*}
	\Ah[\curr] (\canonicalTransferVEM[\curr] \trialprev, \testcurr)
	=
	-
	\big(
		\Proj{\curr,\prev}{\k} (\diffoph[\prev] \trialprev
		+
		(\LTwoRecon[\prev] - \identity) \brokenProj[\prev] \force[\prev]
		-
		(\LTwoRecon[\curr] - \identity) \brokenProj[\curr] \force[\prev] \big), \testcurr
	\big)
\end{align*}
for all $\testcurr \in \Vh[\curr]$,
with $\Proj{\curr,\prev}{\k}$ defined as in~\eqref{eq:commonCoarseningProj}.
This operator is computable because the right-hand side is just a product of (computable) polynomial projections.
We note, however, that this operator no longer reduces to the identity operator when $\Vh[\curr] = \Vh[\prev]$ and the fundamental relation of Definition~\ref{def:ellipticTransfer} is also no longer exactly satisfied.
Instead, we have
\begin{align}\label{eq:canonicalTransferVEM:property}
	&\A(\ellipReconLag{\curr}{\prev}{\canonicalTransferVEM[\curr] \trialprev}, \testcurr)
	=
	\A(\ellipReconLag{\prev}{\prev}{\trialprev}, \testcurr)
	+
	\inconsistencym[\curr](\diffoph[\curr] \canonicalTransferVEM[\curr] \trialprev + \LTwoRecon[\curr]\brokenProj[\curr] \force[\prev], \trialcurr )
	\\&\qquad\qquad
	+
	((\Id - \widehat{\Proj{}{}}^{\curr}_{\k}) \big( \diffoph[\prev] \trialprev
	+
	(\LTwoRecon[\prev] - \identity) \brokenProj[\prev] \force[\prev]
	-
	(\LTwoRecon[\curr] - \identity) \brokenProj[\curr] \force[\prev] \big), \testcurr ),
	\notag
\end{align}
which may be interpreted as providing the same property up to higher order terms.

\subsection{Computable error estimates}\label{sec:vem:ellipticEstimates}

Due to the abstract discrete setting in which they are developed, the error estimates of Section~\ref{sec:abstractFramework} remain valid in the virtual element context.
They cannot all be applied directly, however, because they are not computable in the sense of Definition~\ref{def:vem:computability}.
To estimate the terms in Definition~\ref{def:estimatorterms} computably, we bound $\timeest(t)$ by
\begin{align*}
	\timeest(t)
	&\leq
	\norm{(\Proj{\curr}{\k} \timederivh[\curr] \ucurr - \forceProj[\curr]) - (\Proj{\prev}{\k} \timederivh[\prev] \uprev - \forceProj[\prev])}
	\\&\qquad
	+
	\Cstab
	\Big(
		\norm{(\identity - \Proj{\curr}{\k}) \timederivh[\curr] \ucurr}_h
		+
		\norm{(\identity - \Proj{\prev}{\k}) \timederivh[\prev] \uprev}_h
	\Big),
\end{align*}
and $\meshchangeest(t)$ by
\begin{align*}
	\meshchangeest(t)
	\leq &
	\frac{1}{\timestep^{\curr}}
	\norm{\Proj{\prev}{\k}\uprev - \Proj{\curr}{\k} \transfer[\curr] \uprev}
	\\&+
	\frac{\Cstab}{\timestep^{\curr}}
	\Big(
		\norm{(\identity - \Proj{\curr}{\k}) 
			\transfer[\curr] \uprev
			}_h
		+
		\norm{(\identity - \Proj{\prev}{\k}) \uprev}_h
	\Big).
\end{align*}
By construction, the data estimators 
\begin{align*}
	\dataest{T}(t) = \norm{\force(t) - \force[\curr]}
	\quad\text{ and }\quad
	\dataest{S}(t) = \Cdata \norm{h_{\curr} (\force[\curr] - \forceProj[\curr])},
\end{align*}
remain computable.
The remaining non-computable terms are the elliptic estimators of Section~\ref{sec:abstract:ellipticEstimators}, for which we prove slightly modified versions in Lemma~\ref{lem:vem:ellipticEstimates}, to produce variants of Lemmas~\ref{lem:residualEstimates:singleMesh}
and~\ref{lem:residualEstimates:twoMesh:local} involving computable projected forms of the residual operators.
The single mesh elliptic error estimates we obtain are similar to those of~\citer{Cangiani:2016ug}.

\begin{lemma}[Elliptic error estimates for the virtual element method]\label{lem:vem:ellipticEstimates}
	For $m,\curr \in \{ 0,\dots,N \}$ and $\trialcurr \in \Vh[\curr]$,
	let $\jumpresidualvirt[\curr] : \Vh[\curr] \to L^2(\sides[\curr])$ denote the \emph{projected jump residual operator}
	\begin{align*}
		(\jumpresidualvirt[\curr] \trialcurr)|_{\side} = \jump{\diff \nabla \Proj{\curr}{\k} \trialcurr}_{\side},
	\end{align*}
	for each $\side \in \sides[\curr]$,
	which is extended by zero to the whole of $\domain$, and let $\elresidualvirtLag{\curr}{m}{} : \Vh[\curr] \to \Re$ denote the \emph{projected element residual operator}, given for each $\E \in \mesh[\curr]$ by
	\begin{align*}
		(\elresidualvirtLag{\curr}{m}{\trialcurr})|_{\E}
		&=
		\diffop \Proj{\E}{\k} ( \trialcurr |_{\E}) - \Proj{\E}{\k} (\ellipReconRHSLag{\curr}{m}{\trialcurr})|_{\E}
		\\
		&=
		\diffop \Proj{\E}{\k} ( \trialcurr |_{\E}) - \Proj{\E}{\k} (\diffoph[\curr] \trialcurr - (\LTwoRecon[\curr] - \identity) \Proj{\curr}{\k} \force[m])|_{\E}.
	\end{align*}
	Then, the estimates of Lemmas~\ref{lem:residualEstimates:singleMesh} and~\ref{lem:residualEstimates:twoMesh:local} hold with $\elresidualvirtLag{\curr}{\curr}{}$ and $\jumpresidualvirt[\curr]$ replacing $\elresidualLag{\curr}{\curr}{}$
	and $\jumpresidual[\curr]$ respectively, i.e.
	\begin{align*}
		\ellipest{L^2}^{\curr}(\trialcurr,
		 \force[\curr])
		&=
		\Cellip
		\big(
			\norm{h_{\curr}^2 \elresidualvirtLag{\curr}{\curr}{\trialcurr}}^2
			+
			\norm{h_{\curr}^{3/2} \jumpresidualvirt[\curr] \trialcurr }_{\sides[\curr]}^2
			+
			\inconGroup[\curr](\trialcurr, \force[\curr])^2
		\big)^{1/2},
		\\
		\ellipest{H^1}^{\curr}(\trialcurr,
		\force[\curr])
		&=
		\Cellip
		\big(
			\norm{h_{\curr} \elresidualvirtLag{\curr}{\curr}{\trialcurr}}^2
			+
			\norm{h_{\curr}^{1/2} \jumpresidualvirt[\curr] \trialcurr }_{\sides[\curr]}^2
			+
			\inconGroupEnergy[\curr](\trialcurr, \force[\curr])^2
		\big)^{1/2},
	\end{align*}
	and
	\begin{align*}
		\ellipest{L^2}^{\partial_t}(\trialcurr, \trialprev,
		 \force[\curr], \force[\prev])
		&=
		\frac{1}{\timestep^{\curr}}
		\Big(
		\norm{\transfer[\curr] \trialprev - \trialprev}^2
		\\&\qquad\quad
		+\Cellip
		\Big(
			\norm{h_{\curr}^2 (\elresidualvirtLag{\curr}{\curr}{\trialcurr} - \elresidualvirtLag{\prev}{\prev}{\trialprev})}^2
			 \notag
		\\&\qquad\qquad\qquad\quad
			+
			\norm{h_{\curr}^{3/2} (\jumpresidualvirt[\curr] \trialcurr - \jumpresidualvirt[\prev] \trialprev)}_{\sides[\curr] \cup \sides[\prev]}^2
		\notag
		\\&\qquad\qquad\qquad\quad
			+
			\norm{\hat{h}_{\prev, \curr}^2 \elresidualvirtLag{\prev}{\prev}{\trialprev}}_{\mesh[\prev] \setminus \mesh[\curr]}^2
		\notag
		\\&\qquad\qquad\qquad\quad
			+
			\norm{\hat{h}_{\prev, \curr}^{3/2} \jumpresidualvirt[\prev] \trialprev}_{\sides[\prev] \setminus \sides[\curr]}^2
		\notag
		\\&\qquad\qquad\qquad\quad
		+
    	\inconGroup[\curr](\trialcurr, \force[\curr])^2
	    +
	    \inconGroup[\prev](\trialprev, \force[\prev])^2
		\Big)\Big)^{1/2}.
		\notag
	\end{align*}
\end{lemma}
\begin{proof}
	The definition of the elliptic reconstruction and Lemma~\ref{lem:ellipreconInconsistency} give
	\begin{align*}
		\A(\trialcurr - \ellipReconLag{\curr}{\curr}{\trialcurr}, \testfn)
		&
		=
		(\diffoph[\curr] \trialcurr - \forceh[\curr] + \forceProj[\curr], \testfn - \testcurr)
		+
		\A(\trialcurr, \testfn - \testcurr)
		\\&\qquad
		+
		\inconsistencya[\curr](\trialcurr, \testcurr)
		+
		\inconsistencym[\curr](\diffoph[\curr] \trialcurr + \forceh[\curr], \testcurr),
	\end{align*}
	where $\testcurr \in \Vh[\curr]$ denotes the quasi-interpolant of $\testfn$ satisfying~\eqref{eq:quasiInterpolation}.
	Introducing projectors and using the fact that $\forceProj[\curr] = \Proj{\curr}{\k} \force[\curr]$, this becomes
	\begin{align*}
		\A(\trialcurr - \ellipReconLag{\curr}{\curr}{\trialcurr}, \testfn)
		&
		=
		( \Proj{\curr}{\k} (\diffoph[\curr] \trialcurr + \forceh[\curr] - \force[\curr]), \testfn - \testcurr)
		+
		\A( \Proj{\curr}{\k} \trialcurr, \testfn - \testcurr)
		\\&\quad
		+
		( (\Id - \Proj{\curr}{\k}) (\diffoph[\curr] \trialcurr + \forceh[\curr]), \testfn - \testcurr)
		+
		\A( \trialcurr - \Proj{\curr}{\k} \trialcurr, \testfn - \testcurr)
		\\&\quad
		+
		\inconsistencya[\curr](\trialcurr, \testcurr)
		+
		\inconsistencym[\curr](\diffoph[\curr] \trialcurr + \forceh[\curr], \testcurr).
	\end{align*}
	Using Lemma~\ref{lem:vem:inconsistencyEstimate} and the space approximation bounds~\eqref{eq:quasiInterpolation}, the final four terms of this may be bounded by
	\begin{align*}
		\widehat{C}
		\inconGroup[\curr](\trialcurr, \force[\curr])
		\abs{\testfn}_1,
	\end{align*}
	where $\inconGroup[\curr]$ is given in Lemma~\ref{lem:residualEstimates:singleMesh}
	and
	$\widehat{C} = \max \big\{
		\alpha \frac{\Cclem\Cstab}{1 + \Cstab} + \Cglob,
		\alpha \frac{\Cstab \Cequiv \Cclem}{C_{\A}} + \Cglob
	\big\}$ with
	$C_{\A}$ from Lemma~\ref{lem:vem:inconsistencyEstimate} and $\alpha$ depending only on the mesh regularity.
	Integration by parts therefore gives
	\begin{align*}
		\A(\ellipReconLag{\curr}{\curr}{\trialcurr} - \trialcurr, \testfn)
		&
		\leq\!
		\sum_{\E \in \mesh[\curr]}
			(\elresidualvirtLag{\curr}{\curr}{\trialcurr}, \testfn - \testcurr)_{\E}
		+\!
		\sum_{\side \in \sides[\curr]}
			(\jumpresidualvirt[\curr], \testfn - \testcurr)_{\side}
		+
		\widehat{C}
		\inconGroup[\curr](\trialcurr, \force[\curr])
		\abs{\testfn}_1,
	\end{align*}
	and the energy estimate follows by applying the Cauchy-Schwarz inequality and the scaled trace inequality, using~\eqref{eq:quasiInterpolation}, and selecting $\testfn = \ellipReconLag{\curr}{\curr}{\trialcurr} - \trialcurr$.

	The $L^2$ norm estimate follows via a duality argument using similar arguments.

	A computable estimate for the elliptic reconstruction time derivative error may be proven by combining the arguments proving Lemma~\ref{lem:residualEstimates:twoMesh:local} with those above.
\end{proof}

Although we do not pursue it here, the global mesh modification estimate of Lemma~\ref{lem:residualEstimates:twoMesh:global} can also be translated into this context using the above counterpart $\canonicalTransferVEM[\curr]$ of the elliptic transfer operator $\canonicalTransfer[\curr]$.
The resulting estimate is analogous to Lemma~\ref{lem:residualEstimates:twoMesh:global} but with projected residuals and the extra inconsistency terms appearing in~\eqref{eq:canonicalTransferVEM:property}.

\subsection{Numerical experiments}\label{sec:vem:numerics}

We now demonstrate the practical performance of the virtual element error estimates presented in the previous sections on a challenging set of numerical experiments.

\subsubsection{Convergence tests}
We begin by exploring the convergence properties of the estimates above when applied to the virtual element discretisation of the model parabolic problem~\eqref{eq:modelHeatEquation} with $\Omega = [0,1]^2$, $\alpha = 1$, and data fixed in accordance with the exact solution
\begin{align}\label{eq:vem:solution:oscillating}
	u(x,y,t) = \sin(5 \pi t) \sin(\pi x) \sin(\pi y),
\end{align}
which we refer to as the \emph{oscillating solution}.
The simulations, indexed by $i \in \mathbb{Z}$, use a fixed spatial mesh of $2^{2i}$ square elements with diameter $h_i = 2^{1/2-2i}$ linked to the time-step size $\timestep_i$.
We plot the $L^2(0, t; H^1(\domain))$ and $L^{\infty}(0,t; L^2(\domain))$ errors and estimators alongside the separate estimator components; see Section~\ref{sec:fem} for a full description of the plotted quantities.
The results for $\timestep_i = h_i$ with $i \in \{2,3,4,5,6,7,8\}$ and $\timestep_i = h_i^2$ with $i \in \{2,3,4,5,6\}$ are plotted in Figures~\ref{fig:vem:numerics:oscillating:1} and \ref{fig:vem:numerics:oscillating:2} respectively.

\begin{figure}%
  \centering{%
  \subcaptionbox{Computed and estimated errors}[\textwidth]{\includegraphics[width=0.9\textwidth]{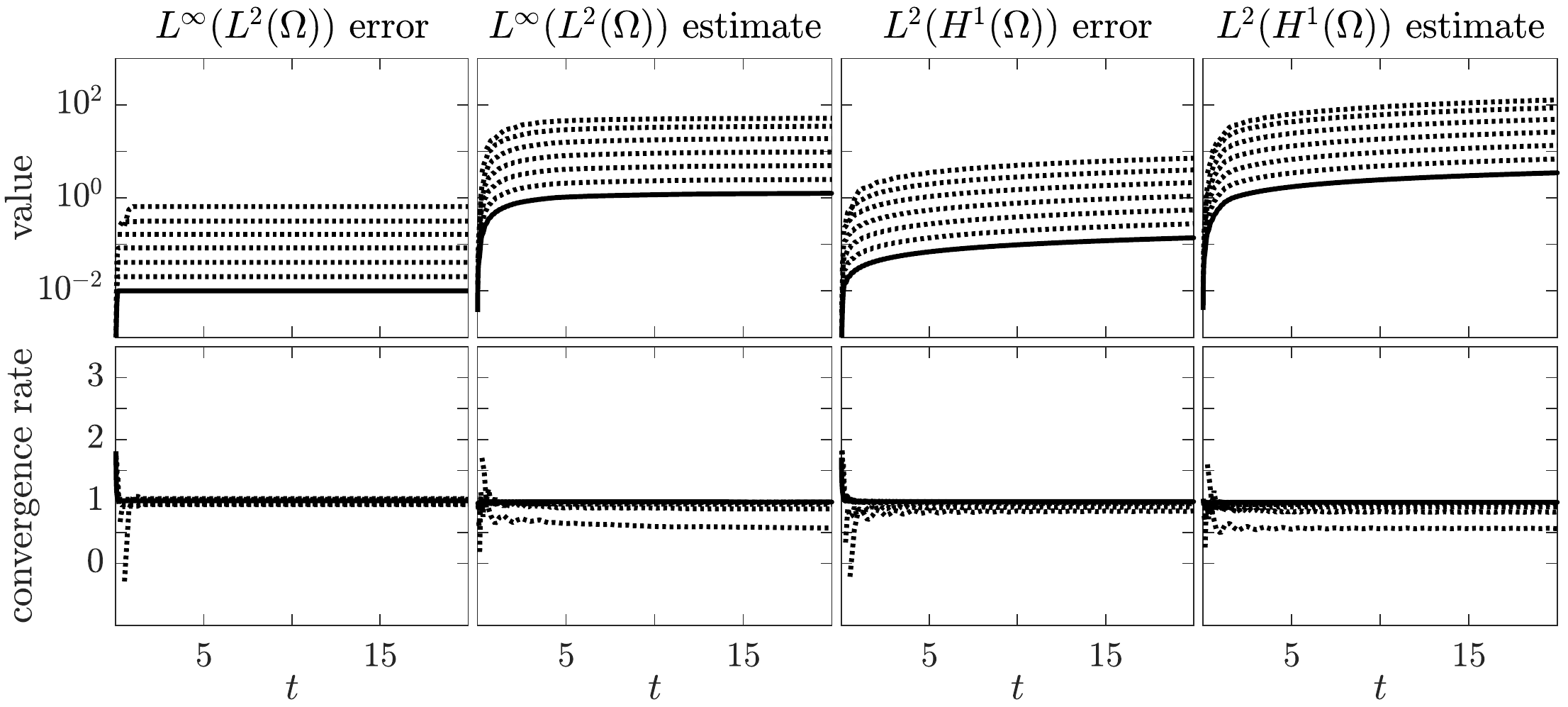}\vspace{-0.75em}}
  \vspace{1em}

  \subcaptionbox{{Components of the $L^2(0,t; H^1(\domain))$ estimator}}[\textwidth]{\includegraphics[width=0.9\textwidth]{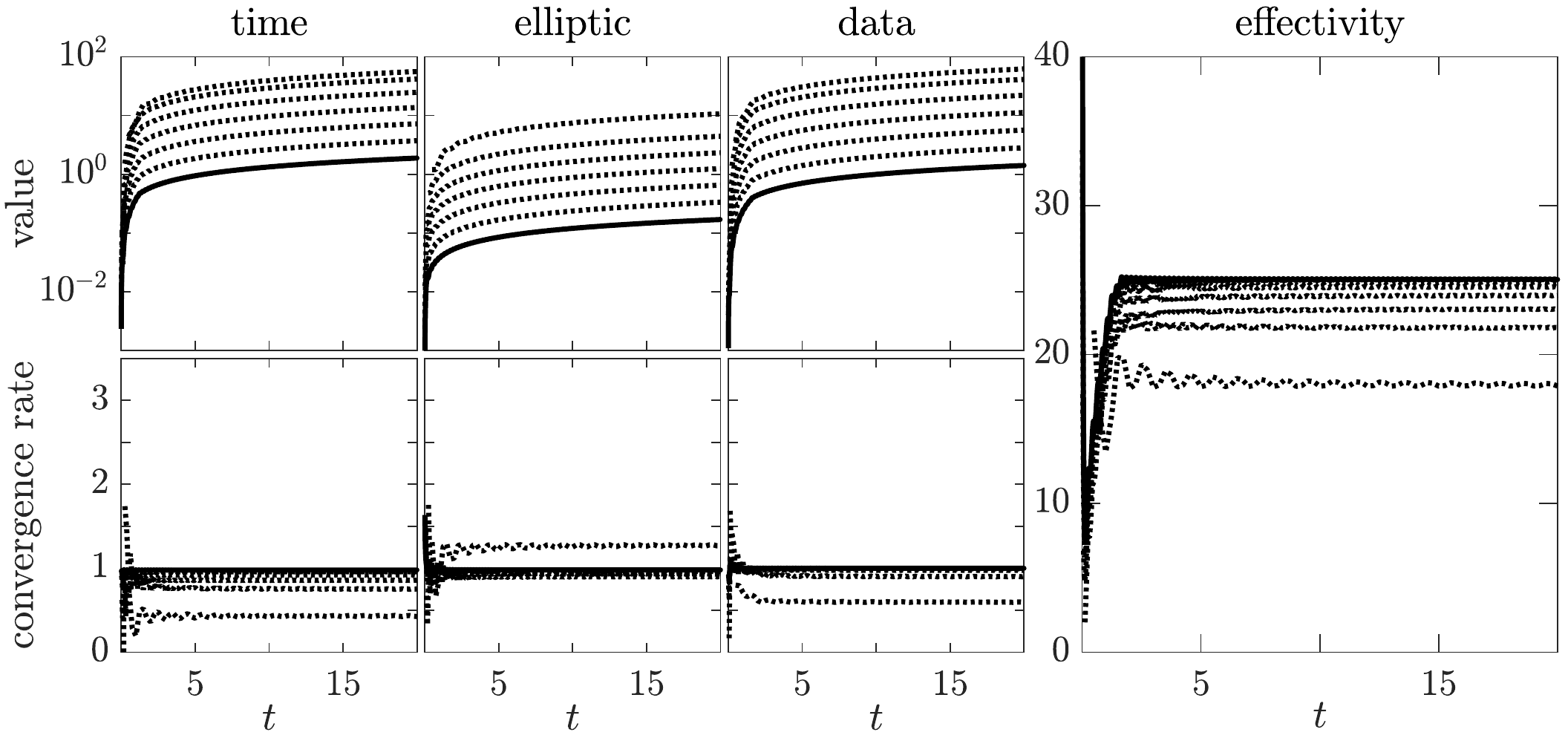}\vspace{-0.75em}}
  \vspace{1em}

  \subcaptionbox{{Components of the $L^{\infty}(0,t; L^2(\domain))$ estimator}}[\textwidth]{\includegraphics[width=0.9\textwidth]{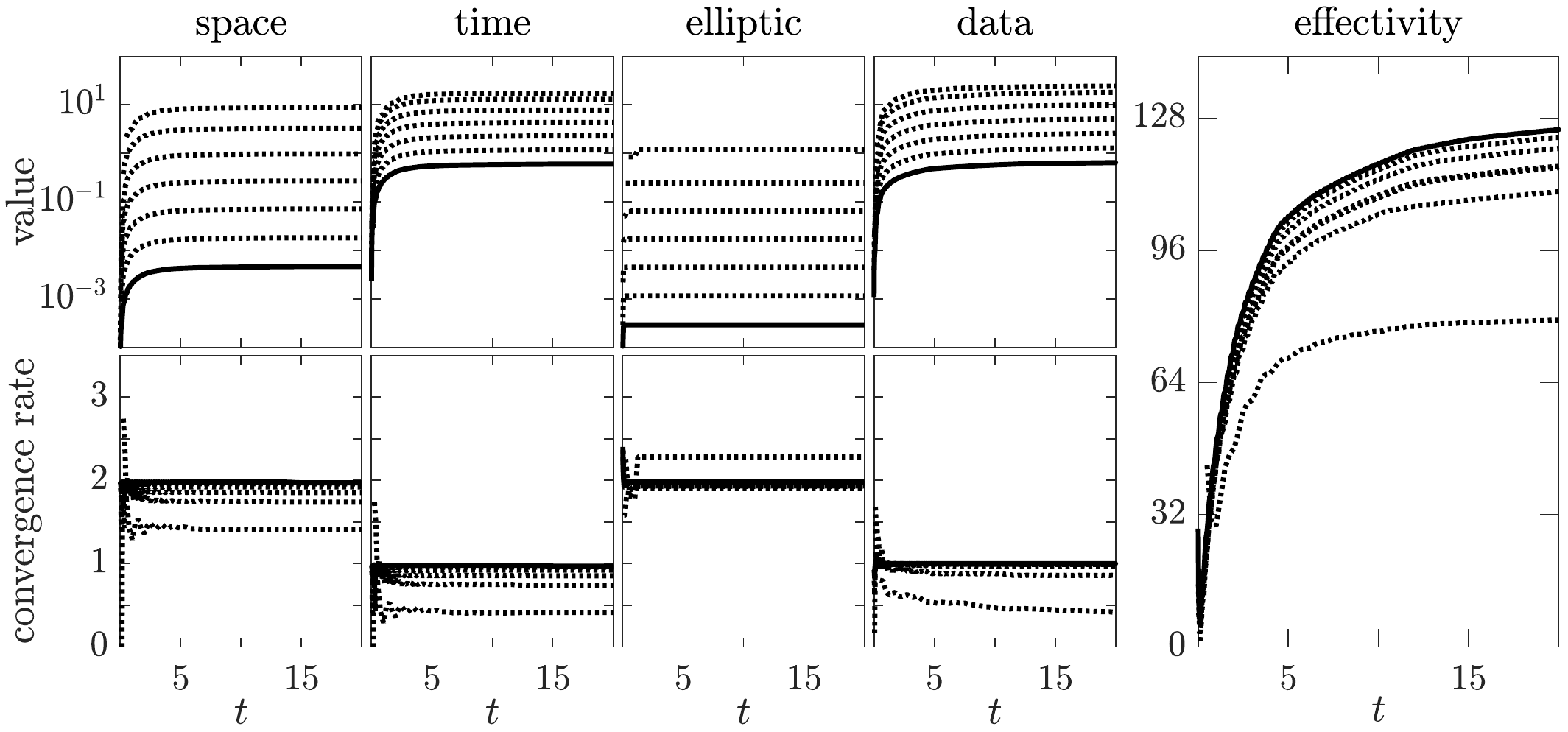}\vspace{-0.75em}}
  }%
  \caption{Behaviour of the error and estimator for the oscillating solution~\eqref{eq:vem:solution:oscillating} with $\tau \approx h$. Solid lines indicate results on the finest mesh.}
  \label{fig:vem:numerics:oscillating:1}
\end{figure}

\begin{figure}%
  \centering{%
  \subcaptionbox{Computed and estimated errors}[\textwidth]{\includegraphics[width=0.9\textwidth]{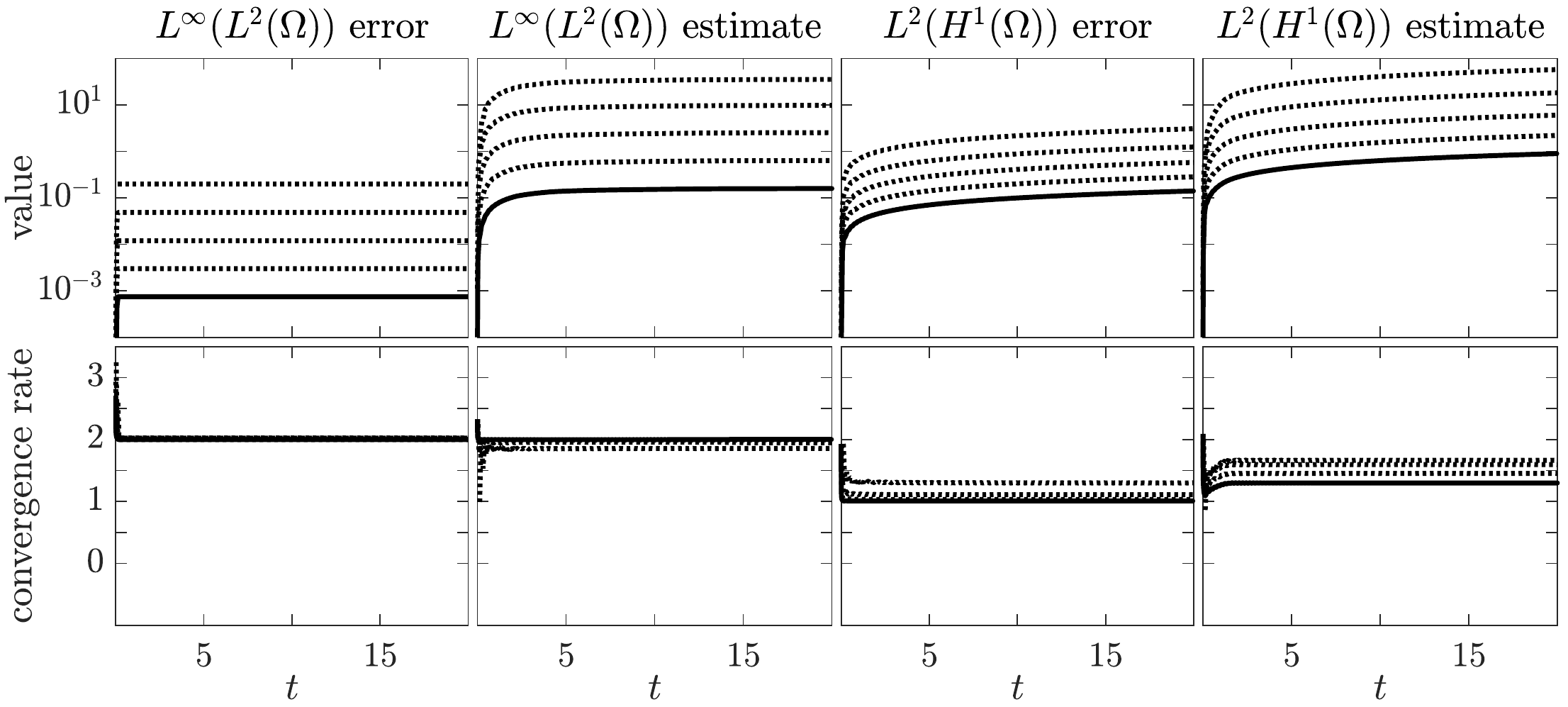}\vspace{-0.75em}}
  \vspace{1em}

  \subcaptionbox{{Components of the $L^2(0,t; H^1(\domain))$ estimator}}[\textwidth]{\includegraphics[width=0.9\textwidth]{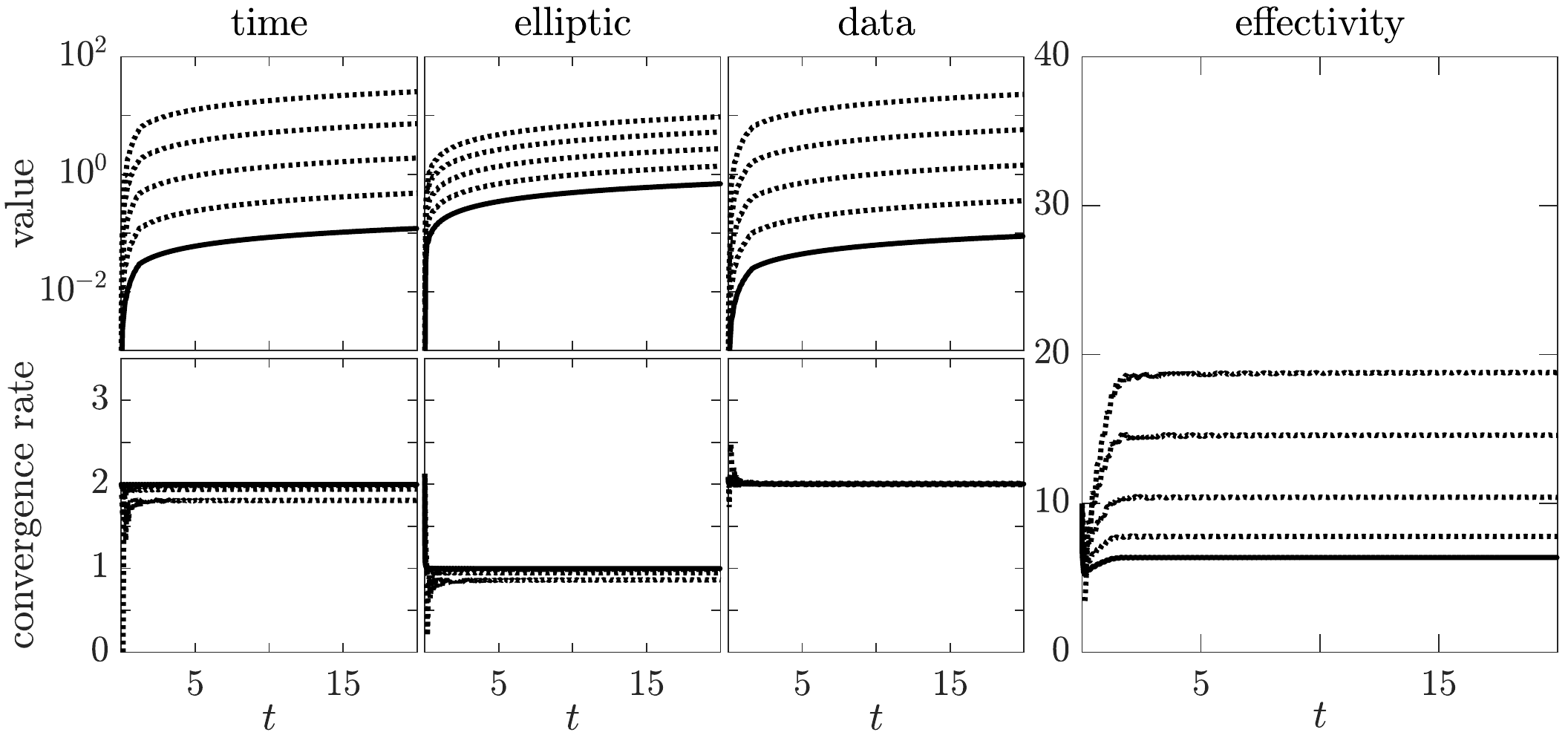}\vspace{-0.75em}}
  \vspace{1em}

  \subcaptionbox{{Components of the $L^{\infty}(0,t; L^2(\domain))$ estimator}}[\textwidth]{\includegraphics[width=0.9\textwidth]{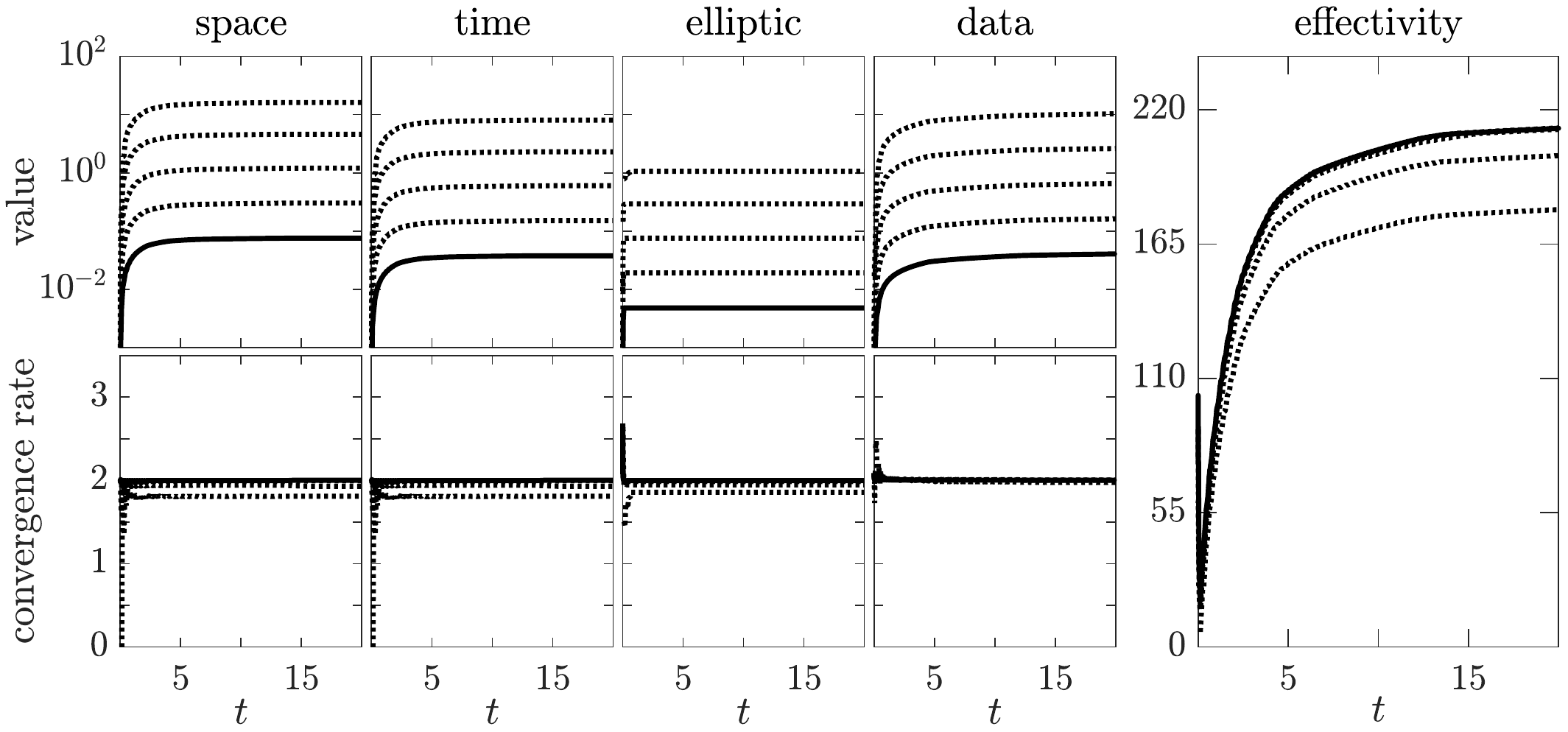}\vspace{-0.75em}}
  }%
  \caption{Behaviour of the error and estimator for the oscillating solution~\eqref{eq:vem:solution:oscillating} with $\tau \approx h^2$. Solid lines indicate results on the finest mesh.}
  \label{fig:vem:numerics:oscillating:2}
\end{figure}

The optimal order of convergence of both errors and estimators is observed in all cases, save for the slightly pre-asymptotic behaviour of the $L^2(H^1)$ estimator when $\timestep = h^2$ (Figure~\ref{fig:vem:numerics:oscillating:2}), which appears to be converging slightly faster than expected on the coarser meshes but approaches the optimal rate of 1 on the finer meshes.
This appears to be due to a slight imbalance between the (second order) time estimator, which initially dominates the estimate, and the (first order) $H^1(\domain)$ elliptic estimator.
A result of this is that the effectivity index in this case decreases as the mesh gets finer, although it may be expected to stabilise once the time estimator has decayed sufficiently and the expected asymptotic first order convergence rate, dictated by the elliptic estimator, is attained.
No similar pre-asymptotic behaviour is observed in the case of $\timestep = h$ (Figure~\ref{fig:vem:numerics:oscillating:1}), where the effectivity quickly stabilises at approximately 25, a typical value comparable with other similar estimates.

The $L^{\infty}(L^2)$ error and estimator both exhibit stellar performance, converging at the theoretically optimal first order when $\timestep = h$ and second order when $\timestep = h^2$, with effectivities of approximately 100 in the former case and 200 in the latter case.
Such effectivity values are typical of comparable $L^{\infty}(L^2)$ estimates, and we observe that sharp changes in their time derivatives are visible where different accumulation norms become optimal and they begin to approach a limiting value towards the end of the time interval (due to the use of $L^{\infty}$ time accumulations).

Finally, we observe that the behaviour of the components of both estimators apparently justifies their names: the space and time estimators converge at the expected rate for the spatial and temporal discretisations, respectively.

\subsubsection{Adaptive experiments}

We now present adaptive experiments with two different benchmark solutions.
The first we refer to as the \emph{layer solution}, defined as
\begin{align}\label{eq:vem:numerics:layer}
	u(x,y,t) = \big(1 + \exp(10 (x + y - t)) \big)^{-1},
\end{align}
which presents the challenge of an internal layer parallel to the line $y=-x$ which moves across the domain along $y=x$.
The second is the \emph{circulating solution}, with
\begin{align}\label{eq:vem:numerics:stirring}
    u(x,y,t) = (10-t)(x^2-x)(y^2-y) \Big( 1 - \frac{1}{1 + e^{\alpha(x,y,t)}} \Big),
\end{align}
where
\begin{align*}
    \alpha(x,y,t) = 25(10-t) \Big( \Big( 2x - \frac{1}{2} \sin \Big( \frac{\pi}{2} t \Big) - 1 \Big)^2 +  \Big( 2y - \frac{1}{2} \cos \Big( \frac{\pi}{2} t \Big) - 1 \Big)^2 - \frac{3}{200} \Big),
\end{align*}
and features a lump of mass circulating the domain whilst slowly diffusing away.

Both simulations begin with an initial mesh of 400 square elements and incorporate a simple spatial adaptive algorithm with a fixed time-step.
The elemental components of the $L^2(\domain)$ elliptic estimator of Lemma~\ref{lem:vem:ellipticEstimates} are used as as an error indicator: elements on which this quantity is above a certain threshold on every $5^{\text{th}}$ time-step are marked for refinement, while those below a lower threshold on every $10^{\text{th}}$ time-step are marked for coarsening.
Coarsening is achieved by simply merging patches of neighbouring marked elements, a process which naturally produces polygonal elements.
Such elements may then be refined at a later stage by splitting them back into the elements from which they were formed.
The numerical solution is then transferred onto this new mesh using the local transfer operator from Definition~\ref{def:ops}.

A key advantage of using these polygonal meshes is evident in that they enable aggressive coarsening to be performed.
Samples of the meshes produced at various time-steps when solving the  layer and circulating examples are shown in Figures~\ref{fig:vem:numerics:layer:meshes} and~\ref{fig:vem:numerics:stirring:meshes}, respectively.
The adapted meshes are highly suited to both solutions, with resolution focussed around the layers and sparsely deployed elsewhere.

\begin{figure}
	\centering{%
	\includegraphics[width=0.3\linewidth]{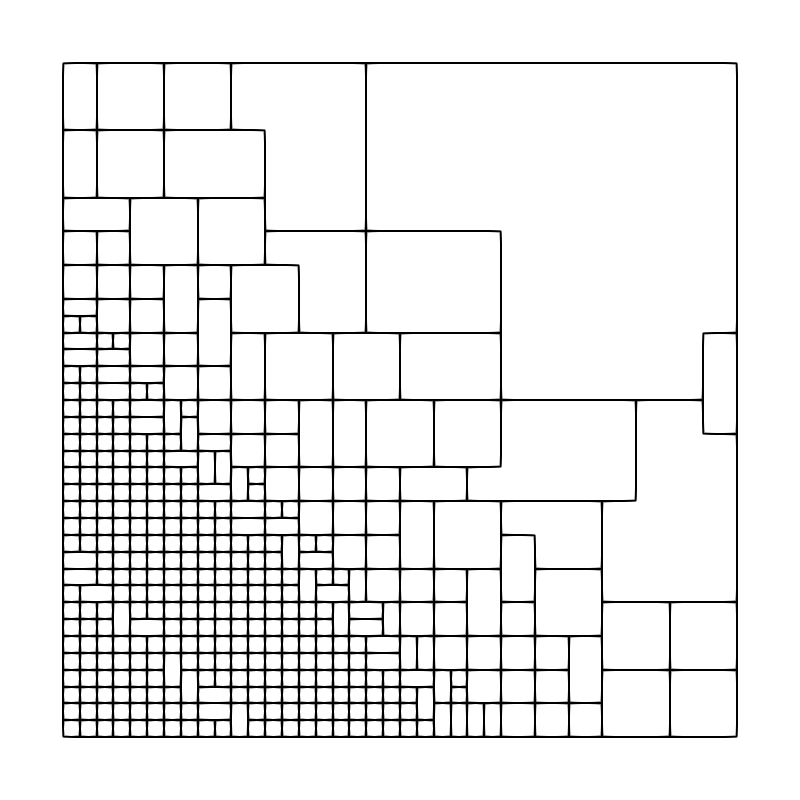}\quad%
	\includegraphics[width=0.3\linewidth]{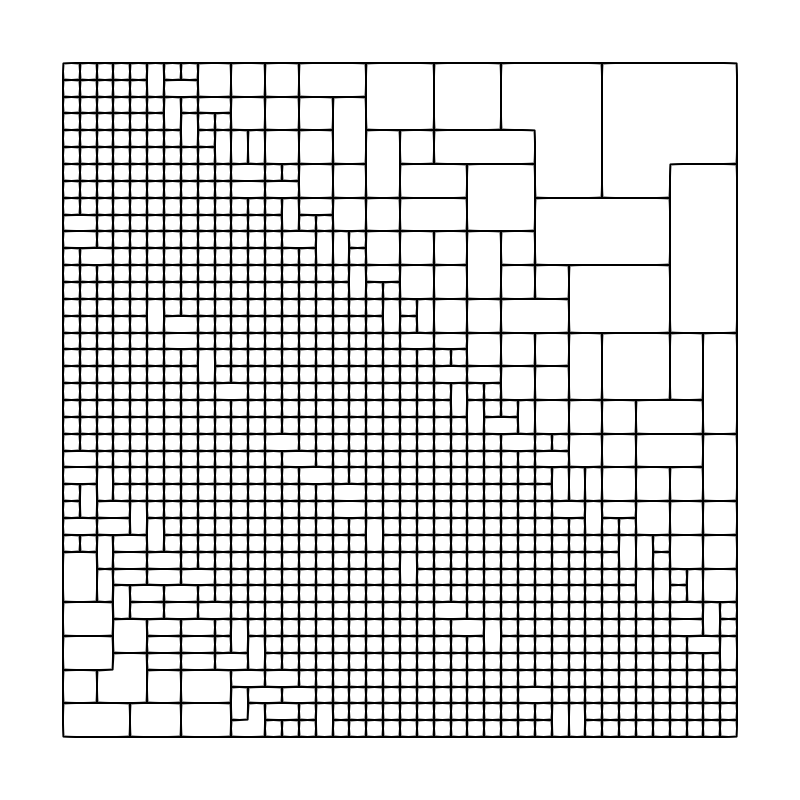}\quad%
	\includegraphics[width=0.3\linewidth]{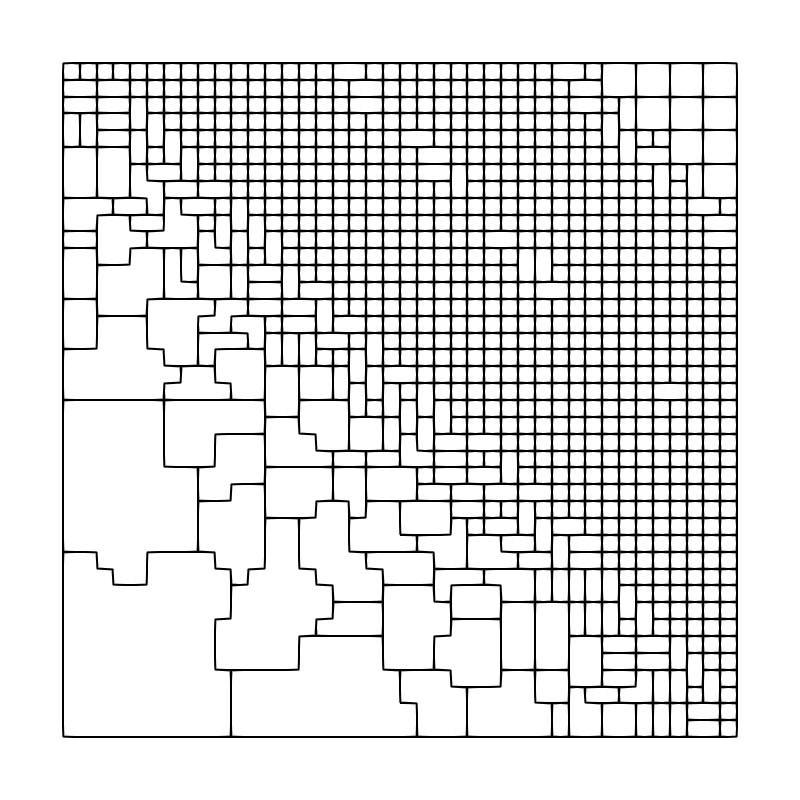}%
	}
	\caption{Examples of meshes adaptively generated to fit the layer solution~\eqref{eq:vem:numerics:layer} at various time moments. The path of the layer (oriented parallel to $y=-x$, travelling parallel to $y=x$) is clearly marked by the regions of refinement.}
	\label{fig:vem:numerics:layer:meshes}
\end{figure}

\begin{figure}
	\centering{%
	\includegraphics[width=0.3\linewidth]{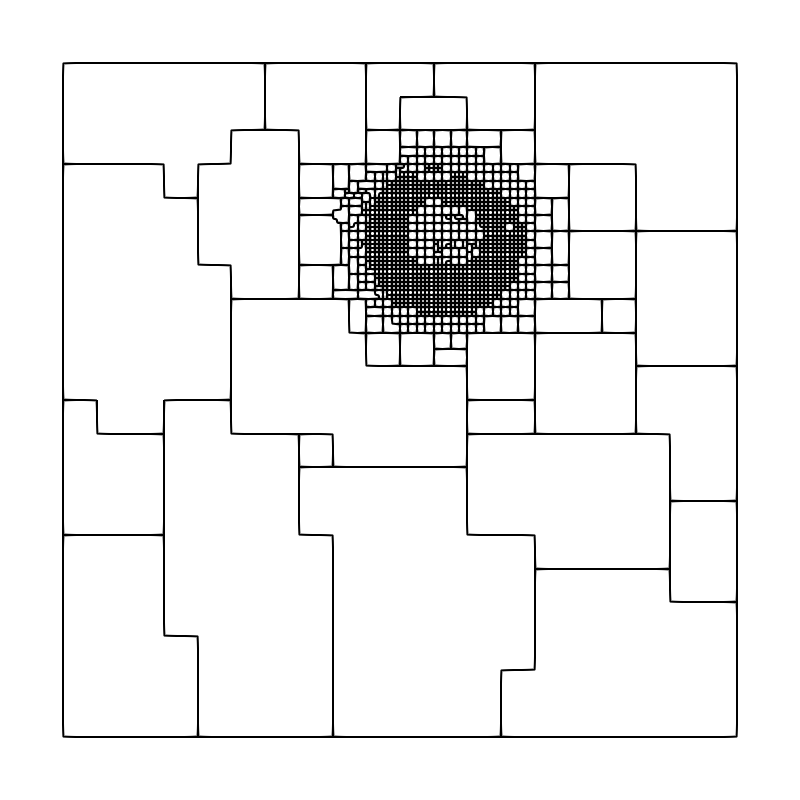}\quad%
	\includegraphics[width=0.3\linewidth]{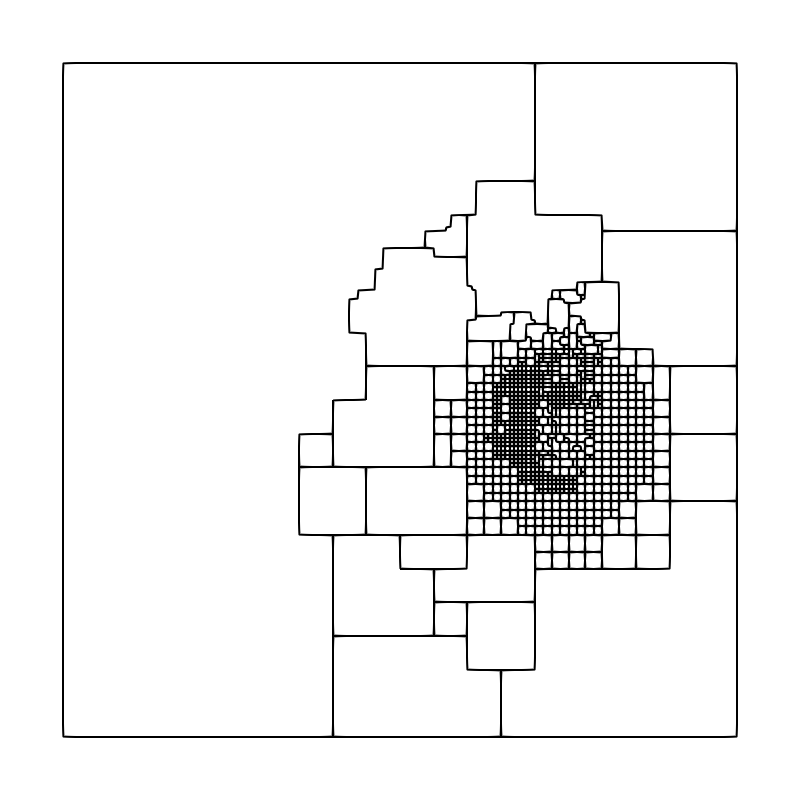}\quad%
	\includegraphics[width=0.3\linewidth]{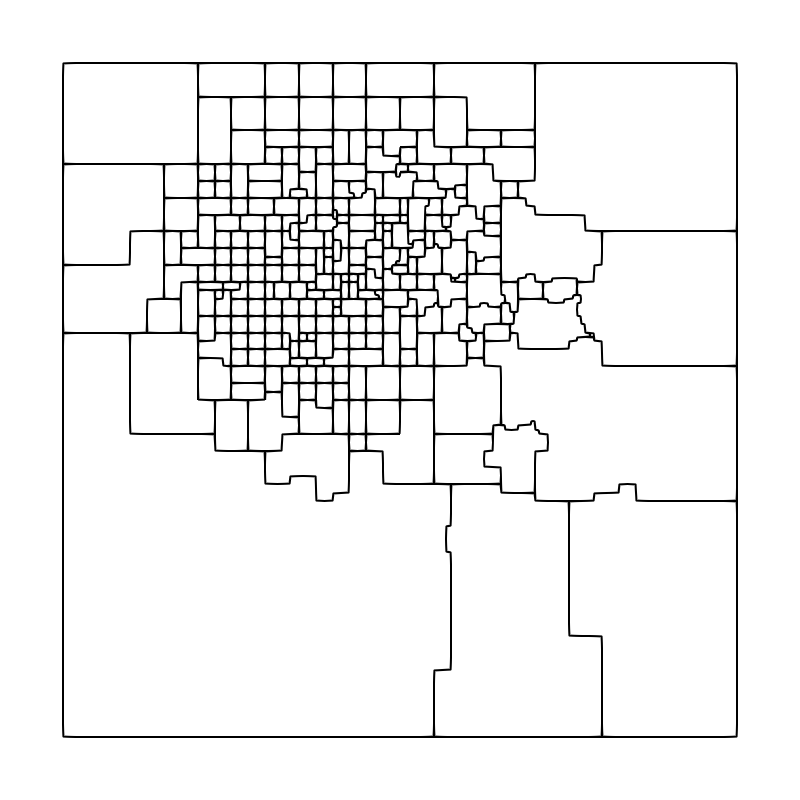}%
	}
	\caption{Examples of meshes adaptively generated to fit the circulating solution~\eqref{eq:vem:numerics:stirring} at various time moments. The concentration of mass is clearly marked by the refined regions as it moves around the domain and diffuses away.}
	\label{fig:vem:numerics:stirring:meshes}
\end{figure}

We also plot the behaviour of the errors and the estimators for the layer and circulating examples, in Figures~\ref{fig:vem:numerics:layer} and~\ref{fig:vem:numerics:stirring}, respectively.
It is worth noting that, even in this challenging adaptive situation, the $L^2(0,t; H^1(\domain))$ estimator behaves extremely well, achieving effectivity values of approximately 10.
As in the previous numerical experiments, the effectivities of the $L^{\infty}(0,t; L^2(\domain))$ estimator are somewhat larger and, although
the cause of this is not clear, we note that they appears to reach a limiting value towards the end of the simulation.
The spike in effectivity which appears at around $t=5$ for the circulating solution seems to be a result of the adaptive algorithm making a poor decision when coarsening elements, causing information to be lost.
Interestingly, there is no effect on the $L^2(0,t; H^1(\domain))$ error or estimate, and the impact appears to be picked up by the $L^{\infty}(0,t; L^2(\domain))$ estimator before the true error, resulting in an effectivity spike.
A more complex adaptive algorithm could be able to correct such a mistake, for example by rewinding a few time steps after the sudden increase in the estimator; we do not explore this here but refer to~\citer{Cangiani:2016} for such an algorithm.

\section{Conclusions}\label{sec:conclusion}

The new computable \emph{a posteriori} error estimates we have presented open up the possibility of using such estimates with very general forms of mesh modification.
This is achieved by removing the assumption of hierarchicality placed on the discrete spaces required by previous comparable estimates, an assumption which, as discussed in Section 1, entailed significant restrictions in practice.
Removing this assumption enabled us to introduce virtual element schemes using general adaptive polygonal meshes.
The new error estimates and adaptive algorithms we have presented in this setting appear to be the first incorporating such a general approach to coarsening in the context of time-dependent problems.
The comprehensive set of numerical experiments we have presented for the virtual element method and for a moving mesh method demonstrate the practical performance of the estimators in several challenging scenarios.

This work is a step towards the development of non-standard adaptive algorithms, able to harness adaptive polygonal meshes in a more nuanced way, to generate meshes tailored to the behaviours and anisotropies of the problem at hand.

\section{Acknowledgements}
This research work was supported by the Hellenic Foundation for Research and Innovation (H.F.R.I.) under the ``First Call for H.F.R.I. Research Projects to support Faculty members and Researchers and the procurement of high-cost research equipment grant'' (Project Number: 3270).  
AC also acknowledges support from the EPSRC (grant EP/L022745/1) and MRC (grant MR/T017988/1). 
EHG also acknowledges the support of The Leverhulme Trust (grant RPG-2015-306).
OJS acknowledges support from the EPSRC (grants EP/P000835/1 and EP/R030707/1).

\begin{figure}%
  \centering{%
  \subcaptionbox{$L^{\infty}(L^2(\domain))$ error and estimate}[0.5\textwidth]{\includegraphics[width=0.5\textwidth]{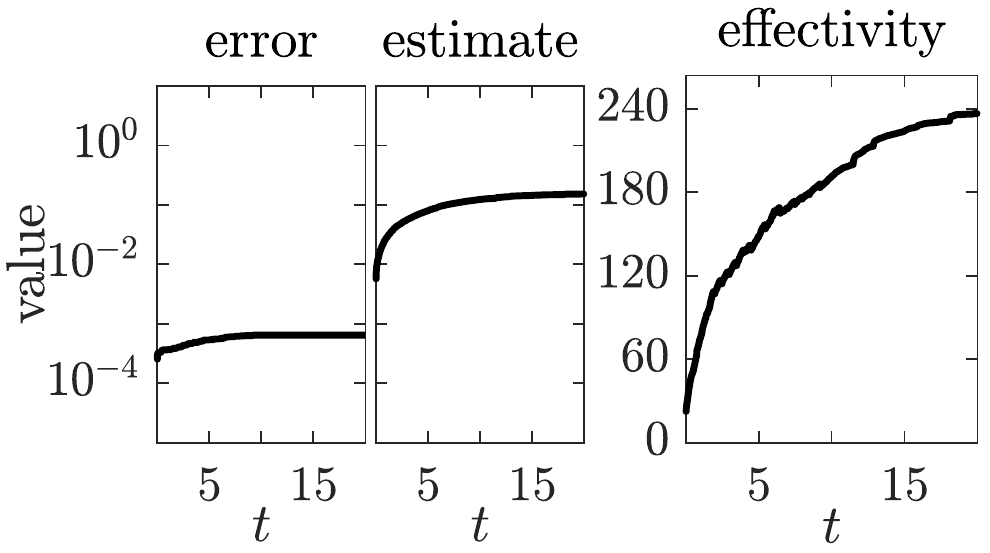}\vspace{-0.5em}}\hfill%
  \subcaptionbox{$L^2(H^1(\domain))$ error and estimate}[0.5\textwidth]{\includegraphics[width=0.5\textwidth]{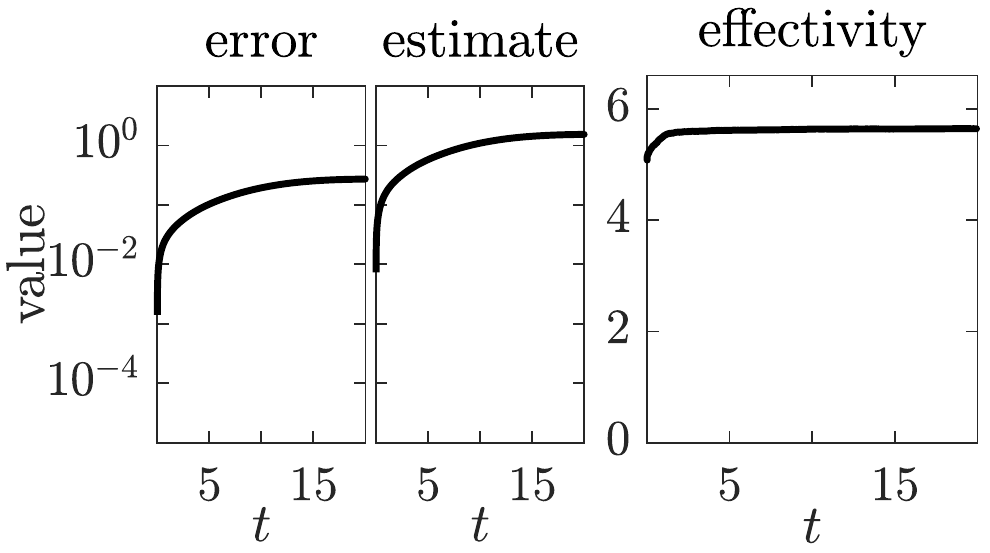}\vspace{-0.5em}}
  \vspace{0.25em}

  \subcaptionbox{{$L^{\infty}(L^2(\domain))$ estimator terms}}[0.5\textwidth]{\includegraphics[width=0.5\textwidth]{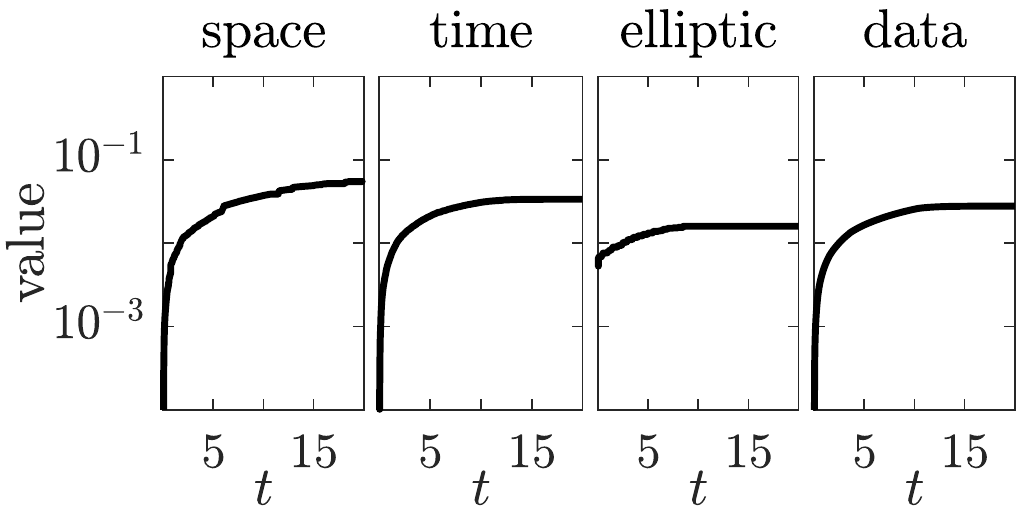}\vspace{-0.5em}}\hfill%
  \subcaptionbox{{$L^2(H^1(\domain))$ estimator terms}}[0.5\textwidth]{\includegraphics[width=0.5\textwidth]{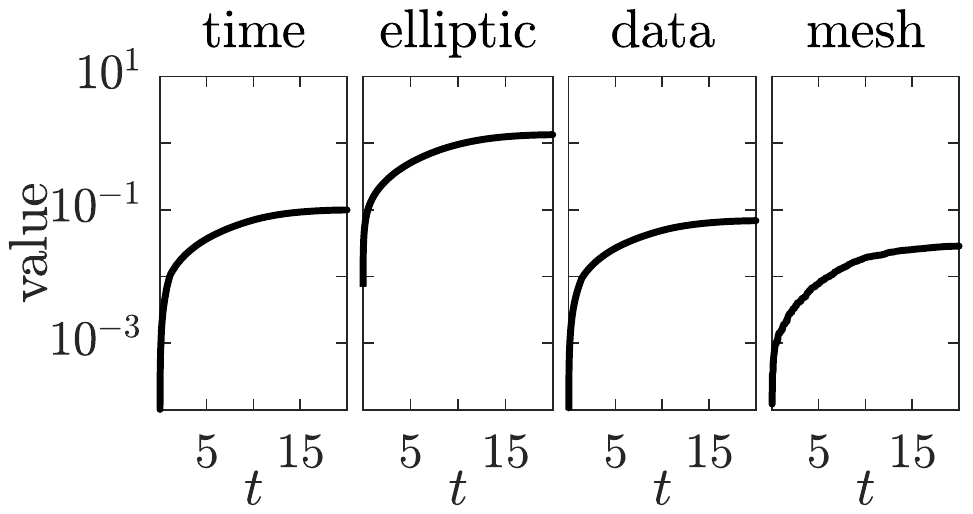}\vspace{-0.5em}}
	}
	\vspace{-0.5em}
  \caption{Error and estimator behaviour for the adaptive simulation of the layer problem~\eqref{eq:vem:numerics:layer}.}
  \label{fig:vem:numerics:layer}
\end{figure}

\begin{figure}%
  \centering{%
  \subcaptionbox{$L^{\infty}(L^2(\domain))$ error and estimate}[0.5\textwidth]{\includegraphics[width=0.5\textwidth]{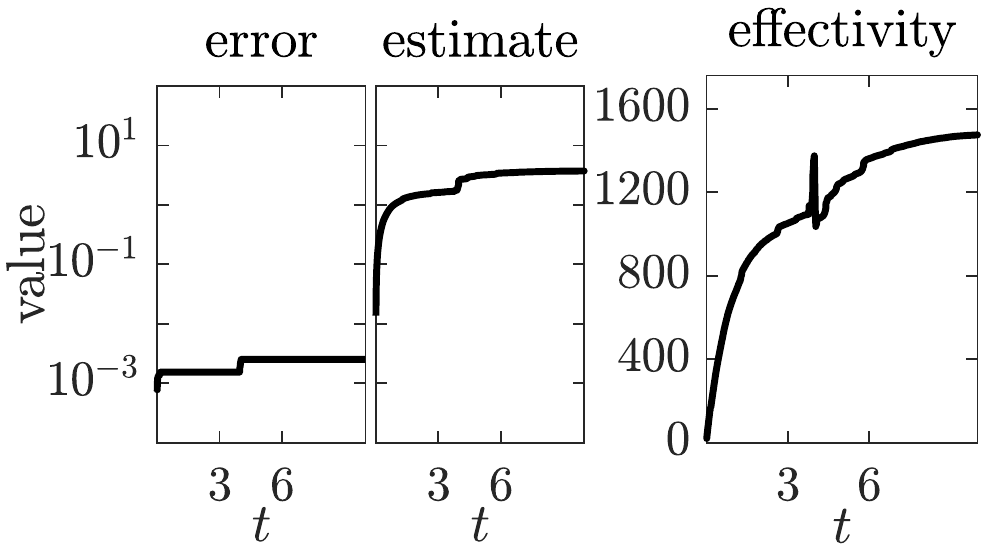}\vspace{-0.5em}}\hfill%
  \subcaptionbox{$L^2(H^1(\domain))$ error and estimate}[0.5\textwidth]{\includegraphics[width=0.5\textwidth]{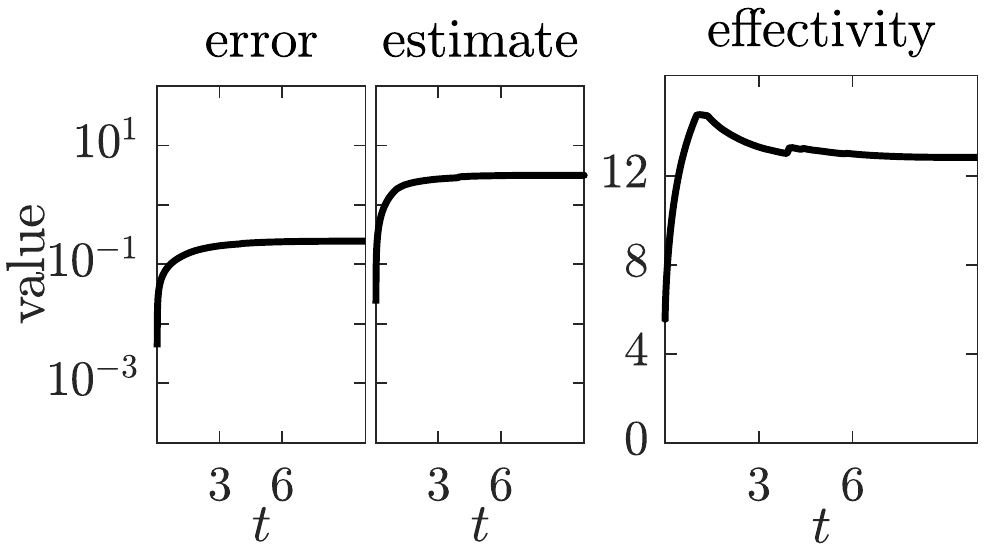}\vspace{-0.5em}}
  \vspace{0.25em}

  \subcaptionbox{{$L^{\infty}(L^2(\domain))$ estimator terms}}[0.5\textwidth]{\includegraphics[width=0.5\textwidth]{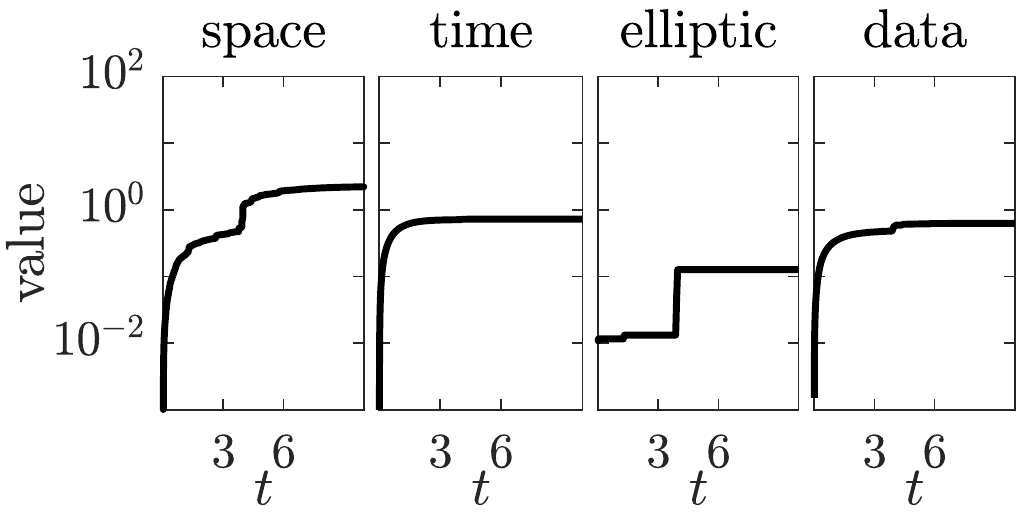}\vspace{-0.5em}}\hfill%
  \subcaptionbox{{$L^2(H^1(\domain))$ estimator terms}}[0.5\textwidth]{\includegraphics[width=0.5\textwidth]{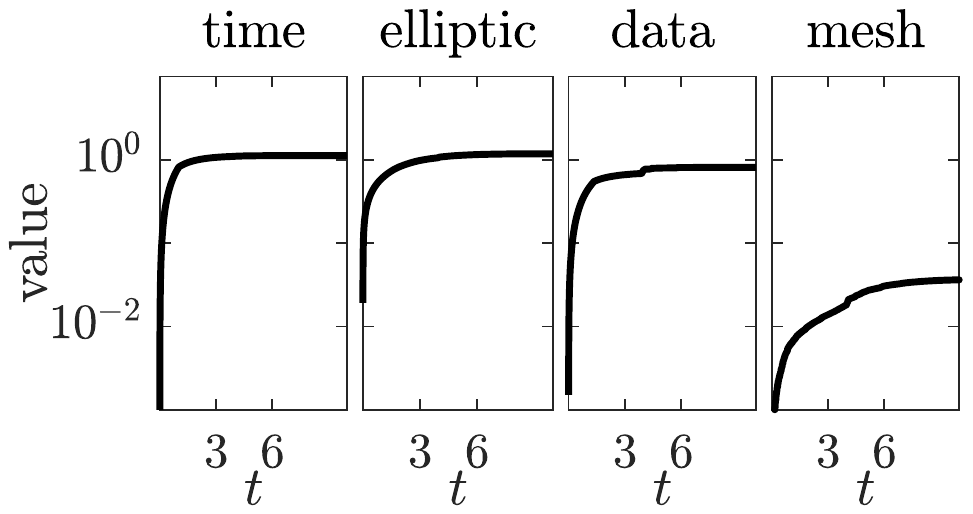}\vspace{-0.5em}}
	}
	\vspace{-0.5em}
  \caption{Error and estimator behaviour for the adaptive simulation of the circulating problem~\eqref{eq:vem:numerics:stirring}.}
  \label{fig:vem:numerics:stirring}
\end{figure}

\renewcommand{\v}[1]{\caron{#1}}

\bibliographystyle{acm}
\bibliography{references_clean}

\end{document}